\documentclass[11pt]{amsart}  

\usepackage[top=1.1in, bottom=1.1in, left=1.2in, right=1.2in]{geometry}
\usepackage{amsmath}
\usepackage{amssymb}
\usepackage{amsthm}

\usepackage{enumerate} %LaTeX package ``enumerate'' is required  

\usepackage{graphicx}

\usepackage[ 
margin=0.04\textwidth
]{caption} % for adjusting the caption of figure 1

\usepackage{color} 
\newcommand{\yo}[1]{\textcolor{black}{#1}}

\usepackage{hyperref}

\theoremstyle{plain}
\newtheorem{apthm}{Theorem}

\newtheorem{thm}{Theorem}

\newtheorem{theorem}{Theorem}[section]
\newtheorem{proposition}[theorem]{Proposition}
\newtheorem{lemma}[theorem]{Lemma}

\newtheorem*{claim}{Claim}
\newtheorem{mclaim}{Claim}

\theoremstyle{definition}
\newtheorem{definition}[theorem]{Definition}

\newtheorem*{acknowledgements}{Acknowledgements}

\theoremstyle{remark}
\newtheorem{fact}[theorem]{Fact}
\newtheorem{remark}[theorem]{Remark}

\newcommand{\rd}{\mathrm{d}}

\newcommand{\bP}{\mathbb{P}}
\newcommand{\sP}{\mathsf{P}} %changed from \mathbf{P}
\newcommand{\bN}{\mathbb{N}}
\newcommand{\bC}{\mathbb{C}}
\newcommand{\bK}{\mathbb{K}}
\newcommand{\sC}{\mathsf{C}}

\newcommand{\sL}{\mathsf{L}}
\newcommand{\sJ}{\mathsf{J}}
\newcommand{\sH}{\mathsf{H}} %changed from \mathsf{H}

\newcommand{\cO}{\mathcal{O}}
\newcommand{\cM}{\mathcal{M}}

\newcommand{\sR}{\mathsf{R}}
\newcommand{\bR}{\mathbb{R}}
\newcommand{\sT}{\mathsf{T}}
\newcommand{\sV}{\mathsf{V}}

\newcommand{\sZ}{\mathsf{Z}}
\newcommand{\sfO}{\mathsf{\Omega}}
\newcommand{\diam}{\operatorname{diam}}
\newcommand{\supp}{\operatorname{supp}}

\newcommand{\Res}{\operatorname{Res}}
\newcommand{\Rat}{\operatorname{Rat}}
\newcommand{\MinResLoc}{\operatorname{MinResLoc}}

\newcommand{\FP}{\operatorname{FP}}
\newcommand{\Fix}{\operatorname{Fix}}
\newcommand{\Span}{\operatorname{Span}} 
\newcommand{\Crucial}{\yo{\operatorname{hypRes}}}
\newcommand{\ordRes}{\operatorname{ordRes}}
\newcommand{\BC}{\operatorname{BC}}
\newcommand{\crit}{\operatorname{crit}}
\newcommand{\PGL}{\operatorname{PGL}}
\newcommand{\SL}{\operatorname{SL}}
\newcommand{\val}{\operatorname{val}}

\newcommand{\widevec}[1]{\overrightarrow{#1}}

\numberwithin{equation}{section}

\begin{document}

\title[\yo{Resultant measure}s and minimal resultant loci]{\yo{Resultant measure}s and minimal resultant loci for non-archimedean polynomial \yo{dynamics}}

\author{Hongming Nie}
\address{Institute for Mathematical Sciences, Stony Brook University, 100 Nicolls Road, Stony Brook, NY 11794, USA}
\email{hongming.nie@stonybrook.edu}

\author{Y\^usuke Okuyama}
\address{Division of Mathematics, Kyoto Institute of Technology, Sakyo-ku, Kyoto 606-8585 JAPAN}
\email{okuyama@kit.ac.jp}

\date{\today}


\begin{abstract}
 We compute the \yo{resultant measure}s for iterations $P^j$, $j\ge 1$,
 of a polynomial $P$ of degree $>1$
 on the $n$-th level Trucco's trees $\Gamma_n$, $n\ge 0$,
 in the Berkovich projective line
 over a non-archimedean field and also determine their barycenters.
 As applications, we study the asymptotic of those barycenters
 as $n\to\infty$, and establish a uniform stationarity of Rumely's minimal resultant loci of $P^j$ or equivalently that of the potential semistable reduction loci of $P^j$ as $j\to\infty$. 
 We also establish several equidistribution results for
 the \yo{resultant measure}s themselves as $n\to\infty$.
\end{abstract}

\maketitle

\section{Introduction}
\label{sec:intro}

Let $K$ be an algebraically closed field 
that is complete with respect to a non-trivial 
and non-archimedean absolute value $|\cdot|$. 
The ring $\{|z|\le 1\}$
of $K$-integers and its unique maximal ideal 
$\{|z|<1\}$ are denoted by $\cO_K$ and $\cM_K$, respectively.
The classical projective line $\bP^1=\bP^1(K)$ is compactly augmented
by the Berkovich projective line $\sP^1=\sP^1(K)$, which has
a canonical (non-strict)
ordering $\prec$ (in which $\infty\in\bP^1$ is the maximum element)
and an induced profinite tree structure.
The Berkovich upper half space $\sH^1:=\sP^1\setminus\bP^1$ is also dense in $\sP^1$,
and the metric space $(\sH^1,\rho)$, where $\rho$
is the Berkovich hyperbolic metric, is called
the Berkovich hyperbolic space.
The action on $\bP^1$ of a rational function $\phi\in K(z)$ of degree $>1$
canonically extends to an endomorphism of $\sP^1$ preserving both $\bP^1$
and $\sH^1$, and the induced action on $(\sH^1,\rho)$ is
non-(strictly)contracting.
Recall also that $\sH^1$ is divided into the totalities $\sH^1_{\mathrm{II}}$,
$\sH^1_{\mathrm{III}}$, and $\sH^1_{\mathrm{IV}}$ of points of 
types II, III, and IV, respectively, and 
$\sH^1_{\mathrm{II}}$ is already dense in $\sP^1$.

In \cite{Rumely13, Rumely17}, 
Rumely introduced a \yo{\em resultant function} $\ordRes_\phi$ for $\phi$ on 
$\sP^1$ and established foundational properties of its minimum locus, 
which is called the {\em minimal resultant locus} and denoted by $\MinResLoc_\phi$; $\ordRes_\phi$
is induced by the $\PGL(2,\cO_K)$-conjugacy invariant function
\begin{gather*}
\PGL(2,K)\ni M\mapsto\log\bigl|\bigl(\text{the resultant of }P(z)
\text{ and }Q(z)\bigr)\bigr|\in[0,+\infty)\\
\text{writing }(M\circ\phi\circ M^{-1})(z)=\frac{P(z)}{Q(z)},\quad
\text{where }P(z)=\sum_{i=0}^{\deg\phi}a_i^{(M)}z^i,\quad
Q(z)=\sum_{j=0}^{\deg\phi}b_j^{(M)}z^j,\\
\text{and }\max\bigl\{|a_0^{(M)}|,\ldots,|a_{\deg\phi}^{(M)}|,|b_0^{(M)}|,\ldots,|b_{\deg\phi}^{(M)}|\bigr\}=1.
\end{gather*}
Analytically,
Rumely characterized $\MinResLoc_\phi$ 
as the barycenter of the crucial measure 
$\nu_\phi$ for $\phi$(, which 
is defined using the Laplacian
of $\ordRes_\phi$ on the Fix-Rep tree for $\phi$ in $\sP^1$). 
Moduli-theoretically, Rumely identified the type II point locus
in $\MinResLoc_\phi$
with the {\em potential semistable reduction locus}
in the left orbit space $\PGL(2,\cO_K)\backslash\PGL(2,K)$
(see also \cite{STW14} for this locus 
in any dimension). In \cite{Okuyama20}, the second author established 
a hyperbolic geometric formula of $\ordRes_\phi$ and
generalized Rumely's crucial measure for $\phi$ 
to the {\em \yo{resultant measure}s} for $\phi$
associated to any non-trivial finite subtrees in $\sP^1$,
introducing the {\em \yo{hyperbolic resultant} function} $\Crucial_\phi$ for $\phi$ on $(\sH^1,\rho)$ (see Section \ref{sec:curvaturetreevertex})\footnote{\yo{In this paper, we adopt those more informative terminologies, the {\em resultant measure} and the {\em hyperbolic resultant function}. Since they are related to Rumely's crucial measure (see Remark \ref{th:crucialtree} below), the terminologies, {\em crucial curvature} and the {\em crucial function}, respectively would also be adequate (cf.\ \cite{Okuyama20}).}}.
In \cite{Kiwi18}, Kiwi and the first author further developed
the study of GIT semistability of rational functions.

Our aim is to contribute to the study of
\yo{resultant measure}s for iterations of polynomials.
For a polynomial $P\in K[z]$ of degree $>1$,
reformulating the original definition in \cite{Trucco14},
we introduce {\em the $n$-th level Trucco's trees} in $\sP^1$
as the iterated preimages
\begin{gather*}
\Gamma_n:=P^{-n}([\xi_P,\infty])
\end{gather*}
of the \yo{closed} interval $[\xi_P,\infty]$
\yo{between the point $\xi_P$ and the point $\infty$}
in $\sP^1$ under $P^n$,
$n\ge 0$, where $\xi_P\in\sH^1_{\mathrm{II}}$ is the 
{\em base point} for $P$
(Rivera-Letelier \cite{Rivera00}, see Subsection \ref{sec:base}),
\yo{so that $\Gamma_n$} increasingly tends 
to the {\em \yo{$\infty$-th level Trucco's} tree} $\Gamma_\infty$,
which is spanned by the Berkovich Julia set of $P$ and \yo{the point} $\infty$ in $\sP^1$, as $n\to\infty$ (see Section \ref{sec:tree} for the details).
For every $j\ge 1$ and every $n\ge 0$, 
using the Laplacian
of $\ordRes_{P^j}$ on the tree $\Gamma_n$, 
the {\em $\Gamma_n$-\yo{resultant measure}} $\nu_{P^j,\Gamma_n}$ for $P^j$
associated to $\Gamma_n$ is defined as
a (signed) Radon measure on $\Gamma_n$ having the total mass $1$,
and the barycenter $\BC(\nu_{P^j,\Gamma_n})$ on $\Gamma_n$ of
the Radon measure $\nu_{P^j,\Gamma_n}$ is defined as usual although
$\nu_{P^j,\Gamma_n}$ is not necessarily positive
(see Subsections \ref{sec:crucialfunc} and \ref{sec:slopeweight}).

Our first main result is a complete computation of
the measure $\nu_{P^j,\Gamma_n}$
in terms of the valency $v_{\Gamma_n}(\cdot)$ 
on the tree $\Gamma_n$ (see Subsection \ref{sec:subtree})
and the local degree $\deg_{\,\cdot}P^j$ on $\Gamma_n$.
See Theorem \ref{thm:support} below
for a more detailed version of Theorem \ref{introthm:support}.

\begin{thm}\label{introthm:support} 
Let $P\in K[z]$ be a polynomial of degree $d>1$. Then 
 for every $j\ge 1$, every $n\ge 0$, and every $\xi\in\Gamma_n$,
 we compute
 \begin{gather*}
 \nu_{P^j,\Gamma_n}(\{\xi\})=\frac{1}{d^j-1}\cdot
 \begin{cases}
 0 & \text{if }\xi=\infty,\\
 v_{\Gamma_n}(\xi)-2 & \text{if }v_{\Gamma_n}(\xi)\ge 2,\\
 \deg_{\xi}(P^j)-1 & \text{if }v_{\Gamma_n}(\xi)=1\text{ and }P^j(\xi)\in[\xi,\infty),\\
 -1 & \text{if }v_{\Gamma_n}(\xi)=1\text{ and }P^j(\xi)\not\in[\xi,\infty].
 \end{cases}
\end{gather*}
In particular, $\nu_{P^j,\Gamma_n}$ is supported on 
(finitely many) branch or end points of $\Gamma_n$.
\end{thm}

For every $j\ge 0$ and every $n\ge 0$, 
the {\em $\Gamma_n$-\yo{resultant} tree} $\sT_{P^j,\Gamma_n}$ 
associated to $P^j$ is the subtree in $\Gamma_n$ 
spanned by the support of $\nu_{P^j,\Gamma_n}$,
and this tree $\sT_{P^j,\Gamma_n}$ is equipped with
the {\em $\Gamma_n$-\yo{resultant vertex} set} $\sV_{P^j,\Gamma_n}$,
which is the union of $\supp\nu_{P^j,\Gamma_n}$
and the totality of branch points of $\sT_{P^j,\Gamma_n}$,
as its vertex set.
Our second main result determines the shape and the location of
the barycenter $\BC(\nu_{P^j,\Gamma_n})$
in terms of $\sV_{P^j,\Gamma_n}$. 
See Theorem \ref{thm:bary} below
for a more detailed version of Theorem \ref{introthm:bary}.

\begin{thm}\label{introthm:bary}
Let $P\in K[z]$ be a polynomial of degree $>1$, and pick $j\ge 1$.
Then for every $n\ge 0$, 
there are (not necessarily distinct) points 
$\xi_n=\xi_n^{(j)},\eta_n=\eta_n^{(j)}\in\sV_{P^j,\Gamma_n}$ such that $\xi_{n+1}\prec\xi_n\prec\eta_n\prec\xi_P$
and that
$\BC(\nu_{P^j,\Gamma_n})$ is the interval $[\xi_n,\eta_n]$.
\end{thm}

From the above definition of the \yo{resultant function} $\ordRes_\phi$ for $\phi\in K(z)$, 
the dependence of $\MinResLoc_{\phi^j}$ on $j\ge 1$ 
is not obvious, even when $\phi\in K[z]$
(see also the paragraph after Theorem \ref{thm:minres}).
Our third main result is the approximation of $\MinResLoc_{P^j}$ by $\BC(P^j,\Gamma_n)$ as $n\to\infty$, and shows that 
$(\BC(\nu_{P^j,\Gamma_n}))_{n\ge 0}$ is indeed {\em stationary} if $j\ge d-1$, in a uniform way.

\begin{thm}\label{thm:minres}
Let $P\in K[z]$ be a polynomial of degree $d>1$.
\begin{enumerate}[{\em (i)}]
\item \label{item:bcmrl}
For every $j\ge 1$, we have the Hausdorff convergence
\begin{gather*}
 \lim_{n\to\infty}\BC(\nu_{P^j,\Gamma_n})=\MinResLoc_{P^j}
\text{ in }(\sH^1,\rho),
\end{gather*}
and $\MinResLoc_{P^j}=[\xi_\infty,\eta_\infty]\subset\Gamma_\infty$ for some $\xi_\infty=\xi_\infty^{(j)},\eta_\infty=\eta_\infty^{(j)}\in\Gamma_\infty$ such that $\xi_\infty\prec\eta_\infty\prec\xi_P$.
\item \label{item:mrlstat}
For every $j\ge d-1$, the \yo{identity}
\begin{gather*}
 \BC(\nu_{P^j,\Gamma_n})\yo{\equiv}\MinResLoc_{P^{d-1}}
\quad\text{for any }n\ge \yo{j-1}
\end{gather*}
of the sequence $(\BC(\nu_{P^j,\Gamma_n}))_{n\ge 0}$ holds, 
and each hand side is a singleton in $\Gamma_{d-2}\cap\sH^1_{\mathrm{II}}$.
\end{enumerate}
\end{thm}

From Theorem \ref{thm:minres} \yo{(or indeed as the final step in the proof of it)}, we \yo{establish} the identity
\begin{gather}
 \MinResLoc_{P^j}\equiv\MinResLoc_{P^{d-1}}\quad\text{for any }j\ge d-1,\label{eq:stationaryminres}
\end{gather}
which are singletons in $\Gamma_{d-2}\cap\sH^1_{\mathrm{II}}$.
By Rumely \cite{Rumely13, Rumely17}, 
for every $j\ge 1$, $\MinResLoc_{P^j}$ is contained
in the subtree in $\sP^1$
spanned by the classical fixed points of $P^j$.
However, it seems already not easy to control those subtrees 
under iteration of $P$.
Our present work overcomes this difficulty by studying the \yo{resultant measure}s on Trucco's trees.
The lower bound $d-1$ for $j$ in \eqref{eq:stationaryminres}
is best possible in that there is $P\in K[z]$ for which
$\MinResLoc_{P^{d-2}}\neq\MinResLoc_{P^{d-1}}$
(see Example \ref{sec:bestpossible}). 
One cannot expect to extend \eqref{eq:stationaryminres} to 
every rational function $\phi$; indeed, 
for every flexible Latt\'es map $\phi\in K(z)$ of even degree, 
$\MinResLoc_{\phi^j}$, $j\ge 1$, 
are pairwise distinct singletons
(\cite[Section 7]{JacobsWilliams16}).

The orbit space $\PGL(2,\cO_K)\backslash\PGL(2,K)$
is naturally isomorphic to $\sH^1_{\mathrm{II}}$,
and for every $j\ge 1$,
as we already mentioned,
the type II point locus in $\MinResLoc_{P^j}$
is identified with the potential semistable reduction locus for $P^j$ in 
$\PGL(2,\cO_K)\backslash\PGL(2,K)$.
We restate \eqref{eq:stationaryminres}
as a theorem for a future reference, 
in a manner independent of the Berkovich space theory.

\begin{thm}\label{thm:stationary}
Let $P\in K[z]$ be a polynomial of degree $d>1$. 
Then for every $j\ge d-1$, the potential
semistable reduction locus for $P^j$ 
in $\PGL(2,\cO_K)\backslash\PGL(2,K)$ coincides with
that for $P^{d-1}$, which is a singleton.
\end{thm}

Now, for each $j\ge 1$, we focus on the asymptotic
of the $\Gamma_n$-\yo{resultant measure}s 
$\nu_{P^j,\Gamma_n}$ themselves  as $n\to\infty$.
Noting that $\nu_{P^j,\Gamma_n}$, $j\ge 1$ and $n\ge 0$, 
is a (signed) Radon measure on $\sP^1$, 
we study the positive part $\nu^+_{P^j,\Gamma_n}$, 
the negative part $\nu^-_{P^j,\Gamma_n}$, 
and the total variation $|\nu_{P^j,\Gamma_n}|=\nu^+_{P^j,\Gamma_n}+\nu^-_{P^j,\Gamma_n}$ 
of $\nu_{P^j,\Gamma_n}$; the negative part $\nu^-_{P^j,\Gamma_n}$ does not vanish when $P$ is non-simple (see the next paragraph for (non-)simpleness) and $n>j$. Let 
\begin{gather*}
\overline{\nu^+_{P^j,\Gamma_n}},\quad   
\overline{\nu^-_{P^j,\Gamma_n}},\quad\text{and}\quad
\overline{|\nu_{P^j,\Gamma_n}|}
\end{gather*}
be the averaged versions as probability measures on $\sP^1$ of the above measures, respectively. 

The equilibrium (or canonical) measure on $\sP^1$
for $P$ 
is denoted by $\mu_P$
(\cite[\S10]{Baker10}, \cite[\S2]{ChambertLoir06}, \cite[\S3.1]{Favre10}).
Recall that $P$ is {\em simple} if $\mu_P=\delta_{\xi_P}$ on $\sP^1$,
and that $P$ is {\em tame} if 
the Berkovich ramification locus $\{\deg_{\,\cdot}(P)>1\}$ of $P$ is a finite
subtree in $\sP^1$ (see Subsection \ref{sec:simpleness} for more details).
Our final main result is the equidistribution of those 
averaged measures towards $\mu_P$.
A {\em new-type} equidistribution result on the iterated preimages of the base point $\xi_P$ 
which does not take into account their local degrees for $P^j$ (Theorem \ref{th:convergence} in Appendix)
plays a key role in the proof of Theorem \ref{thm:equi}. 

 \begin{thm}\label{thm:equi}
Let $P\in K[z]$ be a non-simple and tame polynomial 
of degree $>1$. Then for every $j\ge 1$, 
 all $\overline{\nu^+_{P^j,\Gamma_n}}$, $\overline{\nu^-_{P^j,\Gamma_n}}$ and $\overline{|\nu_{P^j,\Gamma_n}|}$
tend to $\mu_P$ weakly on $\sP^1$ as $n\to\infty$.
If $P$ is simple, then for every $n\ge 0$,
$\nu_{P^j,\Gamma_n}=\overline{\nu^+_{P^j,\Gamma_n}}=
\overline{|\nu_{P^j,\Gamma_n}|}\equiv\mu_P$ on $\sP^1$.
\end{thm}
  
We mention that 
Jacobs \cite{Jacobs17} established the weak convergence 
towards $\mu_\phi$ of Rumely's {\em crucial measures} $\nu_{\phi^j}$
\yo{(see Remark \ref{th:crucialtree} below)} for the iterations $\phi^j$
of $\phi\in K(z)$ of degree $>1$ as $j\to\infty$; 
see also \cite[Theorem 6]{Okuyama20} for its generalization and quantitative sharpening.

\subsection*{Organization of this paper}
In Section \ref{sec:pre}, we gather background materials 
from potential theory and dynamics
on the Berkovich projective line,
reductions of rational functions, 
and \yo{potential semistable reduction}. In Section \ref{sec:curvaturetreevertex},
we develop general properties
of the $\Gamma$-\yo{resultant measure}s $\nu_{\phi,\Gamma}$ and 
their barycenters $\BC_\Gamma(\phi,\Gamma)$ for $\phi$ associated to
subtrees $\Gamma$ in $\sP^1$.
In Section \ref{sec:tree}, we gather fundamental properties
of Trucco's trees $\Gamma_n$ and also of the vertex sets $V(\Gamma_n)$.
In Sections \ref{sec:support} and \ref{sec:pf-bary},
we show more detailed versions of
Theorems \ref{introthm:support} and \ref{introthm:bary}, respectively.
In Section \ref{sec:pf-minres}, we show 
Theorem \ref{thm:minres}, and
study the best possibility of 
the identity \eqref{eq:stationaryminres}
by examples. In Section \ref{sec:pf-equi}, 
we show Theorem \ref{thm:equi}.
In Appendix, 
we establish Theorem \ref{th:convergence}.

\section{Background}\label{sec:pre}

The residue class field
$k:=\cO_K/\cM_K$ is called the residual field of $K$. 
For more details on the foundational
studies of the Berkovich projective line,
we refer to \cite{Baker10, Benedetto19, Berkovich90, Favre10, Jonsson15}.

\subsection{Berkovich projective line}\label{sec:berk} 
The Berkovich projective line 
$\sP^1$ over $K$ is the totality of (generalized)
multiplicative seminorms on $K(z)$ which restrict to $|\cdot|$ on $K$
(see \cite[Definition 3.11]{Jonsson15}). 
The topology of $\sP^1$ is the Gel'fand (or weak) topology.
With this topology, 
$\sP^1$ is
a compact and Hausdorff space and, moreover, contains 
the (classical) projective line 
$\bP^1$ over $K$ as a dense subset identifying
$a\in\bP^1$  with  the evaluation seminorm on $K(z)$
at the point $a$.

Each point in $\sP^1$ is one of types I, II, III and IV.  
For any point $\xi\in\sP^1\setminus\{\infty\}$ of type I, II or III, 
there exists a unique $K$-closed disk 
\begin{gather*}
 B(a,r):=\{z\in K:|z-a|\le r\},\quad a\in K\text{ and }r\ge 0,
\end{gather*}
such that $\xi$ is the supremum seminorm 
$[h]_\xi=\sup_{z\in B(a,r)}|h(z)|$ on $K[z]$(, which extends to $K(z)$ 
possibly taking $+\infty$); then $\xi$ is
said to be represented by the $K$-closed disk $B_\xi:=B(a,r)$,
and we set
\begin{gather*}
 \diam(\xi):=\diam_{|\cdot|}(B_\xi),\quad\text{which indeed equals }r. 
\end{gather*}
A point $\xi\in\sP^1\setminus\{\infty\}$ 
of type I, II, or III is of type I if and only if $\diam\xi=0$, 
is of type II if and only if $\diam\xi\in|K^\times|$, 
and is of type III, otherwise. 
A type IV point $\xi\in\sP^1$ corresponds to 
(a cofinal equivalence class of) a decreasing sequence 
of $K$-closed disks the intersection of which is empty, 
and then the diameter $\diam\xi$
is defined by the decreasing limit $(>0)$ of the diameters of the corresponding decreasing disks.
The point $\infty\in\bP^1=\bP^1(K)=K\cup\{\infty\}$ corresponds
to the evaluation seminorm 
$[h]_\infty=|h(\infty)|$ on $K(z)$ (possibly taking $+\infty$), 
which is a type I point in $\sP^1$, so that $\bP^1$ is canonically 
embedded in $\sP^1$ as the set of all type I points. 
Let us denote by $\sH^1_{\mathrm{II}},\sH^1_{\mathrm{III}}$,
and $\sH^1_{\mathrm{IV}}$ the totality of type II, III, and IV
points, respectively. The point 
\begin{gather*}
 \xi_g\in\sH^1_{\mathrm{II}}\quad\text{such that }B_{\xi_g}=B(0,1)
\end{gather*}
is called the Gauss (or canonical) point in $\sP^1$.
Noting that two $K$-closed disks are disjoint
or nesting, the Berkovich projective line
$\sP^1$ is equipped with a (non-strict and partial) ordering $\prec$,
which extends the inclusion relation $\subset$ among $K$-closed disks 
(i.e., $\xi\prec\xi'$ iff $B_\xi\subset B_{\xi'}$)
having $\infty\in\bP^1$ as the unique maximal element of $\sP^1$
and induces a canonical (profinite) tree structure on $\sP^1$
in the sense of Jonsson (see \cite[\S2.1]{Jonsson15}).
For $\xi,\xi'\in\sP^1$, let $[\xi,\xi']$ 
be a unique (closed) interval in $\sP^1$ joining $\xi$ and $\xi'$, and then set
$(\xi,\xi'):=[\xi,\xi']\setminus\{\xi,\xi'\},
(\xi,\xi']:=(\xi,\xi')\cup\{\xi'\}$,
and $[\xi,\xi'):=(\xi,\xi')\cup\{\xi\}$; for example, for $\xi\prec\xi'$,
$[\xi,\xi']=[\xi',\xi]=\bigl\{\xi''\in\sP^1:\xi\prec\xi''\prec\xi'\bigr\}$.

For any $\xi_0,\xi_1,\xi_2\in\sP^1$, there is a unique point
\begin{gather*}
 \xi_1\wedge_{\xi_0}\xi_2\in\sP^1 
\end{gather*}
in the intersection of all the intervals 
$[\xi_0,\xi_1], [\xi_0, \xi_2]$, and $[\xi_1,\xi_2]$.
The Berkovich upper half space 
$\sH^1:=\sP^1\setminus\bP^1
=\sH^1_{\mathrm{II}}\sqcup\sH^1_{\mathrm{III}}\sqcup\sH^1_{\mathrm{IV}}$
is equipped with the Berkovich
hyperbolic metric $\rho$, which is written as 
\begin{gather*}
\rho(\xi_1,\xi_2)=2\log\diam(\xi_1\wedge_\infty\xi_2)-\log\diam(\xi_1)-\log\diam(\xi_2),\quad\xi_1,\xi_2\in\sH^1,
\end{gather*}
and the metric space $(\sH^1,\rho)$ is called the Berkovich hyperbolic space.
The topology of $(\sH^1,\rho)$ is stronger than the topology of $\sH^1$,
and $\sH^1_{\mathrm{II}}$ is dense in $(\sH^1,\rho)$, and hence in $\sP^1$.
Taking $+\infty$ in an appropriate way, the hyperbolic metric $\rho$
extends to a generalized metric on $\sP^1$
(see \cite[\S2.1]{Jonsson15}), which is denoted by $\tilde{\rho}$. 

\subsection{Tangent spaces and directional derivatives}\label{sec:tangent}
For each $\xi\in\sP^1$, the set of all components of $\sP^1\setminus\{\xi\}$ induces an equivalence relation on $\sP^1\setminus\{\xi\}$, and 
each equivalence class is called a 
tangent direction
of $\sP^1$ at $\xi$.
The tangent space $T_\xi\sP^1$ of $\sP^1$ at $\xi$
is the set of all tangent directions at $\xi$.
The component of $\sP^1\setminus\{\xi\}$ corresponding to
a tangent direction $\vec{v}\in T_\xi\sP^1$ is denoted by
\begin{gather*}
 U(\vec{v})=U_\xi(\vec{v}), 
\end{gather*}
and is also identified with a germ $\widevec{\xi\xi'}$ 
of a non-trivial interval $[\xi,\xi']$ for some (indeed any)
$\xi'\in U(\vec{v})$. 
The topology on $\sP^1$ has the quasi open basis
$\{U_\xi(\vec{v}):\xi\in\sP^1,\vec{v}\in T_\xi\sP^1\}$.
A subset in $\sP^1$ is called a (non-trivial)
Berkovich open (resp.\ closed)
ball in $\sP^1$ if it can be written as $U_\xi(\vec{v})$ 
(resp.\ $\sP^1\setminus U_\xi(\vec{v})$)
for some $\xi\in\sH^1_{\mathrm{II}}\cup\sH^1_{\mathrm{III}}$ and $\vec{v}\in T_\xi\sP^1$; we note that
\begin{gather*}
 \{\cdot\prec\xi\}=\sP^1\setminus U(\widevec{\xi\infty}).
\end{gather*}
A classical open/closed ball in $\bP^1$ is
the intersection 
of $\bP^1$ 
with a Berkovich open/closed ball in $\sP^1$, respectively.

For a point $\xi\in\sH^1$ and a tangent direction $\vec{v}\in T_{\xi}\sP^1$, 
let us denote by $\rd_{\vec{v}}$ the (distributional) directional derivative operator at $\xi$ with respect to $\vec{v}$ on the space of continuous functions $f$ on $(\sH^1,\rho)$,
so that if $f$ is also piecewise affine, then
\begin{gather*}
 \rd_{\vec{v}}f=\lim_{\xi'\to\xi\ \text{in}\ (U(\vec{v}),\rho)}
\frac{f(\xi')-f(\xi)}{\rho(\xi',\xi)}.
\end{gather*}

\subsection{Subtrees, Laplacians, and barycenters}\label{sec:subtree} 
A subtree $\Gamma$ in $\sP^1$ is
a closed and connected non-empty subset in $\sP^1$,
and is regarded as being equipped with neither vertices nor edges
unless they are explicitly specified. 
A subtree $\Gamma$ in $\sP^1$
is said to be spanned by a subset $S$ in $\sP^1$ if
\begin{gather*}
 \Gamma=\Span(S):=\bigcup_{\xi,\xi'\in S}[\xi,\xi'],
\end{gather*}
and is said to be finite (resp.\ trivial) if $\Gamma$ is
spanned by a finite (resp.\ singleton) subset in $\Gamma$.
For example, a singleton in $\sP^1$ is a trivial subtree in $\sP^1$,
and more generally, any closed interval $[\xi,\xi']$
in $\sP^1$ is a finite subtree in $\sP^1$.
For two subtrees $\Gamma, \Gamma'$ in $\sP^1$, 
we say $\Gamma$ is a subtree in $\Gamma'$ if $\Gamma\subset\Gamma'$,
and then both the retraction map $r_{\Gamma',\Gamma}:\Gamma'\to\Gamma$ 
and the inclusion map $\iota_{\Gamma,\Gamma'}:\Gamma\to\Gamma'$ 
are continuous; $\Gamma'$ is also called an {\em end extension} of $\Gamma$ if in addition
\begin{gather*}
 r_{\Gamma',\Gamma}\bigl(\partial\Gamma'\bigr)\subset\partial\Gamma.
\end{gather*}

For a subtree $\Gamma$ in $\sP^1$ and a point $\xi\in\Gamma$, setting
\begin{gather*}
 T_\xi\Gamma:=\bigl\{\vec{v}\in T_\xi\sP^1:U(\vec{v})\cap\Gamma\neq\emptyset\bigr\}, 
\end{gather*}
the valency of $\Gamma$ at $\xi$ is defined by
\begin{gather*}
v_{\Gamma}(\xi):=\#T_\xi\Gamma\in\bN\cup\{0,+\infty\}.
\end{gather*}
Then $v_{\Gamma}(\xi)$ takes the value $0$ 
only when $\Gamma$ is trivial. 
We call $\xi\in\Gamma$ a branch (resp.\ end) point of $\Gamma$ 
if $v_{\Gamma}(\xi)>2$ (resp.\ $v_{\Gamma}(\xi)\in\{0,1\}$), and set
\begin{gather*}
 \partial\Gamma:=\{\text{end points of }\Gamma\},
\end{gather*}
which may not be the topological boundary of $\Gamma$ in $\sP^1$;then
$\Gamma=\bigcup_{\xi'\in\partial\Gamma\setminus\{\xi\}}[\xi,\xi']$
for any $\xi\in\Gamma$. For example, 
as a finite subtree in $\Gamma$, an interval $[\xi,\xi']$ has no branch points,
and we have $\partial[\xi,\xi']=\{\xi,\xi'\}$.
For a finite subtree $\Gamma$ in $\sP^1$, 
the (averaged) $\Gamma$-\emph{valency measure}
 \begin{align*}
  \nu^{\val}_\Gamma:=&(-2)^{-1}\cdot\sum_{\xi\in\Gamma}\bigl(v_\Gamma(\xi)-2\bigr)\cdot(r_{\sP^1,\Gamma})_*\delta_\xi
 \end{align*} 
 is a (signed) Radon measure on $\Gamma$, which has
the total mass $1$ by the Euler genus theorem.

A continuous function $f$ on $\sP^1$ is said to be of $C^1$ if 
there is a finite subtree $\Gamma$ in $\sP^1$ 
such that $f$ is piecewise affine on $(\Gamma\cap\sH^1,\rho)$
and locally affine on $(\sP^1,\tilde{\rho})$
and that $f=(r_{\sP^1,\Gamma})^*(f|\Gamma)$ on $\sP^1$.
Any continuous function on $\sP^1$ is uniformly approximated 
by $C^1$-functions on $\sP^1$ (see \cite[Proposition 5.4]{Baker10}). 
For a non-trivial finite subtree $\Gamma$ in $\sP^1$,
let us also denote by $\Delta_\Gamma$ 
the (distributional) Laplacian on $\Gamma$ so that for 
every $C^1$ function $f$ on $\sP^1$ and 
every $\xi\in\Gamma\cap\sH^1$, 
\begin{gather*}
(\Delta_\Gamma f)(\{\xi\})=\sum_{\vec{v}\in T_{\xi}\Gamma}\rd_{\vec{v}}f.
\end{gather*}
For example, we have 
$\nu^{\val}_\Gamma=
(-2)^{-1}\Delta_{\Gamma}(\rho(\cdot,\xi_0))+(r_{\sP^1,\Gamma})_*\delta_{\xi_0}$ on $\Gamma$ for any $\xi_0\in\Gamma$.
We dispense with 
the Laplacian on a (possibly non-finite) subtree in $\sP^1$ in this paper. 

For a subtree $\Gamma$ in $\sP^1$, 
the {\em barycenter} $\BC_{\Gamma}(\nu)$
for a (not necessarily non-negative) Radon measure $\nu$ on $\Gamma$ 
having the total mass $1$ is defined by
\begin{gather*}
 \BC_{\Gamma}(\nu):=\Bigl\{\xi\in\Gamma: \nu(U(\vec{v}))\le\frac{1}{2}\quad\text{for every }\vec{v}\in T_{\xi}\sP^1\Bigr\};
\end{gather*}
here and below, the pushforward $(\iota_{\Gamma,\sP^1})_*\nu$ of $\nu$
is denoted by the same notation $\nu$, so that, e.g., for every $\xi\in\Gamma$ and 
every direction $\vec{v}\in T_\xi\Gamma$, we write 
\begin{gather*}
 \nu(U(\vec{v}))=\nu(U(\vec{v})\cap\Gamma)
\quad\text{and}\quad
\nu(\sP^1\setminus U(\vec{v}))
=\nu\bigl((\sP^1\setminus U(\vec{v}))\cap\Gamma\bigr),
\end{gather*}
and also $\nu(\{\cdot\prec\xi\})=\nu(\{\cdot\prec\xi\}\cap\Gamma)$
for every $\xi\in\Gamma$, for simplicity.

\subsection{The canonical action on $\sP^1$ of a rational function}\label{sec:locdeg}
The action on $\bP^1$ of a  rational function $\phi\in K(z)$
extends continuously to $\sP^1$ 
by $[h]_{\phi(\xi)}=[h\circ\phi]_\xi$, $\xi\in\sP^1$ and $h\in K(z)$.
Suppose in addition $\deg\phi>0$.
Then the above canonical action on $\sP^1$
of $\phi$ is still open, surjective, and
fiber-discrete and preserves types of points in $\sP^1$.
The local degree function $\deg_{\,\cdot}(\phi)$ on $\bP^1$
also extends to a function 
$\sP^1\to\{1,\ldots,\deg\phi\}$ canonically and upper semicontinuously,
and the extended $\phi$ is still $(\deg\phi)$-to-one 
taking into account the (extended) local degree 
$\deg_{\,\cdot}(\phi)$ of $\phi$.
The Berkovich ramification locus of $\phi$ is defined by 
\begin{gather*}
  \sR_\phi:=\bigl\{\xi\in\sP^1: \deg_{\xi}(\phi)>1\bigr\}, 
\end{gather*}
which is
closed in $\sP^1$ and has no isolated points.
The (classical) critical set of $\phi$ is defined and characterized as
\begin{gather*}
  \crit_\phi:=\bigl\{z\in\bP^1:\phi'(z)=0\text{ (in an affine
coordinate in }\bP^1)\bigr\}=\sR_\phi\cap\bP^1,
\end{gather*} 
which is possibly uncountable when $\operatorname{char}K>0$. 

\begin{fact}\label{th:transitive}
 The projective general linear group $\PGL(2,K)$,
 which is identified with the totality of the linear fractional transformations
 $M$ of $\bP^1$, acts automorphically on $\sP^1$, isometrically on $(\sH^1,\rho)$,
 and transitively on $\sH^1_{\mathrm{II}}$. The subgroup $\PGL(2,\cO_K)$ of $\PGL(2,K)$ is the stabilizer of $\{\xi_g\}$ and acts on $\PGL(2,K)$ by {\em left} multiplication.  The left orbit space 
 \begin{gather*}
 \PGL(2,\cO_K)\backslash\PGL(2,K) 
 \end{gather*}
 is embedded onto $\sH^1_{\mathrm{II}}$ in $(\sH^1,\rho)$ by identifying the orbit $[M]=\PGL(2,\cO_K)M$ of each $M\in\PGL(2,K)$ with $M^{-1}(\xi_g)\in\sH^1_{\mathrm{II}}$.\footnote{Some authors adopt the identification of the right $\PGL(2,\cO_K)$-orbit $M\PGL(2,\cO_K)$ of $M$ with $M(\xi_g)$ by considering the the conjugation of $\phi$ under $M^{-1}$ rather than under $M$.}
\end{fact}

\begin{fact}
We canonically identify 
the tangent space $T_{\xi_g}\sP^1$ with 
the projective line $\bP^1(k)$
by identifying 
each direction $\widevec{\xi_gz}\in T_{\xi_g}\sP^1$, 
$z\in\cO_K$, with the point $\hat{z}:=z+\cM_K\in k$, and $\widevec{\xi_g\infty}$ with $\infty\in\bP^1(k)$.
For every $\xi\in\sH^1_{\mathrm{II}}$, fixing some $M\in\PGL(2,K)$ sending $\xi$ to $\xi_g$, we also identify 
$T_\xi\sP^1$ with $\bP^1(k)$.
\end{fact}

For every $\xi\in\sP^1$, $\phi$ induces the tangent map
\begin{gather*}
\phi_*=(\phi_*)_\xi: T_\xi\sP^1\to T_{\phi(\xi)}\sP^1 
\end{gather*}
of $\phi$ at $\xi$ so that for every $\widevec{\xi\xi'}\in T_\xi\sP^1$,
if $\xi'$ is close enough to $\xi$, then ($\phi(\xi')\neq\phi(\xi)$ and)
\begin{gather*}
 \phi_*(\widevec{\xi\xi'})=\widevec{\phi(\xi)\phi(\xi')}.
\end{gather*}
When $\xi\in\sH^1_{\mathrm{II}}$, 
$(\phi_*)_\xi$ is regarded as a rational function on $\bP^1(k)$ 
of degree $\deg_\xi\phi$ under an identification of both $T_\xi\sP^1$ and $T_{\phi(\xi)}\sP^1$ 
with (different copies of) $\bP^1(k)$, and then for every $\vec{v}\in T_{\xi}\sP^1$, 
the {\em directional local degree} 
$m_\phi(\vec{v})\in\{1,\ldots,\deg_\xi(\phi)\}$ of $\phi$ on $U(\vec{v})$
is nothing but the local degree of this rational function
$(\phi_*)_\xi$ on $\bP^1(k)$
at the point in $\bP^1(k)$ corresponding to the direction $\vec{v}$.
When $\xi\in\sP^1\setminus\sH^1_{\mathrm{II}}$,
we set $m_\phi(\vec{v}):=\deg_\xi(\phi)$ for every $\vec{v}\in T_{\xi}\sP^1$.
For every $\xi\in\sP^1$
and every $\vec{w}\in T_{\phi(\xi)}\sP^1$, 
we have
\begin{gather}
 \sum_{\vec{v}\in T_\xi\sP^1\text{ such that }\phi_*\vec{v}=\vec{w}}m_\phi(\vec{v})=\deg_\xi(\phi).
\label{eq:totaldirect}
\end{gather}
For every $\xi\in\sP^1$ and every $\vec{v}\in T_{\xi}\sP^1$, 
if $\xi'\in U(\vec{v})$ is close enough to $\xi$, then
$\phi:[\xi,\xi']\to[\phi(\xi),\phi(\xi')]$ is a homeomorphism such that
\begin{gather}
 \rho\bigl(\phi(\xi), \phi(\xi')\bigr)=m_\phi(\widevec{\xi\xi'})\cdot\rho(\xi,\xi')
\quad\text{for any }\xi_1,\xi_2\in(\xi,\xi'),\label{eq:multiplicity}
\end{gather}
which also holds for any $\xi_1,\xi_2\in[\xi,\xi')$ when $\xi\in\sH^1$
(see \cite[Proposition 3.5]{Rivera05} and \cite[Theorem 4.7]{Jonsson15}).

For every $\xi\in\sP^1$ and every $\vec{v}\in T_{\xi}\sP^1$,
there is also the {\em surplus local degree}
$s_\phi(\vec{v})\in\{0,1,\ldots,\deg\phi\}$ of $\phi$ on $U(\vec{v})$ such that
for every $\xi'\in\sP^1$
satisfying the defining equality
\begin{gather}
 (\phi^*\delta_{\xi'})(U(\vec{v}))=s_\phi(\vec{v})
+\begin{cases}
0&\text{if }\xi'\in \sP^1\setminus U(\phi_*\vec{v}),\\
m_\phi(\vec{v})
&\text{if }\xi'\in U(\phi_*\vec{v})
\end{cases}\label{eq:surplusdefining}
\end{gather}
\yo{(\cite[Proposition 3.10]{Faber13I}; see also 
\cite[Lemma 2.1]{Rivera03} and \cite[Proposition 9.41]{Baker10}).}
It follows that $s_\phi(\vec{v})>0$ for at most finitely many $\vec{v}\in T_\xi\sP^1$ and that 
\begin{gather}
 \sum_{\vec{v}\in T_\xi\sP^1}s_\phi(\vec{v})=\deg\phi-\deg_\xi(\phi).\label{eq:totalsurplus}
\end{gather}

\subsection{The mapping properties of a polynomial}
We gather some mapping properties of a polynomial $P\in K[z]$
of degree $>1$ below, which are consequences of 
\eqref{eq:totaldirect},
\eqref{eq:multiplicity}, \eqref{eq:surplusdefining} and \eqref{eq:totalsurplus}
and will be repeatedly used in this paper.

\begin{fact}\label{th:polynomial}
Let $P\in K[z]$ be a polynomial degree $d>1$.
\begin{enumerate}[(i)]
 \item \label{head:direcsurp}
\begin{itemize}
  \item For every $\xi\in\sP^1\setminus\{\infty\}$,
	the direction $\widevec{\xi\infty}$ is the unique $\vec{v}\in T_\xi\sP^1$ such that $P_*\vec{v}=\widevec{P(\xi)\infty}$, and we have
	$m_P(\widevec{\xi\infty})=\deg_\xi(P)$
	and $s_P(\widevec{\xi\infty})=d-\deg_\xi(P)$.
 \item On the other hand,
       for every $\xi\in\sP^1\setminus\{\infty\}$
       and every $\vec{v}\in T_\xi\sP^1\setminus\{\widevec{\xi\infty}\}$,
       we have $s_P(\vec{v})=0$,
       $P$ restricts to a surjective proper map 
       $U(\vec{v})\to U(P_*\vec{v})$ of degree $m_P(\vec{v})$,
       and for every $\xi\in\sP^1\setminus\{\infty\}$,
       every $\vec{v}\in T_\xi\sP^1\setminus\{\widevec{\xi\infty}\}$,
       and every $\xi'\in P^{-1}(\xi)$, the family $U(\vec{w})$,
       $\vec{w}\in T_{\xi'}\sP^1\setminus\{\widevec{\xi'\infty}\}$ satisfying
       $P_*\vec{w}=\vec{v}$, of Berkovich open balls in $\sP^1$ is
       a partition of $\{\cdot\prec\xi'\}\cap P^{-1}(U(\vec{v}))$.
 \item Moreover, for every $\xi\in\sP^1$,
       the family $\{\cdot\prec\xi'\}$, $\xi'\in P^{-1}(\xi)$,
       of Berkovich closed balls in $\sP^1$
       is a partition of $P^{-1}(\{\cdot\prec\xi\})$;
       for every $\xi\in\sP^1$ and
       every $\xi'\in P^{-1}(\xi)$, 
       we have $(P^*\delta_{\xi''})(\{\cdot\prec\xi'\})
       =\deg_{\xi'}(P)$
       for every $\xi''\in\{\cdot\prec\xi\}$, and
       $P$ restricts to a surjection $\{\cdot\prec\xi'\}\to\{\cdot\prec\xi\}$.
 \end{itemize} 
\item \label{head:locdegdirec}
      For every $\xi\in\sP^1\setminus\{\infty\}$,
\begin{enumerate}
  \item \label{item:locdegdirect}
	if $\xi'\in(\xi,\infty)$ is close enough to $\xi$, then
      $m_P\bigl(\widevec{\xi'\xi}\bigr)
       =m_P\bigl(\widevec{\xi'\infty}\bigr)
      =m_P(\widevec{\xi\infty})=\deg_\xi(P)=\deg_{\xi'}(P)$;
 \item \label{item:increasing}
       both 
       the functions
       $\xi'\mapsto\deg_{\xi'}(P)$ and 
       $\xi'\mapsto m_P(\widevec{\xi'\xi})$ are non-decreasing, 
       and $m_P(\widevec{\xi'\xi})\le\deg_{\xi'}(P)$,
       on $((\xi,\infty],\prec)$.
\end{enumerate}
  \item \label{head:monotone}
	For every $\xi\in\sP^1\setminus\{\infty\}$,
 \begin{enumerate}
 \item \label{item:ordpres} 
  $P([\xi,\infty])=[P(\xi),\infty]$, and $P$ restricts to a homeomorphism 
       $[\xi,\infty]\to[P(\xi),\infty]$
       preserving $\prec$.
       In particular, if $P(\xi)\prec\xi$, then $P^j(\xi)\prec\xi$
       for every $j\ge 1$;
  \item \label{item:nondechyp}
	$P:(\xi,\infty)\to(P(\xi),\infty)$
       does not strictly decrease the Berkovich hyperbolic metric $\rho$,
       and moreover, 
$\rho(P(\xi_1),P(\xi_2))\ge(\deg_{\xi}P)\cdot\rho(\xi_1,\xi_2)$ for any $\xi_1,\xi_2\in(\xi,\infty)$.
 \end{enumerate}
\item \label{head:treeback}
      For every $\xi\in\sP^1$ and every $\vec{v}\in T_\xi\sP^1$, 
      if $\vec{v}\neq\widevec{\xi\infty}$
      when $\xi\neq\infty$, then for every non-trivial subtree $\Gamma$ in 
      $U(\vec{v})\cup\{\xi\}$ such that $\xi\in\partial\Gamma$,
      every $\xi'\in P^{-1}(\xi)$, and every
      $\vec{w}\in T_{\xi'}\sP^1$ such that
      $P_*\vec{w}=\vec{v}$(, also noting that $\xi'\neq\infty$ 
      and $\vec{w}\neq\widevec{\xi'\infty}$,
      when $\xi\neq\infty$),
      \begin{gather*}
        \Gamma_{\vec{w}}:=\bigl(U(\vec{w})\cup\{\xi'\}\bigr)\cap P^{-1}(\Gamma)
      \end{gather*}       
      is a non-trivial subtree in $U(\vec{w})\cup\{\xi'\}$,
     $\partial(\Gamma_{\vec{w}})=
 (U(\vec{w})\cup\{\xi'\})\cap P^{-1}(\partial\Gamma)$,
      and $(U(\vec{w})\cup\{\xi'\})\cap P^{-1}(\{\text{branch points in }\Gamma\})\subset\{\text{branch points in }\Gamma_{\vec{w}}\}$,
      and then $\{\cdot\prec\xi'\}\cap P^{-1}(\Gamma)$ is 
      the union of all such $\Gamma_{\vec{w}}$.
\end{enumerate}
\end{fact}

\subsection{Equilibrium measure, and the simpleness and tameness}\label{sec:simpleness}

First, let us recall the (non-)simpleness of a rational function.
\begin{definition}
A rational function $\phi\in K(z)$ of degree $>1$ is {\em simple}
if for some $\xi_0\in\sH^1$, 
we have $\phi^{-1}(\xi_0)=\{\xi_0\}$, or equivalently,
have $\phi(\xi_0)=\xi_0$ and $\deg_{\xi_0}(\phi)=\deg\phi$.
\end{definition}
If $\phi$ is simple, then so is $\phi^j$ for every $j\ge 1$.
This simpleness of $\phi$ is characterized in terms of the
equilibrium (or canonical) measure $\mu_\phi$ for $\phi$ or
of the Berkovich Julia set 
$\sJ_\phi$ of $\phi$.
The pullback operator $\phi^*$ on the space of all Radon measures on $\sP^1$
is the transpose of the pushforward $\phi_*$ on 
the space of all continuous functions on $\sP^1$ so that
for each $\xi\in\sP^1$,
\begin{gather*}
 \phi^*\delta_{\xi}=\sum_{\xi'\in\phi^{-1}(\xi)}(\deg_{\xi'}(\phi))\delta_{\xi'}\quad\text{on }\sP^1, 
\end{gather*}
where $\delta_\xi$ is the Dirac measure at each point
$\xi\in\sP^1$ on $\sP^1$, and that
$\phi^*\nu=\int_{\sP^1}(\phi^*\delta_{\xi})\nu(\xi)$ on $\sP^1$
for a general Radon measure $\nu$ on $\sP^1$.
We note that $\mu_{\phi^j}=\mu_\phi$ on $\sP^1$ for every $j\ge 1$.
The Berkovich Julia set $\sJ_\phi$ of $\phi$
is defined by the support of the equilibrium measure $\mu_\phi$ for $\phi$; 
recall that one construction of $\mu_\phi$ is
\begin{gather}
 \mu_\phi=\lim_{n\to\infty}\frac{(\phi^n)^*\delta_\xi}{(\deg\phi)^n}
\quad\text{weakly on }\sP^1\quad\text{for every }\xi\in\sH^1.\label{eq:canonical}
\end{gather}
The Berkovich Fatou set of $\phi$ is defined
by $\sP^1\setminus\sJ_P$, each component of which is called a Berkovich Fatou component of $\phi$.
We note that
$\sJ_{\phi^j}=\sJ_\phi$ for every $j\ge 1$,
and that $\phi^{-1}(\sJ_\phi)=\sJ_\phi$. 
Any fixed point $\xi$ of $\phi$ in $\sH^1$
satisfying $\deg_\xi(\phi)>1$ is in $\sJ_\phi\cap\sH^1_{\mathrm{II}}$
(the fact that $\xi\in\sH^1_{\mathrm{II}}$ is by Rivera-Letelier \cite{Rivera03II})), 
and so is
a (classical) fixed point $z\in\bP^1$ of $\phi$
if and only if $|\phi'(z)|>1$ (in an affine coordinate in $\bP^1$).
The simpleness of $\phi$ is equivalent to 
that one of the following is the case for some $\xi\in\sH^1$:
\begin{itemize}
 \item $\sJ_\phi=\{\xi\}$,
 \item $\mu_\phi=\delta_{\xi}$ on $\sP^1$, and
 \item $(\ordRes_\phi)^{-1}(0)=\{\xi\}=\MinResLoc_\phi$ (see Section \ref{sec:curvaturetreevertex}).
\end{itemize}
\begin{fact}
A quantitative form of the construction \eqref{eq:canonical} of 
$\mu_{\phi}$ is
\begin{gather}
 \Bigl|\int_{\sP^1}f\Bigl(\frac{(\phi^n)^*\delta_\xi}{(\deg\phi)^n}-\mu_\phi\Bigr)\Bigr|
\le\frac{|\Delta f|(\sP^1)\cdot(\rho(\xi,\xi_g)+2\sup_{\sP^1}|g_\phi|)}{(\deg\phi)^n}
\label{eq:canonicalquant}
\end{gather}
for every $\xi\in\sH^1$, every $C^1$-test function $f$ 
on $\sP^1$, and every $n\ge 0$,
where $\Delta f=\Delta_\Gamma(f|\Gamma)$
for any finite subtree $\Gamma$ in $\sP^1$ such that
$f=(r_{\sP^1,\Gamma})^*(f|\Gamma)$ on $\sP^1$, 
where the uniform limit
\begin{align*}
 g_\phi
:=&\sum_{n=0}^\infty\frac{\int_{\sP^1}-\rho(\xi_g,\phi^n(\cdot)\wedge_{\xi_g}\xi')(\phi^*\delta_{\xi_g})(\xi')}{(\deg\phi)^n}\\
=&\lim_{n\to\infty}\frac{\log|P^n|}{(\deg P)^n}-\log|\operatorname{Id}|
\quad\text{when }\phi=P\in K[z]
\end{align*}
on $\sP^1$ is the dynamical Green function for $\phi$ on $\sP^1$
(\cite[\S10]{Baker10}, \cite[\S2]{ChambertLoir06}, \cite[\S3.1]{Favre10}).
\end{fact}

Next, let us recall the tameness of a rational function.
For the topology and geometry of the Berkovich ramification locus $\sR_\phi$ of $\phi$, 
we refer to \cite{Faber13I, Faber13II}.
There are at most $(\deg\phi)-1$ components of $\sR_\phi$ 
(Faber \cite[Theorem A]{Faber13I}). If there is a 
totally ramified point $\xi$ of $\phi$ in $\sP^1$ 
in that $\deg_\xi(\phi)=\deg\phi$ 
(for example, $\deg_\infty P=\deg P$ when $\phi$ is a polynomial $P$), then
$\sR_\phi$ is connected (Faber \cite[Theorem C]{Faber13I}). 

In general, each component of 
$\sR_\phi$ can be a quite wild subtree in $\sP^1$. 
This issue motivates the following.
\begin{definition}
 A rational function $\phi\in K(z)$ of degree $>1$
 is {\em tame} if each component of $\sR_\phi$ 
 is a finite subtree in $\sP^1$. 
\end{definition}
If $\phi$ is tame, then so is $\phi^j$ for every $j\ge 1$.
The following fact on the (classical) critical set $\crit_\phi$ of $\phi$
is useful; if $\phi$ is tame,
then 
\begin{gather*}
 \#\crit_\phi=2(\deg\phi)-2
\end{gather*}
taking into account the multiplicity of $\phi$ at each $c\in\crit_\phi$, 
and if there is also a totally ramified point of $\phi$ 
in $\sP^1$, then $\sR_\phi$ itself is spanned by the finite subset $\crit_\phi$
(\cite[Theorem C]{Faber13I}).

\subsection{Potential semistable reduction locus}\label{sec:res}

\yo{Recall some facts from the geometric invariant theory
(Hilbert-Mumford GIT \cite{MFK94}). Over a given field $\bK$ and a given integer $d>0$,
the linear algebraic group $\SL_2$ acts
on the projective space $\bP^{2d+1}=\bP^{2d+1}(\bK)$ as conjugation.
The set $\Rat_d$ of all degree $d$ rational functions
on $\bP^1$ is regarded as a hyperplane complement of $\bP^{2d+1}$,
and this $\SL_2$-conjugation action restricts to the (genuine) conjugation
action on $\Rat_d$. The GIT-semistable locus 
$(\bP^{2d+1})^{\operatorname{ss}}
=(\bP^{2d+1})^{\operatorname{ss}}(O(1))=(\bP^{2d+1})^{\operatorname{ss}}(\bK)$ 
with respect to the $\SL_2$-linearized line bundle $O(1)=O_{\bP^{2d+1}}(1)$ is, roughly speaking, the largest (Zariski open) subvariety $X$
in $\bP^{2d+1}$
where the quotient {\em variety} $X/\SL_2$ makes sense. More precisely,
the categorical quotient $(\bP^{2d+1})^{\operatorname{ss}}//\SL_2$
exists, and it is a projective variety. For more details, see Silverman \cite{Silverman98}.}

\yo{Let us come back to our setting.}
For a polynomial $F(X_0,X_1)\in\cO_K[X_0,X_1]$, the (coefficient) reduction 
$\widehat{F}(x_0,x_1)\in k[x_0,x_1]$ of $F$
is obtained by reducing the coefficients of
$F$ modulo $\cM_K$.
Let $\phi\in K(z)$ be a rational function of degree $>0$.
A  minimal lift $\Phi$ of this $\phi$ 
is a pair $(\Phi_0,\Phi_1)\in(\cO_K[X_0,X_1]_d)^2\setminus(\cM_K[X_0,X_1]_d)^2$ 
of homogeneous polynomials  such that
\begin{gather*}
 \phi(z)=\frac{\Phi_1(1,z)}{\Phi_0(1,z)}
\end{gather*}
(Here we follow the convention $\infty=[0:1]$ in the projective line
from the book \cite{FvdP04}), and
is unique up to multiplication in $\cO_K^\times$. 
The coefficient reduction of $\phi$ is the point 
$\hat{\phi}\in\bP^{2d+1}(k)$ represented by
the ordered $(2d+2)$-tuple 
in $k^{2d+2}\setminus\{(0,\ldots,0)\}$
of the coefficients of $\Phi_0,\Phi_1$
modulo $\cM_K$. We say this $\phi\in K(z)$ 
has a semistable reduction (modulo $\cM_K$) if 
\begin{gather*}
 \hat{\phi}\in(\bP^{2d+1})^{\operatorname{ss}}(k).
\end{gather*}
Rumely's moduli-theoretic characterization of $\MinResLoc_\phi$
({\cite[Theorem C]{Rumely17}}),
which is mentioned in Section \ref{sec:intro}, is the following.
\begin{theorem}\label{th:MinResLoc-stable}
Let $\phi\in K(z)$ be a rational function of degree $>1$. 
Then for every $M\in\mathrm{PGL}(2,K)$,
$M^{-1}(\xi_g)\in\MinResLoc_\phi$ if and only if 
$\widehat{(M\circ\phi\circ M^{-1})}$ is semistable.
\end{theorem}

We note that the totality of $M\in\PGL(2,K)$ for which
$M\circ\phi\circ M^{-1}$ has a semistable reduction is 
invariant under the left $\PGL(2,\cO_K)$-multiplication action on $\PGL(2,K)$.

\begin{definition}
 The {\em potential semistable reduction locus for $\phi$ in} 
$\PGL(2,\cO_K)\backslash\PGL(2,K)$ is defined by the locus
\begin{gather*}
 \bigl\{\PGL(2,\cO_K)M:M\circ\phi\circ M^{-1}\text{ has a semistable reduction}\bigr\}.
\end{gather*}
\end{definition}

Theorem \ref{th:MinResLoc-stable} is restated as
that {\em the type} II {\em locus $\MinResLoc_\phi\cap\sH^1_{\mathrm{II}}$ 
in $\MinResLoc_\phi$ is identified with
the potential semistable reduction locus for $\phi$ in $\PGL(2,\cO_K)\backslash\PGL(2,K)$} 
using the canonical embedding of $\PGL(2,\cO_K)\backslash\PGL(2,K)$ into
$\sH^1$ the image of which is $\sH^1_{\mathrm{II}}$
(see Fact \ref{th:transitive}). \yo{We use this identification only in
reformulating the identity \eqref{eq:stationaryminres} as Theorem \ref{thm:stationary}.}

\section{\yo{hyperbolic resultant} functions and \yo{resultant measure}s}
\label{sec:curvaturetreevertex}

In this section, let $\phi\in K(z)$ be a rational function of degree $d>1$.
We recall \yo{hyperbolic resultant} functions and \yo{resultant measure}s for $\phi$ developed in  \cite{Okuyama20} and then establish several auxiliary results on \yo{resultant measure}s. 
\subsection{The \yo{hyperbolic resultant} function}\label{sec:crucialfunc}
In  \cite{Okuyama20}, 
the second author introduced the {\em $\phi$-\yo{hyperbolic resultant} function}
\begin{gather*}
 \Crucial_\phi(\xi):=
\frac{\rho(\xi,\xi_g)}{2}+\frac{\rho(\xi,\phi(\xi)\wedge_{\xi_g}\xi)-\int_{\sP^1}\rho(\xi_g,\xi\wedge_{\xi_g}\cdot)\phi^*\delta_{\xi_g}}{d-1}
\quad\text{on }\sH^1,\label{eq:crucialdef}
\end{gather*}
which satisfies a \yo{nice} difference formula
\begin{gather}
 \Crucial_\phi(\xi)-\Crucial_\phi(\xi_0)=
\frac{\rho(\xi,\xi_0)}{2}+\frac{\rho(\xi,\phi(\xi)\wedge_{\xi_0}\xi)-\int_{\sP^1}\rho(\xi_0,\xi\wedge_{\xi_0}\cdot)(\phi^*\delta_{\xi_0})}{d-1}
\label{eq:difference}
\end{gather} 
for any $\xi,\xi_0\in\sH^1$. 
\yo{The following part by part computations are very useful.
\begin{fact}[{\cite[(4.1), (4.2)]{Okuyama20}}]\label{th:local}
For every rational function $\phi\in K(z)$ of degree $>1$, every $\xi\in\sH^1$,
and every $\widevec{\xi\xi''}\in T_{\xi}\sP^1$, 
diminishing $[\xi,\xi'']$
if necessary, we have
\begin{gather*}
[\xi,\xi'']\ni\xi'\mapsto\rho\bigl(\xi',\phi(\xi')\wedge_{\xi}\xi'\bigr)=
\begin{cases}
 0 &\text{if }\phi(\xi)\neq\xi\text{ and }\widevec{\xi\phi(\xi)}=\widevec{\xi\xi''},\\
\rho(\xi',\xi)&\text{if }\phi(\xi)\neq\xi\text{ and }\widevec{\xi\phi(\xi)}\neq\widevec{\xi\xi''},\\
0 &\text{if }\phi(\xi)=\xi\text{ and }\phi_*(\widevec{\xi\xi''})=\widevec{\xi\xi''},\\
\rho(\xi',\xi) &\text{if }\phi(\xi)=\xi\text{ and }\phi_*(\widevec{\xi\xi''})\neq\widevec{\xi\xi''}
\end{cases}
\quad\text{and}\\
[\xi,\xi'']\ni\xi'\mapsto\int_{\sP^1}\rho(\xi,\xi'\wedge_{\xi}\cdot)(\phi^*\delta_{\xi})=(\phi^*\delta_{\xi})\bigl(U(\widevec{\xi\xi''})\bigr)\cdot\rho(\xi',\xi).
\end{gather*}
\end{fact}
}
The function $\Crucial_\phi$ is a proper, convex,
and piecewise affine $\bR$-valued function on $(\sH^1,\rho)$ and moreover,
taking the value $+\infty$ everywhere on $\bP^1$, 
the (extended) $\Crucial_\phi$ 
is locally affine on $(\sP^1,\tilde{\rho})$ except for $\sH^1_{\mathrm{II}}$.
Here, as usual, 
the convexity of $\Crucial_\phi$ on $(\sH^1, \rho)$ means that
for every $\xi\in\sH^1$ and any distinct $\vec{v},\vec{w}\in T_\xi\sP^1$, 
\begin{gather}
(\rd_{\vec{v}}+\rd_{\vec{w}})\Crucial_{\phi}\ge 0 \label{eq:convex}
\end{gather}
(see \cite[Theorems 1 and 2]{Okuyama20}). 
The hyperbolic resultant function  $\Crucial_\phi$ provides an explicit,
global, and hyperbolic geometric
formula 
\begin{gather*}
 \ordRes_\phi(\xi)
=2d(d-1)\cdot\Crucial_\phi(\xi)-\log|\Res(\text{a minimal lift of }\phi)|
\quad\text{on }\sP^1\label{eq:ordcru}
\end{gather*}
of 
the \yo{resultant function} $\ordRes_\phi$, 
so in particular provides the coincidence
\begin{gather}
\MinResLoc_\phi=\text{(the minimum locus of }\Crucial_\phi).\label{eq:same}
\end{gather}

From the difference formula  \eqref{eq:difference}
\yo{and Fact \ref{th:local}}, 
we have
\begin{gather}
 (d-1)\cdot\rd_{\vec{v}}\Crucial_\phi\in\Bigl\{\frac{d+1-2m}{2}:m\in\{0,1,2,\ldots,(\deg\phi)+1\}\Bigr\}\label{eq:sloperange}
\end{gather}
for every $\xi\in\sH^1$ and every $\vec{v}\in T_\xi\sP^1$,
so $\rd_{\vec{v}}\Crucial_\phi$ never vanishes if $\deg\phi$ is even
(\cite[Theorem 2]{Okuyama20}). For every non-trivial finite subtree $\Gamma$ in $\sP^1$,
every $\xi\in\Gamma\cap\sH^1$, and every $\vec{v}\in T_{\xi}\Gamma$, 
the following slope formula 
\begin{gather}
\rd_{\vec{v}}\Crucial_\phi
=\frac{1}{2}-\bigl(\nu_{\phi,\Gamma}\bigr)(U(\vec{v}))\label{eq:slope}
\end{gather}
(\cite[Theorem 3]{Okuyama20}),
where the Radon measure $\nu_{\phi,\Gamma}$ of total mass $1$
is the $\Gamma$-\yo{resultant measure}
for $\phi$ defined in the next subsection,
is very useful \yo{in computing $\nu_{\phi,\Gamma}$ themselves and
their barycenters}. 

\subsection{The \yo{resultant measure} and its barycenter}\label{sec:slopeweight}

Pick a non-trivial finite subtree $\Gamma$ in $\sP^1$,
and recall from Subsection \ref{sec:subtree}
the Laplacian $\Delta_\Gamma$ and
the $\Gamma$-valency measure $\nu^{\val}_\Gamma$. 
The $\Gamma$-\yo{resultant measure} for $\phi$ is defined by 
\begin{gather*}
 \nu_{\phi,\Gamma}:=
\Delta_\Gamma\Crucial_\phi+\nu^{\val}_\Gamma
 \quad\text{on }\Gamma, \label{eq:curvature} 
\end{gather*}
which can be written as 
\begin{gather}
\nu_{\phi,\Gamma}
=\frac{\Delta_\Gamma\bigl(\xi\mapsto\rho(\xi, \phi(\xi)\wedge_{\xi_0}\xi)\bigr)+\bigl(r_{\sP^1,\Gamma}\bigr)_*(\phi^*\delta_{\xi_0}-\delta_{\xi_0})}{d-1}
\quad\text{on }\Gamma
\label{eq:crucialmeas}
\end{gather}
for every $\xi_0\in\sH^1$ (\cite[(1.13)]{Okuyama20}), 
is a 
%(not necessarily non-negative) 
(signed) Radon measure on $\Gamma$ 
having total mass $1$ and 
supported on a finite subset in 
\begin{gather*}
 (\Gamma\cap\sH^1_{\mathrm{II}})\cup\{\text{type I or IV points in }\Gamma\ \text{not fixed by }\phi\}\cup\{\text{type III end points of }\Gamma\},
\end{gather*}
and induces the weight function
\begin{gather*}
\sP^1\ni\xi\mapsto (d-1)\cdot\nu_{\phi,\Gamma}(\{\xi\})\in
\bN\cup\{0,-1\}
\end{gather*}
on $\sP^1$ (\cite[Theorems 2(ii) and 3(ii)]{Okuyama20}).
The $\Gamma$-{\em \yo{resultant} tree} $\sT_{\phi,\Gamma}$ for $\phi$ associated to $\Gamma$
is the subtree
\begin{gather*}
 \sT_{\phi,\Gamma}:=\Span\bigl(\supp(\nu_{\phi,\Gamma})\bigr)
\end{gather*}
in $\Gamma$, which is equipped with the {\em \yo{$\Gamma$-resultant} vertex set}
\begin{gather}\label{eq:V}
\sV_{\phi,\Gamma}:=\supp(\nu_{\phi,\Gamma})\cup
 \supp(\nu^{\val}_{\sT_{\phi,\Gamma}}).
 \end{gather}

\begin{remark}\label{th:crucialtree}
 In \cite{Rumely17}, 
 Rumely originally introduced the {\em crucial measure} $\nu_\phi$ for $\phi$ 
 which is nothing but
 the $\Gamma_\phi^{\FP}$-\yo{resultant measure} $\nu_{\phi,\Gamma_\phi^{\FP}}$
 associated to the \yo{(Fix-Preim)} tree
\begin{gather*}
  \Gamma_\phi^{\FP}:=\bigcap_{a\in\bP^1}
 \Span\bigl(\{\text{classical fixed points of }\phi\}\cup\phi^{-1}(a)\bigr)
\end{gather*} 
 \yo{for $\phi$} and is indeed $\ge 0$, i.e., a probability measure 
 on $\Gamma_\phi^{\FP}$.
 Rumely also introduced $\sT_\phi:=\sT_{\phi,\Gamma_\phi^{\FP}}$
and its vertex set $\sV_\phi:=\sV_{\phi,\Gamma_\phi^{\FP}}$ and
 named them as the {\em crucial tree/vertex set} for $\phi$, and established that
 $\MinResLoc_\phi$ is either a singleton in $\sV_{\phi}$ or a closed edge of 
 $\sT_{\phi}$ any end point
 in which is in $\sH^1$ (\cite[Theorem 1.1]{Rumely13}). 
\end{remark}
Let us denote by $\BC(\nu_{\phi,\Gamma})$
the barycenter of $(\iota_{\sT_{\phi,\Gamma},\Gamma})^*\nu_{\phi,\Gamma}$ on $\sT_{\phi,\Gamma}$, i.e.,
\begin{gather*}
\yo{\BC(\nu_{\phi,\Gamma}):=\BC_{\sT_{\phi,\Gamma}}((\iota_{\sT_{\phi,\Gamma},\Gamma})^*\nu_{\phi,\Gamma});}
\end{gather*}
thanks to the first two propositions in the next subsection, 
for any non-trivial finite subtree $\Gamma'$ containing 
$\sT_{\phi,\Gamma}$ (e.g.\ for $\Gamma'=\Gamma$), we have
\begin{gather}
 \BC(\nu_{\phi,\Gamma})=\BC_{\Gamma'}(\nu_{\phi,\Gamma}).\label{eq:bcwelldef}
\end{gather}

Rumely's analytic characterization of $\MinResLoc_\phi$ ({\cite[Theorem B]{Rumely17}}),
which is also mentioned in Section \ref{sec:intro}, is the following.

\begin{theorem}
For every $\phi\in K(z)$ of degree $>1$,
$\MinResLoc_\phi=\BC(\nu_{\phi,\Gamma_\phi^{\FP}})$.
\end{theorem}

\subsection{Auxiliary results}
Let $\phi\in K(z)$ be a rational function of degree $>1$. 
We conclude this section with several very useful consequences on 
the \yo{resultant measure}s for $\phi$
from the convexity \eqref{eq:convex} and the slope formula \eqref{eq:slope}. 
We begin with the following invariance property.
\begin{proposition}\label{th:measure-sum}
For any non-trivial subtrees $\Gamma, \Gamma'$ in $\sP^1$, 
where $\Gamma\subset\Gamma'$, 
and any $\xi\in\Gamma\cap\sH^1$ and any $\vec{v}\in T_\xi\Gamma$, we have
$\nu_{\phi,\Gamma}(U(\vec{v}))=\nu_{\phi,\Gamma'}(U(\vec{v}))$.
\end{proposition}

\begin{proof}
Using \eqref{eq:slope} twice, we have
\begin{gather*}
\nu_{\phi,\Gamma'}(U(\vec{v}))
=\frac{1}{2}-\rd_{\vec{v}}\Crucial_\phi
=\nu_{\phi,\Gamma}(U(\vec{v})),
\end{gather*}
which completes the proof.
\end{proof}

The statement (\ref{item:closedopen}) in the following monotonicity property
plays a crucial role in our study of  $\BC(\nu_{\phi,\Gamma})$ 
for 
$\phi=P^j$ and the $n$-th dynamical tree $\Gamma=\Gamma_n$ for
any polynomial $P\in K[z]$.

\begin{proposition}[monotonicity properties for Berkovich balls]\label{th:monotonicity}
For every non-trivial finite subtree $\Gamma$ in $\sP^1$,
we have following. 
\begin{enumerate}[{\em (i)}]
\item  \label{item:closedopen}
For any $\xi\in\Gamma\cap\sH^1$ and 
any distinct $\vec{v},\vec{w}\in T_\xi\Gamma$, 
$\nu_{\phi,\Gamma}(U(\vec{v}))
\le
\nu_{\phi,\Gamma}(\sP^1\setminus 
U(\vec{w}))$.
\item \label{item:relatively}
For any distinct $\xi,\xi'\in\Gamma\cap\sH^1$ and 
any $\widevec{v'}\in T_{\xi'}\Gamma$,
if $U(\widevec{v'})\subset U(\widevec{\xi\xi'})$, then ($\xi'\not\in\partial\Gamma$ and)
$\nu_{\phi,\Gamma}(U(\widevec{v'}))
\le
\nu_{\phi,\Gamma}(U(\widevec{\xi\xi'}))$.
\end{enumerate}
\end{proposition}

\begin{proof}
(\ref{item:closedopen}) Using \eqref{eq:slope} twice, we have
\begin{multline*} 
\nu_{\phi,\Gamma}(\sP^1\setminus 
U(\vec{w}))-\nu_{\phi,\Gamma}(U(\vec{v}))=\bigl(1-\nu_{\phi,\Gamma}(U(\vec{w}))\bigr)
-\nu_{\phi,\Gamma}(U(\vec{v}))\\
=1-\Bigl(\frac{1}{2}-\rd_{\vec{w}}\Crucial_{\phi}\Bigr)
-\Bigl(\frac{1}{2}-\rd_{\vec{v}}\Crucial_{\phi}\Bigr)
=(\rd_{\vec{w}}+\rd_{\vec{v}})\Crucial_{\phi}\ge 0,
\end{multline*}
where the final inequality is by \eqref{eq:convex}.

(\ref{item:relatively}) Without loss of generality, we assume that 
the open interval $(\xi,\xi')$ is contained in
an edge of $\sT_{\phi,\Gamma}$, so that $\nu_{\phi,\Gamma}((\xi,\xi'))=0$.  
Noting that $\widevec{v'}\neq\widevec{\xi'\xi}$ or equivalently
$U(\widevec{\xi'\xi})\cap U(\widevec{v'})=\emptyset$ (under the assumption),
we have
$(U(\widevec{\xi\xi'})\setminus U(\widevec{v'}))\cap\Gamma
=(\xi,\xi')\sqcup((\sP^1\setminus(U(\widevec{\xi'\xi})\sqcup U(\widevec{v'})))\cap\Gamma)$,
and in turn
\begin{multline*}
\nu_{\phi,\Gamma}\bigl(U(\widevec{\xi\xi'})\bigr)
-\nu_{\phi,\Gamma}(U(\widevec{v'}))
=0+\bigl(1-\nu_{\phi,\Gamma}\bigl(U(\widevec{\xi'\xi})\sqcup U(\widevec{v'})\bigr)\bigr)\\
=1-\bigl(\nu_{\phi,\Gamma}\bigl(U(\widevec{\xi'\xi})\bigr)
+\nu_{\phi,\Gamma}(U(\widevec{v'}))\bigr)
=\nu_{\phi,\Gamma}\bigl(\sP^1\setminus U(\widevec{\xi'\xi})\bigr)
-\nu_{\phi,\Gamma}(U(\widevec{v'}))\ge 0,
\end{multline*}
where the final inequality is by the first statement (\ref{item:closedopen}).
\end{proof}

The minimum locus of $\Crucial_\phi$ on 
any non-trivial finite subtree $\Gamma$ in $\sP^1$
is a subtree in $\Gamma$ and is contained in $\sH^1$
by the properties of $\Crucial_\phi$ mentioned above.

\begin{lemma}\label{th:bcchar}
Let $\Gamma$ be a non-trivial finite subtree 
in $\sP^1$. Then
\begin{gather}
 \BC(\nu_{\phi,\Gamma})=(\text{the minimum locus of }\Crucial_\phi|\Gamma).\label{eq:samefinite}
\end{gather}
Moreover, for any distinct $\alpha,\beta\in\Gamma$,
the following statements are equivalent; 
\begin{enumerate}[{\em (i)}]
\item \label{item:whole}
$[\alpha,\beta]\subset \BC(\nu_{\phi,\Gamma})$.
\item \label{item:goback}
$\nu_{\phi,\Gamma}(U(\widevec{\alpha\beta}))
=\nu_{\phi,\Gamma}(U(\widevec{\beta\alpha}))
=1/2$.
\item $\nu_{\phi,\Gamma}(U(\vec{v}))=1/2$ for every $\xi\in[\alpha,\beta]$ and every $\vec{v}\in T_\xi([\alpha,\beta])$.
      \label{item:everypoint}
\item $\Crucial_\phi\equiv\min_{\Gamma}\Crucial_\phi$
      on $[\alpha,\beta]$.
      \label{item:minimumtree}
\item $\Crucial_\phi\equiv\min_{\sH^1}\Crucial_\phi$ on $[\alpha,\beta]$.
      \label{item:minimum}
\end{enumerate}
\end{lemma}

\begin{proof}
The equality \eqref{eq:samefinite} follows from 
the convexity \eqref{eq:convex} and the slope formula \eqref{eq:slope},
and together with
the piecewise affineness of $\Crucial_\phi$
yields the equivalence
(\ref{item:whole})$\Leftrightarrow$(\ref{item:minimumtree})
for any distinct $\alpha,\beta\in\Gamma$.
On the other hand,
the implication (\ref{item:minimum})$\Rightarrow$(\ref{item:minimumtree})
is clear, (\ref{item:minimumtree})$\Rightarrow$(\ref{item:everypoint})
is by \eqref{eq:slope}, and (\ref{item:everypoint})$\Rightarrow$(\ref{item:goback}) is clear. The implication
(\ref{item:goback})$\Rightarrow$(\ref{item:minimumtree})
is by \eqref{eq:convex} and \eqref{eq:slope},
and (\ref{item:minimumtree})$\Rightarrow$(\ref{item:minimum})
is by \eqref{eq:convex}.
\end{proof}

The following is an application of Lemma \ref{th:bcchar}
and is a partial generalization of \cite[Theorem 3(iii)]{Okuyama20}, where only the case that $\nu_{\phi,\Gamma}\ge 0$ was treated.

\begin{proposition}\label{lem:baryhalf}
\begin{enumerate}[{\em (A)}]
 \item Let $\Gamma$ be a non-trivial finite subtree in $\sP^1$.
 Then the barycenter $\BC(\nu_{\phi,\Gamma})$ is 
 an interval in $\sT_{\phi,\Gamma}$ 
 having end points $\alpha_1,\alpha_2$ in $\sV_{\phi,\Gamma}$
 (possibly $\alpha_1=\alpha_2$), and we have
$\supp(\nu_{\phi,\Gamma})\cap\BC(\nu_{\phi,\Gamma})
 \subset\{\alpha_1,\alpha_2\}\cup\bigl\{\text{branch points of }
 \sT_{\phi,\Gamma}\bigr\}$.
       \label{head:barymin}
 \item Let $\Gamma,\Gamma'$ 
 be any non-trivial finite subtrees in $\sP^1$
 such that $\Gamma\subset\Gamma'$.
       \label{head:singleton}
 \begin{enumerate}[{\em (i)}]  
 \item \label{item:segment}
 If $\BC(\nu_{\phi,\Gamma})$ is non-trivial, then
 $\BC(\nu_{\phi,\Gamma})\subset\BC(\nu_{\phi,\Gamma'})$, 
 and if in addition $\Gamma'$ is an end extension of $\Gamma$ and
 $\BC(\nu_{\phi,\Gamma})\neq\BC(\nu_{\phi,\Gamma'})$,
 then any component of 
 $\BC(\nu_{\phi,\Gamma'})\setminus\BC(\nu_{\phi,\Gamma})$ is written as 
 $[\xi',\xi)\subset\Gamma'\setminus\Gamma$
 for some $\xi'\in\partial\BC(\nu_{\phi,\Gamma'})$
 and some $\xi\in\partial\Gamma$.
 \item \label{item:singleton}
 If $\BC(\nu_{\phi,\Gamma})=\{\xi\}$ and 
 $\BC(\nu_{\phi,\Gamma'})=\{\xi'\}$, $\xi\neq\xi'$, then
 we have $[\xi',\xi)\subset\Gamma'\setminus\Gamma$, and if in addition $\Gamma'$ is an end extension of $\Gamma$, then we also have $\xi\in\partial\Gamma$.
 \end{enumerate} 
\end{enumerate}
\end{proposition}

\begin{proof}
(\ref{head:barymin}) 
By \eqref{eq:samefinite},
$\BC(\nu_{\phi,\Gamma})$ is a subtree in $\sT_{\phi,\Gamma}(=\Span(\supp(\nu_{\phi,\Gamma})))$.
Moreover, $\BC(\nu_{\phi,\Gamma})$ is in $\MinResLoc_\phi$, 
which is an interval (see Remark \ref{th:crucialtree}),
by \eqref{eq:same} and the equivalence (\ref{item:whole})$\Leftrightarrow$(\ref{item:minimum}) in Lemma \ref{th:bcchar}. Hence $\BC(\nu_{\phi,\Gamma})$ is also an interval.

Let us see that $\partial\BC(\nu_{\phi,\Gamma})\subset\sV_{\phi,\Gamma}(=\supp(\nu_{\phi,\Gamma})\cup
\supp(\nu^{\val}_{\sT_{\phi,\Gamma}}))$; indeed, for any 
$\xi\in\partial\BC(\nu_{\phi,\Gamma})\setminus\{\text{branch points of }
\sT_{\phi,\Gamma}\}$,
$T_\xi(\sT_{\phi,\Gamma})$ consists of
distinct $\widevec{v_1},\widevec{v_2}$. Then 
noting both $\min_{i\in\{1,2\}}\rd_{\widevec{v_i}}\Crucial_\phi\ge 0$ 
(by the equality \eqref{eq:samefinite}) 
and $\max_{i\in\{1,2\}}\rd_{\widevec{v_i}}\Crucial_\phi>0$ 
(also 
by the piecewise affineness of $\Crucial_\phi$ on $(\sH^1,\rho)$)
and using the slope formula \eqref{eq:slope}, we have
$\nu_{\phi,\Gamma}(\{\xi\})
=\nu_{\phi,\Gamma}(\sT_{\phi,\Gamma})-\nu_{\phi,\Gamma}(U(\widevec{v_1}))-\nu_{\phi,\Gamma}(U(\widevec{v_2}))>1-(1/2)-(1/2)=0$, 
so that $\xi\in\supp(\nu_{\phi,\Gamma})$. 

Similarly, if we have
$\BC(\nu_{\phi,\Gamma})=[\alpha_1,\alpha_2]$,
then
for every $\xi\in(\alpha_1,\alpha_2)\setminus\{\text{branch points of }
\sT_{\phi,\Gamma}\}$, we have
$\nu_{\phi,\Gamma}(\{\xi\})=\nu_{\phi,\Gamma}(\sT_{\phi,\Gamma})-(1/2)-(1/2)=0$ 
using the equivalence \eqref{item:whole}$\Leftrightarrow$(\ref{item:everypoint}) in Lemma \ref{th:bcchar},
so that $\xi\in\BC(\nu_{\phi,\Gamma})\setminus\supp(\nu_{\phi,\Gamma})$.

(\ref{head:singleton}) If $\BC(\nu_{\phi,\Gamma})$ is non-trivial, 
then $\BC(\nu_{\phi,\Gamma})\subset\BC(\nu_{\phi,\Gamma'})$
by the equivalence (\ref{item:whole})$\Leftrightarrow$(\ref{item:minimum}) 
in Lemma \ref{th:bcchar}. If in addition 
$\Gamma'$ is an end extension of $\Gamma$
and $\BC(\nu_{\phi,\Gamma})\subset\Gamma\setminus\partial\Gamma$, then we even have
$\BC(\nu_{\phi,\Gamma})=\BC(\nu_{\phi,\Gamma'})$,
which concludes the remaining assertion in (\ref{item:segment}).

Next, suppose that $\BC(\nu_{\phi,\Gamma})$ and 
$\BC(\nu_{\phi,\Gamma'})$ are distinct 
singletons $\{\xi\},\{\xi'\}$, respectively.
If $[\xi',\xi)\cap\Gamma\neq\emptyset$, then
by the equality \eqref{eq:samefinite} and the convexity \eqref{eq:convex},
we must have $\min_{\Gamma'}\Crucial_\phi>\min_{\Gamma}\Crucial_\phi$,
which contradicts $\Gamma\subset\Gamma'$. 
If in addition $\Gamma'$ is an end extension of $\Gamma$ and
$\xi\not\in\partial\Gamma$, then $[\xi',\xi)\cap\Gamma\neq\emptyset$, which is 
impossible as seen above. Hence (\ref{item:singleton}) also holds.
\end{proof}

A sharper version (Proposition \ref{coro:measure1/2strong} below) of the following will be established below when $\phi=P^j$ and $\Gamma$ is the $n$-th dynamical tree $\Gamma_n$ for
any polynomial $P\in K[z]$.

\begin{proposition}[a general mass upper bound $1/2$ for Berkovich closed balls]\label{coro:measure1/2}
For every non-trivial finite subtree $\Gamma$ in $\sP^1$,
every $\xi\in\Gamma$, and
every $\vec{v}\in T_{\xi}\Gamma$ but at most one 
tangent direction, 
we have $\nu_{\phi,\Gamma}(U(\vec{v}))\le 1/2$.
\end{proposition}

\begin{proof}
Suppose to the contrary that 
$\min_{\ell\in\{1,2\}}(\nu_{\phi,\Gamma})(U(\widevec{v_\ell}))>1/2$
for some distinct $\widevec{v_1},\widevec{v_2}\in T_{\xi}\Gamma$. 
Then using \eqref{eq:slope} twice, we must have
\begin{gather*}
(\rd_{\widevec{v_1}}+\rd_{\widevec{v_2}})\Crucial_\phi
=\sum_{\ell\in\{1,2\}}\Bigl(\frac{1}{2}-\nu_{\phi,\Gamma}(U(\widevec{v_{\ell}}))\Bigr)
=1-\sum_{\ell\in\{1,2\}}\nu_{\phi,\Gamma}(U(\widevec{v_{\ell}}))
<1-\Bigl(\frac{1}{2}+\frac{1}{2}\Bigr)=0, 
\end{gather*}
which contradicts \eqref{eq:convex}.
\end{proof}

\section{Trucco's trees for a polynomial}\label{sec:tree}
Let $P\in K[z]$ be a polynomial of degree $d>1$.
The Berkovich (immediate) superattracting basin for $P$ 
associated to the superattracting fixed point $\infty$ is
\begin{gather*}
 \sfO_P(\infty):=\bigl\{\xi\in\sP^1:\lim_{n\to\infty}P^n(\xi)=\infty\bigr\},
\end{gather*}
which is a proper subdomain in $\sP^1$ 
containing $\infty$
and completely invariant 
in that $P^{-1}(\sfO_P(\infty))=\sfO_P(\infty)$. 
The Berkovich Julia set $\sJ_P$ coincides with 
the {\em topological} boundary $\partial\sfO_P(\infty)$
of $\sfO_P(\infty)$ in $\sP^1$.
By Fact \ref{th:polynomial}(\ref{item:ordpres}),
we have
\begin{gather}
P(\xi)\in U(\widevec{\xi\infty})\setminus\{\infty\}\quad
\text{for every }\xi\in\sfO_P(\infty)\setminus\{\infty\}.
\label{eq:observation} 
\end{gather}

\subsection{The base point for a polynomial}\label{sec:base}
By Rivera-Letelier \cite[Proposition 6.7]{Rivera00}, 
there is a unique minimal $K$-closed disk $B_P$
containing the classical filled-in Julia set 
$\bP^1\setminus\sfO_P(\infty)$ of $P$, which has
$\diam_{|\cdot|}B_P\in|K^\times|$.
The point $\xi_P\in\sH^1_{\mathrm{II}}$ represented by $B_P$ 
(see Subsection \ref{sec:berk})
is called the {\em base point for $P$}, and is the unique minimal 
point $\xi$ in $(\sP^1,\prec)$ such that $\sJ_P\subset\{\cdot\prec\xi\}$ or equivalently that
$\Span(\sJ_P)\subset\{\cdot\prec\xi\}$.

\begin{lemma}\label{th:basetoinfty}
We have
\begin{enumerate}[{\em (i)}]
 \item \label{item:diverge} 
       $P^n(\xi)\prec P^{n+1}(\xi)$
       for every $\xi\in[\xi_P,\infty]$
       and every $n\ge 0$,
 \item \label{item:totram}
       $\deg_\xi(P)\equiv d$ 
       for every $\xi\in[\xi_P,\infty]$,
 \item \label{item:totramif}
       $P^{-1}([P(\xi),\infty])=[\xi,\infty]$ 
       for every $\xi\in[\xi_P,\infty]$, and
 \item \label{item:directional}
       $m_P(\widevec{\xi_P\infty})\equiv d$, and 
       $m_P(\widevec{\xi\xi_P})=m_P(\widevec{\xi\infty})\equiv d$ for every $\xi\in(\xi_P,\infty)$.
\end{enumerate}
\end{lemma}

\begin{proof}
For every $\xi\in[\xi_P,\infty]$,
$(\emptyset\neq)\sJ_P\subset\{\cdot\prec\xi_P\}\subset\{\cdot\prec\xi\}$,
and then $\sJ_P=P(\sJ_P)\subset\{\cdot\prec P(\xi)\}$ 
using Fact \ref{th:polynomial}(\ref{head:direcsurp}),
so either $\xi\prec P(\xi)$ or $P(\xi)\prec\xi$.
When $P(\xi)\prec\xi$, we have 
$P^\ell(\xi)\prec\xi$ for any $\ell\ge 0$
by Fact \ref{th:polynomial}(\ref{item:ordpres}), so
$\xi=\xi_P\in\sJ_P$ and in turn $P(\xi)=\xi$.
Hence the item (\ref{item:diverge}) is the case for $n=0$, 
and then for any $n\ge 0$ by Fact \ref{th:polynomial}(\ref{item:ordpres}) again.
For every $\xi\in[\xi_P,\infty]$, from the above argument,
we also compute 
$\deg_\xi(P)=(P^*\delta_{\xi'})(\{\cdot\prec\xi\})
=(P^*\delta_{\xi'})(\sJ_P)=d$
for any $\xi'\in\sJ_P=P^{-1}(\sJ_P)\subset\{\cdot\prec\xi\}\subset\{\cdot\prec P(\xi)\}$
by Fact \ref{th:polynomial}(\ref{head:direcsurp}), 
so the item (\ref{item:totram}) is the case.
The item (\ref{item:totramif}) is by
(\ref{item:diverge}) and (\ref{item:totram}), 
and the item (\ref{item:directional}) is
by (\ref{item:totram}) and Fact \ref{th:polynomial}(\ref{head:locdegdirec})(\ref{item:increasing}).
\end{proof}

For every integer $n\ge 0$, we set
\begin{gather*}
 \sL_n:=P^{-n}(\{\xi_P\})\subset\sH^1_{\mathrm{II}}\cap(\Omega_P(\infty)\sqcup\sJ_P)
\end{gather*}
using the polynomial $P$. 
The simpleness of $P$ is equivalent to any of 
the following statements;
\begin{itemize}
 \item $\xi_P\in\sJ_P$,
 \item $P(\xi_P)=\xi_P$, and 
 \item $P^{-1}(\xi_P)=\{\xi_P\}$,
\end{itemize}
and then $\sL_n\equiv\sL_0$
(see also Subsection \ref{sec:simpleness}).

\subsection{The $n$-th and $\infty$-th level Trucco's trees \yo{in $\sP^1$}}\label{sec:Trucco}
For every integer $n\ge 0$,
we define the {\em $n$-th level Trucco's tree} as
\begin{gather*}
 \Gamma_n:=P^{-n}([\xi_P,\infty])
\end{gather*}
using the polynomial $P$, 
which is a non-trivial finite subtree 
in $\sfO_P(\infty)\sqcup\sJ_P$ having
\begin{gather*}
 \partial\Gamma_n=\sL_n\sqcup\{\infty\}, 
\end{gather*}
and is equipped with the vertex set 
\begin{gather*}
 V(\Gamma_n):=\partial\Gamma_n\sqcup\{\text{branch points of }\Gamma_n\}. 
\end{gather*}
We also define the {\em $\infty$-th level Trucco's tree} as
a non-trivial subtree in $\sfO_P(\infty)\sqcup\sJ_P$ 
\begin{gather*}
 \Gamma_\infty:=\Span(\sJ_P\sqcup\{\infty\})
\end{gather*}  
using $P$, 
where $\xi_P\in\Span(\sJ_P)\subset\{\cdot\prec\xi_P\}$ by
the definition of $B_P$ and which has
\begin{gather*}
 \partial\Gamma_\infty=\sJ_P\sqcup\{\infty\}
\end{gather*}
and is equipped with the vertex set\footnote{Those vertex sets $V(\Gamma_n)$ and $V(\Gamma_\infty)$ differ from the ones
(``geometric sequences'' in \cite{Trucco14})
 used by Trucco.} 
\begin{gather*}
V(\Gamma_\infty):=\partial\Gamma_\infty\sqcup\{\text{branch points of }\Gamma_\infty\}
\end{gather*}
(see Lemma \ref{th:inparticular} below for more details).

If $P$ is simple or equivalently $P^{-1}(\xi_P)=\{\xi_P\}=\sJ_P$, 
then by Lemma \ref{th:basetoinfty}, for every $n\ge 0$, 
\begin{gather}  
\sL_n\equiv\{\xi_P\}=\sJ_P,\quad
\Gamma_n\equiv\Gamma_0=\Gamma_\infty=P^{-1}(\Gamma_\infty),\quad\text{and}\quad V(\Gamma_n)\equiv\{\xi_P,\infty\}\equiv\partial\Gamma_n.\label{eq:simpletree} 
\end{gather}

The following is shown by using
Fact \ref{th:polynomial}, Lemma \ref{th:basetoinfty}, and 
\eqref{eq:totaldirect},
also recalling 
$P^{-1}(\sJ_P)=\sJ_P$,
$P^{-1}(\sfO_P(\infty))=\sfO_P(\infty)$,
the connectedness of $\sfO_P(\infty)$,
and the (defining) maximality property of $\xi_P$.
For more details, we refer to \cite[Sections 3 and 4]{Trucco14}.

\begin{fact}\label{th:inparticular}
If $P$ is non-simple, i.e., $P^{-1}(\xi_P)\neq\{\xi_P\}\subset\sfO_P(\infty)$,
then the following hold.
\begin{enumerate}[(i)]
\item \label{head:basehull}
 For every $n\ge 1$, 
\begin{enumerate}[(a)]
 \item \label{item:ends}
       $\sL_n\subset\sfO_P(\infty)\cap\bigcup_{\xi'\in\sL_{n-1}}\{\cdot\precneqq\xi'\}$,
       $\#((\xi,\infty]\cap\sL_{n-1})=1$ for any $\xi\in\sL_n$, and $\{\cdot\precneqq\xi'\}\cap\sL_n\neq\emptyset$ for any $\xi'\in\sL_{n-1}$,
       so in particular
       $\sL_0=\{\xi_P\}\subset\Span(\sL_{n-1})\subset\Span(\sL_n)\subset\{\cdot\prec\xi_P\}$ for any $n\ge 1$;
 \item \label{item:endext}
       $P^{-1}(\Gamma_{n-1})=\Gamma_n=\Span(\sL_n\sqcup\{\infty\})\subset\sfO_P(\infty)$,
       $P^{-1}(\partial\Gamma_{n-1})=\partial\Gamma_n
       =\sL_n\sqcup\{\infty\}$,
       $P^{-1}(V(\Gamma_{n-1})\setminus\partial\Gamma_{n-1})\subset
       V(\Gamma_n)\setminus\partial\Gamma_n$, 
       and $\Gamma_n$ is an end extension of $\Gamma_{n-1}$
       and moreover,
       $\Gamma_{n-1}\subset\Gamma_n\setminus\sL_n$ and
       $r_{\Gamma_n,\Gamma_{n-1}}(\sL_n)=\sL_{n-1}$;      
 \item \label{item:basebranch}
       $\sL_{n-1}\subset V(\Gamma_n)\setminus\partial\Gamma_n$;
 \item  \label{item:locdegedge}
	$P$ maps every edge of $\Gamma_n$ 
	into that of $\Gamma_{n-1}$ 
	homeomorphically, preserving the ordering $\prec$, and
	without strictly decreasing the (generalized) metric 
	$\tilde{\rho}$.
\end{enumerate}
 \item
       \label{eq:infty}
\begin{enumerate}
  \item \label{item:exhaust}
	$P^{-1}(\Gamma_\infty)=\Gamma_\infty$,
       $P^{-1}(\partial\Gamma_\infty)=\partial\Gamma_\infty=\sJ_P\sqcup\{\infty\}$,
       and 
       $P^{-1}(V(\Gamma_\infty)\setminus\partial\Gamma_\infty)\subset
       V(\Gamma_\infty)\setminus\partial\Gamma_\infty$,
and for every $n\ge 0$, $\Gamma_\infty$ is an end extension of $\Gamma_n$ and moreover, 
       $r_{\Gamma_\infty,\Gamma_n}(\sJ_P)=\sL_n$,
$\bigcup_{n\ge 0}\Gamma_n=\Gamma_\infty\setminus\partial\Gamma_\infty$, and $\bigsqcup_{n\ge 0}\sL_n\subset V(\Gamma_\infty)\setminus\partial\Gamma_\infty$.
 \item \label{item:nonmaximal}
For every $\xi\in\{\cdot\prec\xi_P\}\cap\Gamma_\infty$
and every $\vec{v}\in T_\xi\Gamma_\infty\setminus\{\widevec{\xi\infty}\}$, we have $m_{\vec{v}}(P)\le d-1$, and
for every $\xi\in\{\cdot\precneqq\xi_P\}\cap\Gamma_\infty$,
we have $\deg_\xi(P)\le d-1$.
\end{enumerate}
\end{enumerate}
\end{fact}

\subsection{Vertex sets of Trucco's trees}
Suppose that $P$ is tame in this subsection.
Then $\#(\crit_P)=d-1$ (see Subsection \ref{sec:simpleness});
in this subsection, each $c\in\crit_P$ is counted taking into account
its multiplicity $\deg_c(P)-1$ for $P$.

\begin{fact}[a Riemann-Hurwitz-type formula {\cite[Section 2.6]{Trucco14}}]\label{th:RH}
For every (non-trivial) classical closed/open ball $B$ in $K$,
we have
\begin{gather*}
 \#(\crit_P\cap B)=\deg\bigl(P:B\to P(B)\bigr)-1,
\end{gather*}
where
\begin{multline*}
 \deg\bigl(P:B\to P(B)\bigr)\\
\begin{cases}
:=\deg_\xi(P) & 
 \text{when }B=K\cap\{\cdot\prec\xi\}\text{ for some }\xi\in\sH^1_{\mathrm{II}}\cup\sH^1_{\mathrm{III}},\\
= m_P(\vec{v}) & 
   \text{when }B=K\cap U(\vec{v})\text{ for some }\xi\in\sH^1_{\mathrm{II}}\cup\sH^1_{\mathrm{III}}\text{ and some }\vec{v}\in(T_\xi\sP^1)\setminus\{\widevec{\xi\infty}\}
  \end{cases} 
\end{multline*}
(see also Fact \ref{th:polynomial}(\ref{head:direcsurp})).
\end{fact}

We also set 
\begin{gather}
 \sC_P:=r_{\sP^1,\Gamma_\infty}
\bigl(\crit_P\cap(\sfO_P(\infty)\setminus\{\infty\})\bigr),\label{eq:critretract}
\end{gather}
which is 
in $(\Gamma_\infty\setminus\partial\Gamma_\infty)\cap\sH^1_{\mathrm{II}}$;
under the tameness assumption,
\begin{gather}
 \begin{cases}
 \text{if }P\text{ is simple, then }\sC_P=\emptyset\text{ (see \cite[Corollary 2.11]{Kiwi06})};\\
 \text{if }P\text{ is 
 %(not only tame but also) 
 non-simple, then }\xi_P\in\sC_P\text{ (see \cite[Proposition 4.3]{Trucco14})}. 
\end{cases}\label{eq:tamebranchempty}
\end{gather}

\begin{lemma}\label{lem:tree-vertices}
$\sC_P\subset V(\Gamma_\infty)\setminus\partial\Gamma_\infty$.
\end{lemma}

\begin{proof}
Pick $\xi=r_{\sP^1,\Gamma_\infty}(c)\in\sC_P$ for some 
$c\in\crit_P\cap(\sfO_P(\infty)\setminus\{\infty\})$, 
so
$\xi\in\Gamma_\infty\setminus\partial\Gamma_\infty$.
Then $P(\xi)\in\Gamma_\infty\setminus\partial\Gamma_\infty$ 
(see Lemma \ref{th:inparticular}(\ref{item:exhaust})),
so there is $\vec{w}\in T_{P(\xi)}\Gamma_\infty\setminus\{\widevec{P(\xi)\infty}\}$, and we have
\begin{gather*}
 \bigl\{\vec{v}\in T_\xi\sP^1:P_*\vec{v}=\vec{w}\bigr\}
\subset T_\xi\Gamma_\infty\setminus\bigl\{\widevec{\xi\infty}\bigr\} 
\end{gather*}
by Facts \ref{th:polynomial}(\ref{head:treeback})
and \ref{th:inparticular}(\ref{item:exhaust}).
Let us see that $\#\{\vec{v}\in T_\xi\sP^1:P_*\vec{v}=\vec{w}\}>1$,
which will complete the proof; for, otherwise,
there is a unique 
$\vec{v}\in T_\xi\sP^1\setminus\{\widevec{\xi\infty}\}$
such that $P_*\vec{v}=\vec{w}$, and then 
by \eqref{eq:totaldirect}, we have 
$m_P(\vec{v})=\deg_{\xi}(P)$.
Moreover, we have not only $\widevec{\xi c}\in(T_\xi\sP^1)\setminus(T_\xi\Gamma_\infty)\subset(T_\xi\sP^1)\setminus\{\vec{v},\widevec{\xi\infty}\}$ but, using Fact \ref{th:polynomial}(\ref{item:increasing}),
also $m_P(\widevec{\xi c})\ge\deg_c(P)>1$. Hence we must have
\begin{multline*}
\deg_\xi(P)-1=\#(\crit_P\cap\{\cdot\prec\xi\})
\ge\#\bigl(\crit_P\cap(U(\vec{v})\cup U(\widevec{\xi c}))\bigr)\\
=\bigl(m_P(\vec{v})-1\bigr)+\bigl(m_P(\widevec{\xi c})-1\bigr)>(\deg_\xi(P)-1)+(1-1)=\deg_\xi(P)-1
\end{multline*}
using a Riemann-Hurwitz-type formula (Fact \ref{th:RH}) in $3$ times.
This is a contradiction.
\end{proof}

The following proposition
is clear by \eqref{eq:simpletree} and Lemma \ref{th:basetoinfty}
when $P$ is simple. Under the tameness assumption, 
if $P$ is non-simple, then
the equality \eqref{eq:tamevertex} below
is a precision of 
Lemma \ref{lem:tree-vertices},
and together with \cite[Proposition 4.10]{Trucco14}
yields \eqref{eq:locdegedge} below.

\begin{proposition}\label{th:branchar}
For every $n\ge 0$,
\begin{gather}
 V(\Gamma_n)
=\Bigl(\Gamma_n\cap\bigcup_{i=0}^nP^{-i}(\sC_P)\Bigr)\cup\{\xi_P,\infty\},\label{eq:tamevertex}
\end{gather}
and for every $j\ge 1$, every $n\ge 0$,  
and every open edge $(\xi_1,\xi_2)$ of $\Gamma_n$,
$\xi_1\precneqq\xi_2$, we have
\begin{gather}
 \deg_{\xi}(P^j)\equiv\deg_{\xi_1}(P^j)
 =m_{P^j}\bigl(\widevec{\xi_1\xi_2}\bigr)\equiv m_{P^j}\bigl(\widevec{\xi\xi_1}\bigr)
\quad\text{for every }\xi\in(\xi_1,\xi_2).\label{eq:locdegedge}
\end{gather}
\end{proposition}

\begin{proof}[Proof of \eqref{eq:tamevertex} when $P$ is non-simple and tame]
For $n=0$, \eqref{eq:tamevertex}
follows from Lemma \ref{lem:tree-vertices}. 

Pick $n>0$. The inclusion ``$\supset$'' 
in \eqref{eq:tamevertex} follows from 
Lemma \ref{lem:tree-vertices} and
Fact \ref{th:inparticular}(\ref{item:endext}).
Let us see the reverse inclusion ``$\subset$''.
Pick $\xi\in V(\Gamma_n)\setminus\{\infty\}$.
Setting $\ell:=\min\{j\ge 0:P^{j+1}(\xi)
\in\Gamma_{\infty}\setminus V(\Gamma_{\infty})\}
\in\{0,1,\ldots,n\}$ (by Fact \ref{th:inparticular}(\ref{item:exhaust})), we have 
$P^\ell(\xi)\in V(\Gamma_{\infty})\setminus\partial \Gamma_{\infty}$, i.e., 
\begin{gather*}
 v_{\Gamma_{\infty}}(P^\ell(\xi))-1
=\#(T_{P^\ell(\xi)}\Gamma_{\infty}\setminus\bigl\{\widevec{P^\ell(\xi)\infty}\bigr\})>1 
\end{gather*}
as well as
$T_{P^{\ell+1}(\xi)}\Gamma_{\infty}
=\bigl\{\widevec{P^{\ell+1}(\xi)\infty},\vec{v}\bigr\}$ for some $\vec{v}\in(T_{P^{\ell+1}(\xi)}\Gamma_{\infty})\setminus\{\widevec{P^{\ell+1}(\xi)\infty}\}$, and then
\begin{gather*}
(T_{P^\ell(\xi)}\Gamma_n)\setminus\bigl\{\widevec{P^\ell(\xi)\infty}\bigr\}=
(T_{P^\ell(\xi)}\Gamma_\infty)\setminus\bigl\{\widevec{P^\ell(\xi)\infty}\bigr\}
=\bigl\{\vec{w}\in T_{P^\ell(\xi)}\sP^1:P_*\vec{w}=\vec{v}\bigr\}
\end{gather*}
(also by Facts \ref{th:inparticular}(\ref{item:exhaust})
and \ref{th:polynomial}(\ref{head:treeback})).
Consequently, we compute
\begin{align*}
&\,\#\Bigl(\crit_P\cap\textstyle\bigcup_{\vec{w}\in
(T_{P^\ell(\xi)}\Gamma_n)\setminus
\{\widevec{P^\ell(\xi)\infty}\}}
U(\vec{w})\Bigr)\\
=\,&
\sum_{\vec{w}\in T_{P^\ell(\xi)}\sP^1\text{ such that }P_*\vec{w}=\vec{v}}\bigl(m_P(\vec{w})-1\bigr)
=\bigl(\deg_{P^\ell(\xi)}(P)\bigr)-\bigl(v_{\Gamma_{\infty}}(P^\ell(\xi))-1\bigr)\\ 
<\,&\bigl(\deg_{P^\ell(\xi)}(P)\bigr)-1
=\#\bigl(\crit_P\cap\{\cdot\prec P^\ell(\xi)\}\bigr), 
\end{align*}
the first and third equalities in which are
by a Riemann-Hurwitz-type formula (Fact \ref{th:RH}),
and the second in which is by \eqref{eq:totaldirect},
respectively. Hence we have $P^\ell(\xi)\in\sC_P$
recalling $\{\cdot\prec P^\ell(\xi)\}=\bigcup_{\vec{w}\in
(T_{P^\ell(\xi)}\sP^1)\setminus
\{\widevec{P^\ell(\xi)\infty}\}}
U(\vec{w})$.
\end{proof}

\subsection{Example of Trucco's trees}\label{sec:example}
Pick a prime number $p>4$  
and let $|\cdot|=|\cdot|_p$ denote the $p$-adic absolute value on $K=\bC_p$. Set
\begin{gather*}
 P(z):=\frac{1}{4p^2}z^4-\frac{p+1}{3p^3}z^3+\frac{1}{2p^3}z^2\in\bC_p[z],
\end{gather*}
so that $P'(z)=\frac{1}{p^2}z(z-1)(z-\frac{1}{p})$ and $\crit_P=\{0,1,1/p,\infty\}$. Then
$P$ is tame by \cite[Corollary 6.6]{Faber13I},  
so in particular $\sR_P=\Span(\{0,1,1/p,\infty\})$. 

Noting that
\begin{gather*}
|P(z)|=|1/p|^2|z|^4\text{ on }|z|>|1/p|(>1),\quad
|P(1)|=|1/p|^3,\quad
\text{and}\quad|P(1/p)|=|1/p|^6,
\end{gather*}
we have
\begin{itemize}
  \item $\sJ_P\subset\{\cdot\prec\xi_{B(0,|1/p|)}\}$, and indeed $\xi_P=\xi_{B(0,|1/p|)}$
       (also by a Riemann-Hurwitz-type formula (Fact \ref{th:RH}) and Lemma \ref{th:basetoinfty}(\ref{item:totram})),
 \item $\{1,1/p\}\subset\sfO_P(\infty)$ and $\{\xi_g,\xi_P\}\subset\Gamma_\infty\setminus\partial\Gamma_\infty$ (since $0=P(0)\prec\xi_g\precneqq\xi_P\precneqq\xi_{B(0,|1/p|^3)}
=P(\xi_g)\precneqq\xi_{B(0,|1/p|^6)}=P(\xi_P)$ 
also by Fact \ref{th:polynomial}(\ref{item:ordpres})), so in particular,
 \item $r_{\sP^1,\Gamma_\infty}(1)=\xi_g$ and
       $r_{\sP^1,\Gamma_\infty}(1/p)=\xi_P$.
\end{itemize}
In Figure \ref{fig:tree}, using a Riemann-Hurwitz-type formula (Fact \ref{th:RH}) again, we write
\begin{itemize}
 \item $\sL_1=\{\xi_i:i\in\{1,2,3\}\}$, where $\xi_1,\xi_2\precneqq\xi_g(\precneqq\xi_P)$,
$\widevec{\xi_P\xi_3}\not\in\{\widevec{\xi_P\xi_g},\widevec{\xi_P\xi_\infty}\}$,
$\deg_{\xi_1}(P)=2$, and $\deg_{\xi_2}(P)=\deg_{\xi_3}(P)=1$, and
 \item $\sL_2=\{\xi_{i1}^{(\ell)}:i\in\{1,2,3\},\ell\in\{1,2\}\}\cup
\bigcup_{j\in\{2,3\}}\{\xi_{ij}=\xi_{ij}^{(1)}:i\in\{1,2,3\}\}$, where
$\xi_{i1}^{(\ell)}\precneqq\xi_1$ for any $i,\ell$,
$\xi_{ij}\precneqq\xi_j$ for any $i,j$, $P(\xi_{ij}^{(\ell)})=\xi_i$ for any $i,j,\ell$,
and $\xi_{11}^{(1)}=\xi_{11}^{(2)}$,
\end{itemize}
and indeed $V(\Gamma_0)=\sL_0\sqcup\{\infty\}$,
$V(\Gamma_1)=V(\Gamma_0)\sqcup\{\xi_g\}\sqcup\sL_1$,
$V(\Gamma_2)=V(\Gamma_1)\sqcup P^{-1}(\xi_g)\sqcup\sL_2$.
\begin{figure}[ht]
\centering
\includegraphics[bb=0 0 379 310, width=0.5\textwidth]{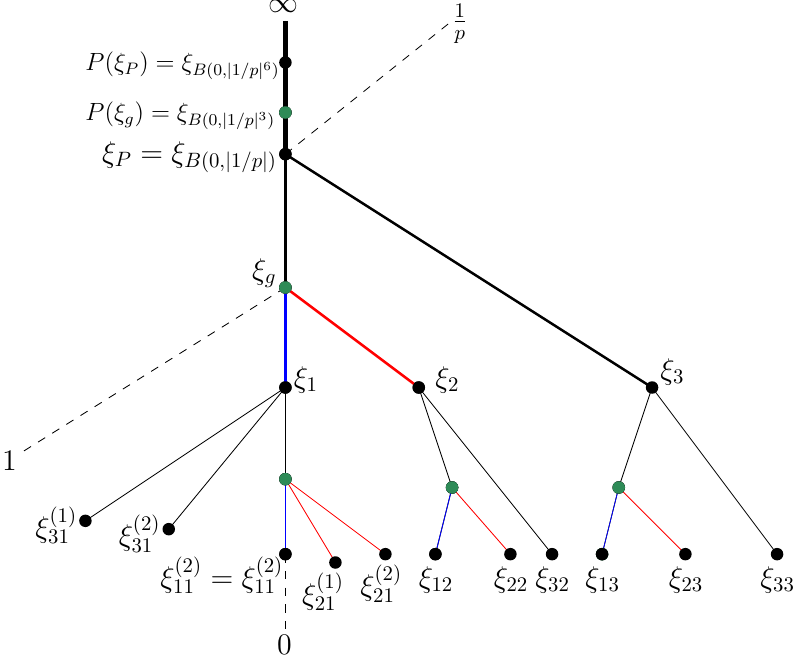}
\caption{Trucco's trees $\Gamma_0$ (fat),
$\Gamma_1$ (fat \& thick),
$\Gamma_2$ (fat, thick, \& thin)
for the $P$}
\label{fig:tree}
\end{figure}

\section{Proof of Theorem \ref{introthm:support}}\label{sec:support}
We establish the following detailed version of Theorem \ref{introthm:support}; the $\sL_n$, $\Gamma_n$, 
and $\Gamma_\infty$ are defined using $P$.
We also estimate from above the cardinality of the subset
\begin{gather}\label{eq:Z}
 \sZ_{P^j,\Gamma_n}:=\bigl\{\xi\in\partial\Gamma_n\setminus\sR_{P^j}:
 P^j(\xi)\in [\xi,\infty)\bigr\}
\end{gather}
in $\partial\Gamma_n\setminus\{\infty\}$.

\begin{theorem}\label{thm:support}
Let $P\in K[z]$ be a polynomial of degree $d>1$. 
\begin{enumerate}[{\em (A)}]
 \item For every $j\ge 1$ and every $n\ge 0$,
       \label{head:weight}
 \begin{gather*}
 \nu_{P^j,\Gamma_n}(\{\xi\})=\frac{1}{d^j-1}\cdot
 \begin{cases}
  v_{\Gamma_n}(\xi)-2 & \text{for every }\xi\in\Gamma_n\setminus\partial\Gamma_n,\\
  0 & \text{for }\xi=\infty,\\
  \deg_{\xi}(P^j)-1 & \text{for every }
  \xi\in\partial\Gamma_n\setminus\{\infty\}\text{ such that }\xi\prec P^j(\xi),\\
  -1 & \text{for every }\xi\in\partial\Gamma_n\setminus\{\infty\}\text{ such that }\xi\not\prec P^j(\xi),
 \end{cases}
 \end{gather*}
and
 \begin{enumerate}[{\em (a)}] 
 \item $\supp(\nu_{P^j,\Gamma_n}|\partial\Gamma_n)
=\partial\Gamma_n\setminus(\sZ_{P^j,\Gamma_n}\sqcup\{\infty\})$,
      \label{item:supportcurvature}
  \item $\sV_{P^j,\Gamma_n}=\supp(\nu_{P^j,\Gamma_n})=V(\Gamma_n)\setminus(\sZ_{P^j,\Gamma_n}\sqcup\{\infty\})\subset\sH^1_{\mathrm{II}}\cap\{\cdot\prec\xi_P\}$,\label{item:vertex}
 \item $\nu_{P^j,\Gamma_n}|(V(\Gamma_n)\setminus\partial\Gamma_n)\ge 0$,
 \item $\nu_{P^j,\Gamma_n}|(V(\Gamma_n)\setminus\partial\Gamma_n)=0$ if and only if 
 either $n=0$ or $P$ is simple, and
 \item  $\nu_{P^j,\Gamma_n}\ge 0$ if and only if either $j\ge n$ or $P$ is simple.
 \label{item:positive}
 \end{enumerate}
 \item If $P$ is either {\em non-simple and tame} or {\em simple}, then
 for every $j\ge 1$, we moreover have
       \label{head:zero}
 \begin{gather*} 
 \sup_{n\ge m}\#\sZ_{P^j,\Gamma_n}\le d^m\quad\text{for }m\gg 1.
 \end{gather*}
\end{enumerate}
\end{theorem} 

In the rest of this subsection, we do 
some preparatory computations.

Let $P\in K[z]$ be a polynomial of degree $d>1$,
and pick $j\ge 1$ and $n\ge 0$. 
Then by 
the first equality in Fact \ref{th:local},
Fact \ref{th:polynomial}(\ref{item:ordpres}), and
\eqref{eq:observation}, we compute 
\begin{multline}
 \rd_{\vec{v}}\bigl(\xi'\mapsto\rho(\xi', P^j(\xi')\wedge_\xi \xi')\bigr)\\
=\begin{cases}
  0 &\text{for any }\xi\in\Gamma_n\setminus\{\infty\}
  \text{ and }\vec{v}=\widevec{\xi\infty},\\
  1 & \text{for any } 
\xi\in\Gamma_n\cap(\sfO_P(\infty)\setminus\{\infty\})
  \text{ and any }\vec{v}\in T_\xi\Gamma_n\setminus\{\widevec{\xi\infty}\},
 \end{cases}\label{prelaplacevalence}
\end{multline}
so that for every $\xi\in\Gamma_n\setminus\{\infty\}$,
\begin{gather}
\bigl(\Delta_{\Gamma_n}(\xi'\mapsto\rho(\xi', P^j(\xi')\wedge_\xi \xi'))\bigr)(\{\xi\})
=\#\bigl((T_\xi\Gamma_n)\setminus\{\widevec{\xi\infty}\}\bigr)=v_{\Gamma_n}(\xi)-1\label{eq:laplacevalence}
\end{gather}
(recalling that if $\xi\in\sJ_P$,
then $\xi=\xi_P$ and $\Gamma_n=[\xi_P,\infty]$).

We use the following computation repeatedly.

\begin{lemma}\label{lem:measure-computation}
For every $\xi\in\Gamma_n$, 
\begin{gather*} 
\nu_{P^j,\Gamma_n}(\{\cdot\prec\xi\})
=\frac{1}{d^j-1}\cdot
\begin{cases} 
\deg_{\xi}(P^j)-1\ge 0 & \text{if }\xi\prec P^j(\xi),\\
-1 & \text{otherwise}.
\end{cases}
\end{gather*}
\end{lemma}

\begin{proof}
The case of $\xi=\infty$ is clear.
Pick $\xi\in V(\Gamma_n)\setminus\{\infty\}$. 
Then by the slope formula \eqref{eq:slope}
and the first case in \eqref{prelaplacevalence},
we compute
\begin{align*}
\nu_{P^j,\Gamma_n}(\{\cdot\prec\xi\})
&=1-\nu_{P^j,\Gamma_n}(U(\widevec{\xi\infty}))
=\frac{1}{2}+\rd_{\widevec{\xi\infty}}\Crucial_{P^j}\\
&=1+\frac{\rd_{\widevec{\xi\infty}}\bigl((\xi'\mapsto\rho(\xi',P^j(\xi')\wedge_\xi \xi'))
-(\xi'\mapsto\int_{\sP^1}\rho(\xi,\xi'\wedge_\xi\cdot)(P^j)^*\delta_\xi)\bigr)}{d^j-1}\\
&=1-\frac{\rd_{\widevec{\xi\infty}}\bigl(\xi'\mapsto\int_{\sP^1}\rho(\xi,\xi'\wedge_\xi\cdot)(P^j)^*\delta_\xi\bigr)}{d^j-1}.
\end{align*}
If $\xi\in\Gamma_n\cap(\Omega_P\setminus\{\infty\})$, 
then by \eqref{eq:observation} and 
Fact \ref{th:polynomial}(\ref{head:direcsurp}), we have
($P^j(\xi)\in U(\widevec{\xi\infty})\setminus\{\infty\}$ and)
\begin{gather*} 
 (P^j)_*\widevec{\xi\infty}=\widevec{P^j(\xi)\infty}
\begin{cases}
 \neq\widevec{P^j(\xi)\xi} &\text{if }\xi\prec P^j(\xi),\\
 =\widevec{P^j(\xi)\xi} &\text{if }\xi\not\prec P^j(\xi), 
\end{cases}
\end{gather*}
and then by
the second equality in Fact \ref{th:local},
\eqref{eq:surplusdefining},
and Fact \ref{th:polynomial}(\ref{head:direcsurp}), we have
\begin{multline*}
\rd_{\widevec{\xi\infty}}\Bigl(\xi'\mapsto\int_{\sP^1}\rho(\xi,\xi'\wedge_\xi\cdot)((P^j)^*\delta_\xi)\Bigr)\\
=\bigl((P^j)^*\delta_\xi\bigr)(U(\widevec{\xi\infty}))
=
\begin{cases}
 s_{P^j}(\widevec{\xi\infty})=d^j-\deg_{\xi}(P^j) &\text{if }\xi\prec P^j(\xi),\\
s_{P^j}(\widevec{\xi\infty})+m_{P^j}(\widevec{\xi\infty})=d^j &\text{if }\xi\not\prec P^j(\xi),
\end{cases}
\end{multline*}
and we are done in this case.
If $\xi\in\sJ_P$ or equivalently $\xi=\xi_P=P(\xi_P)$, 
then we have $(P^j)^{-1}(\xi)=\{\xi\}\subset[\xi,\infty)$ 
(so in particular $\deg_\xi(P^j)=d^j$) and in turn
\begin{gather*}
 \nu_{P^j,\Gamma_n}(\{\cdot\prec\xi\})=1-0=1
=\frac{\deg_\xi(P^j)-1}{d^j-1}
\end{gather*}
by the above computation of 
$\nu_{P^j,\Gamma_n}(\{\cdot\prec\xi\})$ and
the second equality in Fact \ref{th:local}.
Hence we are also done in this case.
\end{proof}

The assertion $\nu_{P^j,\Gamma_n}(\{\infty\})=0$
follows from a general result (\cite[Theorem 3(ii)]{Okuyama20}).

\begin{proof}[Proof of Theorem \ref{thm:support} in the simple case]
Suppose that $P$ is simple.
Then for every $j\ge 1$, we have not only
$\nu_{P^j,\Gamma_n}\equiv\delta_{\xi_P}$ on $\sP^1$
by recalling $\Gamma_n\equiv[\xi_P,\infty]$ and 
applying Lemma \ref{th:basetoinfty}(\ref{item:diverge})(\ref{item:totram})
and Lemma \ref{lem:measure-computation} to any
$\xi\in\Gamma_n\setminus\{\infty\}=[\xi_P,\infty)$, but also
$\sZ_{P^j,\Gamma_n}\equiv\emptyset$
by Lemma \ref{th:basetoinfty}(\ref{item:totram}).
Now all the assertions are clear.
\end{proof}

\begin{proof}[Proof of Theorem \ref{thm:support}(\ref{head:weight}) in the non-simple case]
Suppose that $P$ is non-simple.
Pick $j\ge 1$ and $n\ge 0$.
The computation of $\nu_{P^j,\Gamma_n}(\{\xi\})$
for every $\xi\in\partial\Gamma_n\setminus\{\infty\}$
is included in Lemma \ref{lem:measure-computation}.

For every $\xi\in\Gamma_n\setminus\partial\Gamma_n$,
we have $\Span(\{\xi\}\sqcup P^{-j}(\xi))\subset\Gamma_n$ 
(by Fact \ref{th:inparticular}(\ref{item:endext})), so that
\begin{gather*}
 \bigl((r_{\sP^1,\Gamma_n})_*((P^j)^*\delta_{\xi}-\delta_{\xi})\bigr)(\{\xi\})
 =0-\delta_{\xi}(\{\xi\})=-1. 
\end{gather*}
Then also by the equalities
\eqref{eq:crucialmeas} and \eqref{eq:laplacevalence},
we have
\begin{gather*}
  (d^j-1)\cdot\nu_{P^j,\Gamma_n}(\{\xi\}) 
 =\bigl(v_{\Gamma_n}(\xi)-1\bigr)+(-1)=v_{\Gamma_n}(\xi)-2.
\end{gather*}
Hence the computation of $\nu_{P^j,\Gamma_n}(\{\xi\})$ is also done in this case.

Among the items (\ref{item:supportcurvature})-(\ref{item:positive}),
it remains to show 
\begin{itemize}
 \item  the equality $\sV_{P^j,\Gamma_n}(:=\supp(\nu_{P^j,\Gamma_n})\cup\supp(\nu^{\val}_{\sT_{P^j,\Gamma_n}}))=\supp(\nu_{P^j,\Gamma_n})$ in (\ref{item:vertex}), and 
 \item the ``only if'' part in (\ref{item:positive}).
\end{itemize}
Let us see the inclusion
$\supp(\nu^{\val}_{\sT_{P^j,\Gamma_n}})\subset\supp(\nu_{P^j,\Gamma_n})$,
which will yield the first item above; 
for, by the computation of $\nu_{P^j,\Gamma_n}$, we already have
\begin{gather*}
 \supp(\nu_{P^j,\Gamma_n})=V(\Gamma_n)\setminus(\sZ_{P^j,\Gamma_n}\sqcup\{\infty\}),
\end{gather*} 
which with $\sZ_{P^j,\Gamma_n}\sqcup\{\infty\}\subset\partial\Gamma_n$
and $\sT_{P^j,\Gamma_n}:=\Span(\supp(\nu_{P^j,\Gamma_n}))$ 
yields the desired inclusion.

Finally, let us see the second item above;
for, (under the assumption that $P$ is non-simple,)
suppose also that $(1\le)j<n$.
Then by Fact \ref{th:inparticular}(\ref{item:ends})(\ref{item:basebranch}), there are $\xi_1\in\sL_j\subset\{\cdot\precneqq\xi_P\}$ 
and $\xi_2\in\sL_{n-j}\subset\{\cdot\precneqq\xi_P\}$
such that $\widevec{\xi_P\xi_1}\neq\widevec{\xi_P\xi_2}$,
and then there is also
 $\xi_2'\in\sL_n\cap P^{-j}(\xi_2)
\subset\{\cdot\precneqq\xi_P\}$
by Fact \ref{th:polynomial}(\ref{head:direcsurp}).
 Then we have $P^j(\xi_2')=\xi_2\in\sL_{n-j}\setminus[\xi_2',\xi_P]$ 
so that $\xi_2'\not\prec P^j(\xi_2')$
since $\widevec{\xi_P\xi_2'}=\widevec{\xi_P\xi_1}\neq\widevec{\xi_P\xi_2}$,
and in turn have $(d^j-1)\cdot\nu_{P^j,\Gamma_n}(\{\xi_2'\})=-1<0$ {by Lemma \ref{lem:measure-computation}}.
\end{proof}

\begin{proof}[Proof of Theorem \ref{thm:support}(\ref{head:zero}) 
in the non-simple and tame case]
Suppose that $P$ is non-simple and tame, 
 and pick $j\ge 1$. Then $P^j$ is also tame, so in particular
 \begin{gather}
 \sR_{P^j}=\Span(\crit_{P^j}).\label{eq:tameramif}
 \end{gather}
 For any $n\ge m\ge 0$, by Fact \ref{th:inparticular}({\ref{item:endext}}),
 we have  
 \begin{gather}
 r_{\Gamma_n,\Gamma_m}(\sZ_{P^j,\Gamma_n})
 \subset\partial\Gamma_m\setminus\{\infty\}.\label{eq:retzero}
 \end{gather}
 
\begin{claim}\label{claim:1}
 For any $n>m\gg 1$, 
  $r_{\Gamma_n,\Gamma_m}(\sZ_{P^j,\Gamma_n})\cap\sR_{P^j}=\emptyset$.
 \end{claim} 
We postpone giving a proof of this Claim until the end of this subsection.
If there are $\xi\in\partial\Gamma_m\setminus\{\infty\}$ and 
distinct $z_1,z_2\in\{\cdot\precneqq\xi\}\cap\sZ_{P^j,\Gamma_n}(\subset\{\cdot\precneqq\xi\}\cap\partial\Gamma_n)$ 
for some $m\in\bN\cup\{0\}$ and some $n>m$,
then $r_{\Gamma_n,\Gamma_m}(z_1)=r_{\Gamma_n,\Gamma_m}(z_2)=\xi$
(by Fact \ref{th:inparticular}(\ref{item:endext})),
and setting 
\begin{gather*}
 \xi':=z_1\wedge_{\xi}z_2\in\{\cdot\prec\xi\}\cap\Gamma_n,
\end{gather*} 
we have $z_i\precneqq\xi'$,
$z_i\precneqq P^j(z_i)$ (since $z_i\in\sZ_{P^j,\Gamma_n}\subset\sfO_P(\infty)\setminus\{\infty\}$),
and $P^j(z_i)\precneqq P^j(\xi')$ 
(by Fact \ref{th:polynomial}(\ref{item:ordpres})
for $P^j$ and $z_i\precneqq\xi'$),
and then $\xi'\precneqq P^j(\xi')$
also noting that $P^j(\xi')\in U(\widevec{\xi'\infty})$
(by \eqref{eq:observation} for $P^j$ and
$\xi'\in\sfO_P(\infty)\setminus\{\infty\}$),
for each $i\in\{1,2\}$. In particular,
 \begin{gather*}
 (P^j)_*\widevec{\xi' z_1}=
 \widevec{P^j(\xi')P^j(z_1)}=
 \widevec{P^j(\xi')z_1}=
 \widevec{P^j(\xi')\xi'}
 =\widevec{P^j(\xi')z_2}
 =\widevec{P^j(\xi')P^j(z_2)}
 =(P^j)_*\widevec{\xi'z_2}
 \end{gather*}
 (the first and last equalities in which are by Fact \ref{th:polynomial}(\ref{item:ordpres}) again).
 Hence by Fact \ref{th:polynomial}(\ref{item:nondechyp})
 and \eqref{eq:totaldirect}, we have 
 $\deg_{\xi}(P^j)\ge\deg_{\xi'}(P^j)\ge
\sum_{i\in\{1,2\}}m_{\widevec{\xi' z_i}}(P)>1$ 
so $\xi\in\sR_{P^j}$.
 This never occurs if $m\gg 1$ by the above Claim. 

Hence for $m\gg 1$, 
$\sup_{n\ge m}\#Z_{P^j,\Gamma_n}\le\#(\partial\Gamma_m\setminus\{\infty\})\le d^m$, and we are done.
\end{proof}

\begin{proof}[Proof of Claim in the above proof]
Recall that $P$ is non-simple and tame and that $j\ge 1$ is fixed.
 If $\crit_P\subset\sfO_P(\infty)$, then 
 by the chain rule, $P^{-1}(\sfO_P(\infty))=\sfO_P(\infty)$,
 and the connectedness of $\sfO_P(\infty)$,
 we have $\crit_{P^j}=\bigcup_{\ell=0}^{j-1}P^{-\ell}(\crit_P)\subset\sfO_P(\infty)$ and in turn
 $\sR_{P^j}\subset\sfO_P(\infty)$. Then 
 by Fact \ref{th:inparticular}(\ref{item:endext})(\ref{item:exhaust}), 
 we have $(\partial\Gamma_m\setminus\{\infty\})\cap\sR_{P^j}=\emptyset$
 for $m\gg 1$, and we are done in this case.

 Suppose now that $\crit_P\not\subset\sfO_P(\infty)$.
 Then for every $c\in\crit_{P^j}\setminus\sfO_P(\infty)$,
 there is a unique $\alpha_c\in\sJ_P\cap[c,\xi_P)$, and then
 by \eqref{eq:tameramif},
 there is also $\beta_c\in(\alpha_c,\xi_P)$ so close to $\alpha_c$ that
\begin{gather*}
  \bigl((r_{\sP^1,[\alpha_c,\beta_c]})^{-1}((\alpha_c,\beta_c))\bigr)\cap\sR_{P^j}=(\alpha_c,\beta_c)
\subset(\Gamma_\infty\setminus\partial\Gamma_\infty)\cap\sR_{P^j}.
\end{gather*} 
Hence
 for $m\gg 1$, by Fact \ref{th:inparticular}(\ref{item:exhaust}), we have
 \begin{gather*}
 \bigl(\partial\Gamma_{m}\setminus\{\infty\}\bigr)\cap\sR_{P^j}\subset\bigcup_{c\in\crit_{P^j}\setminus\sfO_P(\infty)}(\alpha_c,\beta_c),\label{eq:endramify} 
 \end{gather*}
 %\quad\text{and that}\\
 %\deg_{\xi}(P^j)\equiv\deg_{\alpha_c}(P^j)=m_{P^j}\bigl(\widevec{\alpha_c\beta_c}\bigr)
 %\equiv m_{P^j}\bigl(\widevec{\xi\alpha_c}\bigr)
 %\quad\text{for every }\xi\in (\alpha_c,\beta_c).
 and then for any $n>m\gg 1$, also by 
 Fact \ref{th:inparticular}({\ref{item:endext}}), we have
 \begin{multline*}
 \Bigl((r_{\Gamma_n,\Gamma_m})^{-1}\bigl(\partial\Gamma_m\setminus\{\infty\}\bigr)\Bigr)
 \cap\Bigl((r_{\Gamma_n,\Gamma_m})^{-1}(\Gamma_m\cap\sR_{P^j})\Bigr)\\ 
 =\bigl(r_{\Gamma_n,\Gamma_n\cap\sR_{P^j}}\bigr)^{-1}\bigl(
 (r_{\Gamma_n\cap\sR_{P^j},\Gamma_m})^{-1}\bigl((\partial\Gamma_m\setminus\{\infty\})\cap\sR_{P^j}\bigr)\bigr)\\
 \subset
 \bigl(r_{\Gamma_n,\Gamma_n\cap\sR_{P^j}}\bigr)^{-1}
 \Bigl(\Gamma_n\cap\bigcup_{c\in\crit_{P^j}\setminus\sfO_P(\infty)}(\alpha_c,\beta_c)\Bigr)
 =
 \bigcup_{c\in\crit_{P^j}\setminus\sfO_P(\infty)}
 \bigl(r_{\Gamma_n,\Gamma_n\cap[\alpha_c,\beta_c]}\bigr)^{-1}\bigl(\Gamma_n\cap(\alpha_c,\beta_c)\bigr),
 \end{multline*}
 which with \eqref{eq:retzero} yields
 \begin{multline*}
 \sZ_{P^j,\Gamma_n}\cap\bigl((r_{\Gamma_n,\Gamma_m})^{-1}(\Gamma_m\cap\sR_{P^j})\bigr)\\
 \subset\bigcup_{c\in\crit_{P^j}\setminus\sfO_P(\infty)}
 \biggl(\sZ_{P^j,\Gamma_n}\cap\Bigl(\bigl(r_{\Gamma_n,\Gamma_n\cap[\alpha_c,\beta_c]}\bigr)^{-1}\bigl(\Gamma_n\cap(\alpha_c,\beta_c)\bigr)\Bigr)\biggr).
 \end{multline*}

 Let us see that
 for every $c\in\crit_{P^j}\setminus\sfO_P(\infty)$,
 making $\beta_c$ closer to $\alpha_c$ if necessary,
 \begin{gather*}
 \sZ_{P^j,\Gamma_n}\cap 
 \Bigl(\bigl(r_{\Gamma_{n},\Gamma_{n}\cap[\alpha_c,\beta_c]}\bigr)^{-1}\bigl(\Gamma_n\cap(\alpha_c,\beta_c)\bigr)\Bigr)
 =\emptyset\quad\text{for }n\gg 1,
 \end{gather*}
 which will complete the proof;
 for, otherwise, there are $c\in\crit_{P^j}\setminus\sfO_P(\infty)$,
 $n_1>n_2\gg 1$, and  
 $\xi_i\in \sZ_{P^j,\Gamma_{n_i}}$ for $i\in\{1,2\}$, such that 
$\eta_1\precneqq\eta_2$, where
\begin{gather*}
 \eta_i:=r_{\Gamma_{n_i},\Gamma_{n_i}\cap[\alpha_c,\beta_c]}(\xi_i)
=\xi_i\wedge_{\alpha_c}\beta_c\in 
V(\Gamma_{n_i})\cap(\alpha_c,\beta_c),\quad
 i\in\{1,2\}. 
\end{gather*} 
Then for each $i\in\{1,2\}$, 
$\xi_i\precneqq\eta_i$ and $\alpha_c\precneqq\eta_i$, and 
then we have $\xi_i\precneqq P^j(\xi_i)\precneqq P^j(\eta_i)$
 noting that $\xi_i\precneqq P^j(\xi_i)$ 
(since $\xi_i\in\sZ_{P^j,\Gamma_{n_j}}\subset\sfO_P(\infty)\setminus\{\infty\}$)
 and that $P^j(\xi_i)\precneqq P^j(\eta_i)$
 (by Fact \ref{th:polynomial}(\ref{item:ordpres})
 for $P^j$ and $\xi_i\precneqq\eta_i$),
 and in turn have $\eta_i\precneqq P^j(\eta_i)$ 
 noting also that $P^j(\eta_i)\in U(\widevec{\eta_i}\infty)$
 (by \eqref{eq:observation} 
 for $P^j$ and $\eta_i\in\sfO_P(\infty)\setminus\{\infty\}$).
 In particular,
 \begin{gather*}
 (P^j)_*\widevec{\eta_i\xi_i}
 =\widevec{P^j(\eta_i)P^j(\xi_i)}
 =\widevec{P^j(\eta_i)\xi_i}
 =\widevec{P^j(\eta_i)\eta_i}
 =\widevec{P^j(\eta_i)\alpha_c}
 \end{gather*}
 (the first equality in which is by Fact \ref{th:polynomial}(\ref{item:ordpres}) again), 
$i\in\{1,2\}$ (but only the $i=2$ case is needed).
 On the other hand, noting that $\eta_1\precneqq P^j(\eta_1)$
(seen above),
 that
 $P^j(\eta_1)\precneqq P^j(\eta_2)$ (by Fact \ref{th:polynomial}(\ref{item:ordpres}) for $P^j$ and $\eta_1\precneqq\eta_2$),
and $\alpha_c\precneqq\eta_1\precneqq\eta_2\precneqq P^j(\eta_2)$
(seen above), 
 we also have
 \begin{gather*}
 (P^j)_*\widevec{\eta_2\alpha_c}
 =(P^j)_*\widevec{\eta_2\eta_1}
 =\widevec{P^j(\eta_2)P^j(\eta_1)}
 =\widevec{P^j(\eta_2)\eta_1}
 =\widevec{P^j(\eta_2)\eta_2}
 =\widevec{P^j(\eta_2)\alpha_c} 
 \end{gather*}
 (the second equality in which is by Fact \ref{th:polynomial}(\ref{item:ordpres}) again).
 Now making $\beta_c$ closer to $\alpha_c$ further if necessary,
 by Fact \ref{th:polynomial}(\ref{item:locdegdirect})
and \eqref{eq:totaldirect}, we must have
 \begin{gather*}
 \deg_{\alpha_c}(P^j)=\deg_{\eta_2}(P^j)
 \ge m_{P^j}\bigl(\widevec{\eta_2\xi_2}\bigr)
+m_{P^j}\bigl(\widevec{\eta_2\alpha_c}\bigr)
 >0+\deg_{\alpha_c}(P^j),
 \end{gather*} 
 which is a contradiction. Now the proof is complete.
 \end{proof}

\section{Proof of Theorem \ref{introthm:bary}}\label{sec:pf-bary}

We establish the following detailed version of
Theorem \ref{introthm:bary} including
the Hausdorff convergence of $\BC(P^j,\Gamma_n)$
in $(\sH^1,\rho)$ as $n\to\infty$ 
for each $j\ge 1$;
the $\sL_n$, $\Gamma_n$, and
$\Gamma_\infty$ are defined using $P$.

\begin{theorem}\label{thm:bary}
Let $P\in K[z]$ be a polynomial of degree $>1$, and pick $j\ge 1$.
Then for every $n\ge 0$, there are $\xi_n=\xi_n^{(j)}\in\sV_{P^j,\Gamma_n}$ 
and $\eta_n=\eta_n^{(j)}\in\sV_{P^j,\Gamma_n}$ such that
\begin{gather*}
 \BC(\nu_{P^j,\Gamma_n})=[\xi_n,\eta_n]
\subset\sT_{P^j,\Gamma_n}\cap\sR_P\quad\text{and that}\quad
\xi_{n+1}\prec\xi_n\prec\eta_n\prec\xi_P,
\end{gather*}
and both the limits 
$\xi_{\infty}=\xi_\infty^{(j)}:=\lim_{n\to\infty}\xi_n$ 
and $\eta_{\infty}=\eta_\infty^{(j)}:=\lim_{n\to\infty}\eta_n$ in 
$\Gamma_\infty$ exist 
(so $\xi_\infty\prec\eta_\infty\prec\xi_P$). 
Moreover, 
\begin{enumerate}[{\em (a)}]
 \item \label{head:nontrivial}
       If $[\xi_m,\eta_m]$ is non-trivial for some $m\ge 0$, then 
 $\eta_n\equiv\eta_{\infty}\in\sV_{P^j,\Gamma_m}$ for every $n\ge m$. 
 \item \label{head:tri}
 One and only one of the following three statements is the case$;$ 
 \begin{enumerate}[{\em (i)}]
 \item\label{item:stationary} there is $m\ge 0$ such that $\xi_m\not\in\partial\Gamma_m$; then $[\xi_n,\eta_n]\equiv[\xi_{\infty},\eta_{\infty}]$ for every $n\ge m$, so that $\xi_\infty,\eta_\infty
\in\sV_{P^j,\Gamma_m}\setminus\partial\Gamma_m$.
 \item for every $n\ge 0$, 
$\xi_{n+1}\precneqq\xi_n\in\partial\Gamma_n$; then
       $\xi_\infty\in\partial\Gamma_\infty\setminus\{\infty\}$.
       \label{item:non-stationary}
 \item $P$ is simple, and then 
 $\xi_n\equiv\xi_\infty=\xi_P=\eta_\infty\equiv\eta_n\in\partial\Gamma_n\setminus\{\infty\}$ for every $n\ge 0$.
 \end{enumerate}
 \item \label{head:hausdorff}
$[\xi_{\infty},\eta_{\infty}]\subset\sH^1\cap\{\cdot\prec\xi_P\}\cap\Gamma_\infty\cap\sR_P$ and, in particular, the Hausdorff convergence
 $\lim_{n\to\infty}\BC(\nu_{P^j,\Gamma_n})=[\xi_{\infty},\eta_{\infty}]$ holds in $(\Gamma_\infty\cap\sH^1,\rho)$.
\end{enumerate}
\end{theorem}

In Theorem \ref{thm:bary}, when $P$ is non-simple,
the assumption in the statement
(\ref{head:nontrivial}) can be the case, and
both the statements (\ref{item:stationary}) and (\ref{item:non-stationary})
can also be the case (see Subsection \ref{sec:bestpossible}).

Let $P\in K[z]$ be a polynomial of degree $d>1$.

\begin{proof}[Proof of Theorem \ref{thm:bary} in the simple case]
Suppose that $P$ is simple. Then for every $j\ge 1$ and every $n\ge 0$,
 %$\MinResLoc_{P^j}\equiv\{\xi_P\}$ (see Subsection \ref{sec:simpleness})
 %and 
 we have already seen that $\nu_{P^j,\Gamma_n}\equiv\delta_{\xi_P}$ 
 on $\sP^1$ in Section \ref{sec:support}, so in particular
 \begin{gather}
 \BC(\nu_{P^j,\Gamma_n})\equiv\{\xi_P\}
 \equiv\sH^1\cap\{\cdot\prec\xi_P\}\cap
\Gamma_\infty\cap\sR_P. \label{eq:coincidence}
 \end{gather}
 Now all the assertions are clear.
\end{proof}

For a while, we do not assume that $P$ is non-simple.
The following is an improvement of Proposition \ref{coro:measure1/2}.

\begin{proposition}[a strict mass upper bound $1/2$ for Berkovich closed balls]\label{coro:measure1/2strong}
Pick $j\ge 1$ and $n\ge 0$.
Then for every $\xi\in V(\Gamma_n)\setminus\{\infty\}$ 
and every $\vec{v}\in T_{\xi}\sP^1\setminus\{\widevec{\xi\infty}\}$
but at most one direction, we have
$\nu_{P^j,\Gamma_n}(U(\vec{v}))<1/2$.
\end{proposition}

\begin{proof}
Pick $j\ge 1$, $n\ge 0$, and $\xi\in V(\Gamma_n)\setminus\partial\Gamma_n$.
Then $P^j(\xi)\neq\xi\in\sfO_P(\infty)\setminus\{\infty\}$.
For any $\vec{v}\in T_{\xi}\sP^1
\setminus\{\widevec{\xi\infty}\}$ such that
$\nu_{P^j,\Gamma_n}(U(\vec{v}))\ge 1/2$,
by the finiteness of the support 
of $\nu_{P^j,\Gamma_n}$, 
Lemma \ref{lem:measure-computation}, and 
Fact \ref{th:polynomial}(\ref{item:locdegdirect}), 
if $\xi'\in U(\vec{v})\cap(\Gamma_n\setminus\partial\Gamma_n)$ close enough to $\xi$, then
\begin{align*}
\frac{1}{2}\le\nu_{P^j,\Gamma_n}(U(\vec{v}))
=\nu_{P^j,\Gamma_n}(\{\cdot\prec\xi'\})
&=\frac{\deg_{\xi'}(P^j)-1}{d^j-1}\\
&=\frac{m_{P^j}(\widevec{\xi'\infty})-1}{d^j-1} 
=\frac{m_{P^j}(\widevec{\xi'\xi})-1}{d^j-1}
=\frac{m_{P^j}(\vec{v})-1}{d^j-1} 
\end{align*}
and 
$\xi'\prec P^j(\xi')$.
Hence if such ``exceptional'' $\vec{v}\in T_{\xi}\sP^1
\setminus\{\widevec{\xi\infty}\}$
exists, then $\xi\prec P^j(\xi)$
by the the continuity of $P^j$,
and for any $\xi' \in U(\vec{v})\cap(\Gamma_n\setminus\partial\Gamma_n)$ close enough to $\xi$,
we also have 
\begin{gather}
 (\xi'\precneqq)\xi\precneqq P^j(\xi');\label{eq:belowtoabove} 
\end{gather}
for, otherwise, there must exist a sequence $(\xi'(i))_{i=1}^\infty$ 
in $U(\vec{v})\cap(\Gamma_n\setminus\partial\Gamma_n)$ tending to $\xi$ as $i\to\infty$ such that 
$\xi'(i)\prec P^j(\xi'(i))\prec\xi$ for 
any $i\ge 1$, and then $P^j(\xi)=\lim_{i\to\infty}P^j(\xi'(i))=\xi$
by the continuity of $P^j$, which contradicts 
$P^j(\xi)\neq\xi$. 

Suppose to the contrary that there are distinct such ``exceptional'' 
$\widevec{v_1},\widevec{v_2}\in T_{\xi}\sP^1
\setminus\{\widevec{\xi\infty}\}$.
We claim that $(P^j)_*\widevec{v_1}=(P^j)_*\widevec{v_2}$; for, otherwise, we assume 
$(P^j)_*\widevec{v_1}\neq\widevec{P^j(\xi)\xi}$
without loss of generality. Then 
for any $\xi'\in U(\widevec{v_1})\cap(\Gamma_n\setminus\partial\Gamma_n)$ so close to $\xi$
that $(P^j)_*\widevec{v_1}=\widevec{P^j(\xi)P^j(\xi')}$, we must have
$\xi\precneqq P^j(\xi')\precneqq P^j(\xi)$
by \eqref{eq:belowtoabove} and Fact \ref{th:polynomial}(\ref{item:ordpres}),
so that $(P^j)_*\widevec{v_1}=\widevec{P^j(\xi)\xi}$,
which contradicts the assumption.
Once this claim is at our disposal, 
also by \eqref{eq:totaldirect},
we must have
\begin{gather*}
1=\frac{1}{2}+\frac{1}{2}
\le\frac{m_{P^j}(\widevec{v_1})-1}{d^j-1}+\frac{m_{P^j}(\widevec{v_2})-1}{d^j-1}
\le\frac{\deg_\xi(P^j)-2}{d^j-1}\le\frac{d^j-2}{d^j-1}<1, 
\end{gather*}
which is a contradiction.
\end{proof}

\begin{lemma}\label{th:main-bary-1}
For every $j\ge 1$ and every $n\ge 0$, the barycenter 
$\BC(\nu_{P^j,\Gamma_n})$
is a (possibly trivial) interval in $\sT_{P^j,\Gamma_n}\cap\sR_{P}$ having end points in
$\sV_{P^j,\Gamma_n}$. 
Moreover, 
$r_{\sP^1,\BC(\nu_{P^j,\Gamma_n})}(\infty)
\in\partial\BC(\nu_{P^j,\Gamma_n})$.
\end{lemma}

\begin{proof}
Everything in the first assertion
but the inclusion $\BC(\nu_{P^j,\Gamma_n})\subset\sR_P$ is 
by Proposition \ref{lem:baryhalf}(\ref{head:barymin}). 
For every $\xi\in\Gamma_n\setminus\sR_P$, %$(\subset\Gamma_n\setminus\{\infty\})$, 
noting that $\deg_\xi(P^j)=\deg_\xi (P)\cdot\deg_{P(\xi)}(P^{j-1})\le 1\cdot d^{j-1}$,
by Lemma \ref{lem:measure-computation}, we have
\begin{gather*}
\nu_{P^j,\Gamma_n}\bigl(U(\widevec{\xi\infty})\bigr)
=1-
\nu_{P^j,\Gamma_n}(\{\cdot\prec\xi\})
\ge 1-\frac{d^{j-1}-1}{d^j-1}>\frac{1}{2},
\end{gather*}
so that $\xi\in\Gamma_n\setminus\BC(\nu_{P^j,\Gamma_n})$. Hence this remaining inclusion also holds.
Let us see the second assertion.
Suppose that 
$\eta:=r_{\sP^1,\BC(\nu_{P^j,\Gamma_n})}(\infty)\not\in\partial\BC(\nu_{P^j,\Gamma_n})$. 
Then $\BC(\nu_{P^j,\Gamma_n})$ is non-trivial,
and by Proposition \ref{coro:measure1/2strong}, 
there is $\xi\in\partial\BC(\nu_{P^j,\Gamma_n})$ such that 
$\nu_{P^j,\Gamma_n}(U(\widevec{\eta\xi}))<1/2$, which 
contradicts Lemma \ref{th:bcchar}(\ref{item:everypoint}).
\end{proof}

\begin{proof}[Proof of Theorem \ref{thm:bary} in the non-simple case]
Suppose that $P$ is non-simple, 
and pick $j\ge 1$. Then for each $n\ge 0$, 
by Theorem \ref{thm:support}(\ref{item:vertex}) and
Lemma \ref{th:main-bary-1},
we have $\BC(\nu_{P^j,\Gamma_n})=[\xi_n,\eta_n]\subset\sT_{P^j,\Gamma_n}\cap\sR_P$, where
$\xi_n,\eta_n\in\sV_{P^j,\Gamma_n}\subset\sH^1_{\mathrm{II}}$ and
$\xi_n\prec\eta_n\prec\xi_P$. 

By Fact \ref{th:inparticular}(\ref{item:ends})({\ref{item:endext}}), if $\xi_{n+1}\precneqq\xi_n\in\sL_n$ for every $n\ge 0$,
then writing $\xi_n=\xi_{B_n}$ for each $n\ge 0$,
we have $\lim_{n\to\infty}\xi_n=\xi_\infty\in
\partial\Gamma_\infty\setminus\{\infty\}$ in $\sP^1$
where $\xi_\infty$ is the point in $\sP^1$ represented by
the decreasing sequence $(B_n)$ of the $K$-closed disks.
Now everything about the asymptotic of 
$[\xi_n,\eta_n]$ as $n\to\infty$
but the statement (\ref{head:hausdorff}) holds
by Proposition \ref{lem:baryhalf}(\ref{head:singleton})
(and Fact \ref{th:inparticular}(\ref{item:ends})({\ref{item:endext}}) again).
In any case, we also have $[\xi_\infty,\eta_\infty]
\subset\sR_P$
since $\sR_P$ is closed. 
Let us see $\xi_\infty\in\sH^1$, which will complete the proof; 
indeed, this is the case
by $\xi_\infty\equiv\xi_n\in\sH^1_{\mathrm{II}}$ 
for $n\gg 1$
when the item (\ref{item:stationary}) is the case.
Suppose that the item (\ref{item:non-stationary}) is the case.
Then for every $n\ge 1$, 
by Lemma \ref{lem:measure-computation}
and Fact \ref{th:inparticular}(\ref{item:ends})(\ref{item:endext}), 
we not only have
\begin{gather*}
 0<\frac{1}{2}\le 1-\nu_{P^j,\Gamma_n}(U_{\widevec{\xi_n\xi_{n-1}}})
=\nu_{P^j,\Gamma_n}(\{\cdot\prec\xi_n\})
=\frac{\deg_{\xi_n}(P^j)-1}{d^j-1},
\end{gather*}
so $1/(\deg_{\xi_n}(P^j))\le 2/(d^j+1)<1$,
but also have $\xi_n\prec P^j(\xi_n)=\xi_{n-j}\in\sL_{n-j}$.
Hence also by Fact \ref{th:polynomial}(\ref{head:monotone}),
for every $n\ge 2j$, $P^j$ restricts to
a homeomorphism $[\xi_n,\xi_{n-j}]\to[\xi_{n-j},\xi_{n-2j}]$, 
which preserves $\prec$ and the inverse of which is
$(1/\deg_{\xi_n}(P^j))$-Lipschitz continuous 
in the $\rho$-length parameters. Consequently, we have
$\rho(\xi_n,\xi_{n-j})=O((2/(d^j+1))^{n/j})$ as $n\to\infty$
and in turn have $\rho(\xi_P,\xi_{\infty})\le\limsup_{n\to\infty}\rho(\xi_P,\xi_n)<+\infty$ by the triangle inequality. 
\end{proof}

\section{Proof of Theorem \ref{thm:minres}}\label{sec:pf-minres}
Let $P\in K[z]$ be a polynomial of degree $d>1$;
the $\sL_n$, $\Gamma_n$, and $\Gamma_\infty$ are defined using $P$. 
When $P$ is simple, all the assertions in Theorem \ref{thm:minres} are clear
from \eqref{eq:coincidence} and $\MinResLoc_{P^j}\equiv\{\xi_P\}$ for every $j\ge 1$.

In the rest of this section,
suppose that $P$ is non-simple. 

\subsection{Proof of Theorem \ref{thm:minres}(\ref{item:bcmrl}) in the non-simple case}\label{sec:A1}
Pick $j\ge 1$, and set 
\begin{gather*}
 \Fix_{P^j}:=\{\text{classical fixed points of }P^j\}.
\end{gather*}
Note that $\Gamma_{P^j}^{\FP}=\Span(\Fix_{P^j})$ 
so $\MinResLoc_{P^j}$ is an interval in $\Span(\Fix_{P^j})\cap\sH^1$
(see Remark \ref{th:crucialtree}).

\begin{lemma}\label{lem:MinResLoc-in-tree}
$\MinResLoc_{P^j}\subset\Gamma_\infty$. 
\end{lemma}

\begin{proof}
Pick $\xi\in(\Span(\Fix_{P^j})\cap\sH^1_{\mathrm{II}})\setminus\Gamma_\infty$.
By Fact \ref{th:inparticular}(\ref{item:exhaust}), 
there is $z_0\in(\{\cdot\precneqq\xi\}\cap\Fix_{P^j})\setminus\sJ_P$ and,
setting
$\alpha:=r_{\sP^1,\Span(\Fix_{P^j})\cap\Gamma_\infty}(z_0)$,
we have 
\begin{gather*}
 z_0\precneqq\xi\precneqq\alpha\precneqq\infty,
\end{gather*}
and the Berkovich open ball $U(\widevec{\alpha z_0})$
is the Berkovich Fatou component $\sfO_P(z_0)$
of $P$ containing (the either (super)attracting or indifferent
fixed point) $z_0$ of $P^j$
and having $\partial\sfO_P(z_0)=\{\alpha\}$,
so in particular $P^j(\alpha)=\alpha$. 
By Fact \ref{th:polynomial}(\ref{head:monotone}), 
$P^j$ restricts to a homeomorphic automorphism
of $[z_0,\alpha]$ preserving $\prec$, so in particular
\yo{
\begin{gather*}
 (P^j)_*\widevec{\alpha z_0}\Bigl(=\widevec{\alpha(P^j(z_0))}\Bigr)=\widevec{\alpha z_0}.
\end{gather*}
By the difference formula \eqref{eq:difference},
Fact \ref{th:local}, the defining equality \eqref{eq:surplusdefining},
and Fact \ref{th:polynomial}(\ref{head:direcsurp}),
we have
\begin{gather*}
 \rd_{\widevec{\alpha z_0}}\Crucial_{P^j}=\frac{1}{2}+\frac{1}{d^j-1}(0-0)=\frac{1}{2}>0
\end{gather*}
noting that $\xi'\mapsto\rho(\xi',P^j(\xi')\wedge_{\alpha}\xi')\equiv 0$
on $(\alpha,z_0]$ since $P^j(\alpha)=\alpha$ and 
$(P^j)_*\widevec{\alpha z_0}=\widevec{\alpha z_0}$ and that
$((P^j)^*\delta_{\alpha})(U(\widevec{\alpha z_0}))=s_{\widevec{\alpha z_0}}(P^j)=0$ since $P^j(\alpha)=\alpha\in\sP^1\setminus U(\widevec{\alpha z_0})$
and $\widevec{\alpha z_0}\neq\widevec{\alpha\infty}$.
Hence the $\Crucial_{P^j}$ does not take the minimum on $U(\widevec{\alpha z_0})$
by the convexity \eqref{eq:convex} of $\Crucial_{P^j}$, and we are done by \eqref{eq:samefinite}.
}
\end{proof}

Let us show the desired equality $[\xi_\infty,\eta_\infty]=\MinResLoc_{P^j}$ using Theorem \ref{thm:bary} and the notations there.
If the interval $[\xi_\infty,\eta_\infty]$ is non-trivial, then
so is $\BC(\nu_{P^j,\Gamma_n})=[\xi_n,\eta_n]$ for $n\gg 1$, 
and $\eta_n\equiv\eta_\infty$ for $n\gg 1$ and
$\lim_{n\to\infty}[\xi_n,\eta_n]=[\xi_\infty,\eta_\infty]$
increasingly. 
Then the inclusion
$[\xi_\infty,\eta_\infty]\subset\MinResLoc_{P^j}$ 
follows from the implication 
``(\ref{item:whole})$\Rightarrow$(\ref{item:minimum})'' in Lemma \ref{th:bcchar} 
and the continuity of $\Crucial_{P^j}$, and so does 
the reverse inclusion
$[\xi_\infty,\eta_\infty]\supset\MinResLoc_{P^j}$
from Lemma \ref{lem:MinResLoc-in-tree},
Fact \ref{th:inparticular}({\ref{item:exhaust}}), 
and the implication ``(\ref{item:whole})$\Leftarrow$(\ref{item:minimum})'' in Lemma \ref{th:bcchar}.

Alternatively if $[\xi_\infty,\eta_\infty]$ is trivial,
then so is $\BC(\nu_{P^j,\Gamma_n})=[\xi_n,\eta_n]$ 
for $n\ge 0$, and $\lim_{n\to\infty}\xi_n^{(j)}=\xi_\infty^{(j)}$
decreasingly in $(\Gamma_\infty,\prec)$
writing $[\xi_\infty,\eta_\infty]=\{\xi_{\infty}^{(j)}\}$
and $[\xi_n,\eta_n]=\{\xi_n^{(j)}\}$, $n\ge 0$.
Then 
$\MinResLoc_{P^j}$ is also trivial by 
Lemma \ref{lem:MinResLoc-in-tree},
Fact \ref{th:inparticular}({\ref{item:exhaust}}),
and the implication 
``(\ref{item:whole})$\Leftarrow$(\ref{item:minimum})''
in Lemma \ref{th:bcchar},
and writing $\MinResLoc_{P^j}=\{\xi_{\min}^{(j)}\}
\subset\Gamma_\infty\cap\sH^1$, we have
\begin{gather}
r_{\Gamma_\infty,\Gamma_n}(\xi_{\min}^{(j)})
=\xi_n^{(j)}\quad\text{for every }n\ge 0\label{eq:retractmin} 
\end{gather}
by \eqref{eq:samefinite} and the convexity \eqref{eq:convex}.
If $\xi_{\min}^{(j)}\in\Gamma_\infty\setminus\partial\Gamma_\infty$, then even $\xi_n^{(j)}\equiv\xi_{\min}^{(j)}$ for $n\gg 1$
by Fact \ref{th:inparticular}({\ref{item:exhaust}}).
Finally, if $\xi_{\min}^{(j)}\in\partial\Gamma_\infty$, then
even $\xi_{\min}^{(j)}\prec\xi_{n+1}^{(j)}\precneqq\xi_n^{(j)}$
for every $n\ge 0$, and moreover 
$\xi_{\infty}^{(j)}\in\partial\Gamma_\infty$
by Fact \ref{th:inparticular}({\ref{item:exhaust}}).
Hence $\xi_{\min}^{(j)}\prec\xi_{\infty}^{(j)}$, 
and in turn
$\xi_{\min}^{(j)}=\xi_{\infty}^{(j)}$ 
since both sides are in $\partial\Gamma_\infty$.
\qed

\subsection{Proof of Theorem \ref{thm:minres}(\ref{item:mrlstat}) in the non-simple case}\label{sec:barystat}
By Theorems \ref{thm:bary} and \ref{thm:minres}(\ref{item:bcmrl}),
for every $j\ge 1$ and every $n\ge 0$,
there are $\alpha_j,\beta_j\in\Gamma_\infty\cap\sH^1_{\mathrm{II}}$ and $\xi_n^{(j)},\eta_n^{(j)}\in\Gamma_n\cap\sH^1_{\mathrm{II}}$ such that
\begin{gather*}
 \MinResLoc_{P^j}
=[\alpha_j,\beta_j]\subset\sH^1\cap\{\cdot\prec\xi_P\}\cap
\Gamma_\infty\cap\sR_P,\\
\BC(\nu_{P^j,\Gamma_n})=[\xi_n^{(j)},\eta_n^{(j)}],\\
\alpha_j\prec\beta_j\prec\xi_P,\quad \alpha_j\prec\xi_n^{(j)}\prec\eta_n^{(j)}\prec\xi_P,
\quad\text{and}\\
\xi_n^{(j)},\eta_n^{(j)}\to
\alpha_j,\beta_j\text{ in }(\Gamma_\infty\cap\sH^1,\rho)\text{ as }n\to\infty,\text{ respectively}.
\end{gather*}
Recall again the characterization
of the (GIT)-semistability of a coefficient reduction 
of a rational function in terms of its depth
(see Subsection \ref{sec:res}).

\begin{lemma}\label{lem:deg}
For every $j\ge d-1$ and every $\xi\in[\alpha_j,\beta_j]$, 
we have
\begin{enumerate}[{\em (a)}]
\item \label{item:minresintree}
      $\xi\in\Gamma_{j-1}$,
so $[\alpha_j,\beta_j]\subset\Gamma_{j-1}$, 
\item \label{item:degreelower}
      $\deg_{\xi}(P^j)\ge(d^j+1)/2$, and 
\item \label{item:upper}
      for every $\vec{v}\in(T_{\xi}\Gamma_\infty)\setminus\{\widevec{\xi\infty}\}$, 
\begin{gather*}
 m_{P^j}(\vec{v})
\begin{cases}
=\displaystyle\frac{d^j+1}{2} & \text{if }
\vec{v}\in (T_{\xi}[\alpha_j,\beta_j])\setminus\{\widevec{\xi\infty}\},\\
\le\displaystyle\frac{
d^j-
\begin{cases}
  1 & \text{if }d\text{ is odd}\\
  0 & \text{if }d\text{ is even}
 \end{cases}
}{2}\le\frac{d^j}{2} & \text{otherwise}.
 \end{cases}\label{item:73}
\end{gather*}
\end{enumerate}
\end{lemma}

\begin{proof}
Suppose first $j=1=d-1$. Then
$\deg_\xi(P)\equiv 1(=d-1)$ on $\{\cdot\precneqq\xi_P\}\cap\Gamma_\infty$ by Fact \ref{th:inparticular}(\ref{item:nonmaximal}),
so $\{\cdot\prec\xi_P\}\cap\Gamma_\infty\cap\sR_P\subset\{\xi_P\}$.
The reverse inclusion
$\{\cdot\prec\xi_P\}\cap\Gamma_\infty\cap\sR_P\supset\{\xi_P\}$ 
is by Lemma \ref{th:basetoinfty}(\ref{item:totram}) and Fact \ref{th:inparticular}(\ref{item:exhaust}). 
Hence $[\alpha_1,\beta_1]=\{\xi_P\}=\sL_0\subset\Gamma_0$, so
the items (\ref{item:minresintree}) and (\ref{item:degreelower})
are the cases, and so is the item (\ref{item:upper}) also by
Fact \ref{th:inparticular}(\ref{item:nonmaximal}).
 
From now on, pick $j>1$ satisfying either 
$j\ge d-1>1$ or $j>1=d-1$.

\subsection*{Proof of (\ref{item:minresintree})}
Suppose to the contrary that 
$[\alpha_j,\beta_j]\setminus\Gamma_{j-1}\neq\emptyset$\yo{, and then
$\alpha_j\in\Gamma_\infty\setminus\Gamma_{j-1}$}. 
Then 
we have
\begin{gather}
\deg_{\yo{\alpha_j}}(P^j)=\prod_{\ell=0}^{j-1}\deg_{P^\ell(\yo{\alpha_j})}(P)
\le(d-1)^j\le\frac{d^j-1}{2}, \label{eq:degupper}
\end{gather}
where the first inequality holds by Fact \ref{th:inparticular}(\ref{item:nonmaximal}), and so does the second one under the above assumption on $d,j$.

\yo{We claim that 
\begin{gather*}
 P^j(\yo{\alpha_j})=\yo{\alpha_j}
\in\sJ_P(=\partial\Gamma_\infty\setminus\{\infty\}
\text{ by Fact \ref{th:inparticular}(\ref{item:exhaust})}),
\end{gather*}
which also yields $(P^j)_*(\widevec{\alpha_j\infty})=\widevec{\alpha_j\infty}$
by Fact \ref{th:polynomial}(\ref{head:direcsurp});
for, unless $\yo{\alpha_j}\in\sJ_P$,}
for $n\gg 1$, by Fact \ref{th:inparticular}({\ref{item:exhaust}}) and Lemma \ref{th:bcchar}, we must have \yo{$\xi_n^{(j)}\equiv\alpha_j$},
and then by Lemma \ref{lem:measure-computation} and \eqref{eq:degupper}, 
we compute
\begin{multline*}
\frac{1}{2}\ge\nu_{P^j,\Gamma_n}(U(\widevec{\yo{\alpha_j}\infty}))
=1-\nu_{P^j,\Gamma_n}(\{\cdot\prec\yo{\alpha_j}\})\\
 \ge 1-\frac{\deg_{\yo{\alpha_j}}(P^j)-1}{d^j-1}
=\frac{d^j-\deg_{\yo{\alpha_j}}(P^j)}{d^j-1}
\ge\frac{1}{2}\frac{d^j+1}{d^j-1},
\end{multline*}
which is a contradiction. \yo{Hence $\yo{\alpha_j}\in\sJ_P$. On the other hand,}
by \eqref{eq:same} and the convexity \eqref{eq:convex},
we have $\rd_{\widevec{\xi\yo{\alpha_j}}}\Crucial_{P^{d-1}}\le 0$
for every $\xi\in(\yo{\alpha_j},\infty)\subset\Gamma_\infty\setminus\partial\Gamma_\infty$,
so that for $n\gg 1$, using Proposition \ref{th:monotonicity}(\ref{item:closedopen}) and
the slope formula \eqref{eq:slope}, we have
\begin{gather*}
\nu_{P^j,\Gamma_n}(\{\cdot\prec\xi\})
 \ge
 \nu_{P^j,\Gamma_n}\bigl(U(\widevec{\xi\yo{\alpha_j}})\bigr)
=\frac{1}{2}-\rd_{\widevec{\xi\yo{\alpha_j}}}\Crucial_{P^j}
 \ge \frac{1}{2},
\end{gather*}
and in turn $\xi\prec P^j(\xi)$
by Lemma \ref{lem:measure-computation}.
Hence $\yo{\alpha_j}\prec P^j(\yo{\alpha_j})$ by the continuity of $P^j$,
and in turn $P^j(\yo{\alpha_j})=\yo{\alpha_j}$ since 
$P^j(\yo{\alpha_j})\in P^j(\sJ_P)=\sJ_P$ and 
$(\yo{\alpha_j},\infty]\subset\sfO_P(\infty)$.

\yo{Once the claim is at our disposal, by the difference formula \eqref{eq:difference}, Fact \ref{th:local}, 
the defining equality \eqref{eq:surplusdefining},
and Fact \ref{th:polynomial}(\ref{head:direcsurp}),
we must have
\begin{gather*}
 0\le\rd_{\widevec{\alpha_j\infty}}\Crucial_{P^j}
=\frac{1}{2}+\frac{1}{d^j-1}(0-(d^j-\deg_{\alpha_j}(P^j)))\le\frac{1}{2}\Bigl(1-\frac{d^j+1}{d^j-1}\Bigr)<0
\end{gather*}
noting that $\xi\mapsto(\xi,P^j(\xi)\wedge_{\alpha_j}\xi)\equiv 0$
on $(\alpha_j,\infty]$ since $P^j(\alpha_j)=\alpha_j$ and 
$(P^j)_*\widevec{\alpha_j\infty}=\widevec{\alpha_j\infty}$ and that
$((P^j)^*\delta_{\alpha_j})(U(\widevec{\alpha_j\infty}))
=s_{\widevec{\alpha_j\infty}}(P^j)+0
(=s_{\widevec{\alpha_j\infty}}(P^j)=d^j-\deg_{\alpha_j}(P^j))$
since $\alpha_j=P^j(\alpha_j)\in\sP^1\setminus U((P^j)_*(\widevec{\alpha_j\infty}))$.
This is a contradiction. \qed
}

\subsection*{Proofs of (\ref{item:degreelower}) and (\ref{item:upper})} 
Pick $\xi\in[\alpha_j,\beta_j]\subset\{\cdot\prec\xi_P\}\cap\Gamma_{j-1}\subset\Gamma_j$ (by the item (\ref{item:minresintree})
and Fact \ref{th:inparticular}(\ref{item:endext})), so that
$\xi\prec P^j(\xi)$. First, we compute
\begin{align*}
\deg_{\xi}(P^j)
&=(d^j-1)\cdot
\nu_{P^j,\Gamma_{j-1}}(\{\cdot\prec\xi\})+1
\quad\text{(by Lemma \ref{lem:measure-computation})}\\
&=(d^j-1)\cdot\Bigl(1-\nu_{P^j,\Gamma_{j-1}}(U(\widevec{\xi\infty}))\Bigr)+1\\
&=(d^j-1)\Bigl(1-\bigl(\frac{1}{2}-\rd_{\widevec{\xi\infty}}\Crucial_{P^j}\bigr)\Bigr)+1
\quad\text{(by the slope formula \eqref{eq:slope})}\\
&\ge(d^j-1)\Bigl(\frac{1}{2}+0\Bigr)+1=\frac{d^j+1}{2}
\quad(\text{by \eqref{eq:same}}),
\end{align*}
which completes the proof of the item (\ref{item:degreelower}).

Next, for every 
$\vec{v}\in T_{\xi}\Gamma_\infty\setminus\{\widevec{\xi\infty}\}
(=T_{\xi}\Gamma_j\setminus\{\widevec{\xi\infty}\}$
by Fact \ref{th:inparticular}(\ref{item:exhaust}))
and every $\xi'\in U(\vec{v})$ close enough to $\xi$, 
we still have $\xi'\prec P^j(\xi')$ and compute
\begin{align*}
\frac{1}{2}+\rd_{\widevec{\xi'\xi}}\Crucial_{P^j} 
&=1-\nu_{P^j,\Gamma_j}(U(\widevec{\xi'\xi}))
\quad\text{(by the slope formula \eqref{eq:slope})}\\
&=\nu_{P^j,\Gamma_j}(\{\cdot\prec\xi'\})
=\frac{\deg_{\xi'}(P^j)-1}{d^j-1}\quad\text{(by Lemma \ref{lem:measure-computation})}\\
&=\frac{m_{P^j}(\widevec{\xi'\xi})-1}{d^j-1}
=\frac{m_{P^j}(\vec{v})-1}{d^j-1}
\quad\text{(by Fact \ref{th:polynomial}(\ref{item:locdegdirect})\text{ and \eqref{eq:multiplicity})}},
\end{align*}
so that $m_{P^j}(\vec{v})= 
(d^j+1)/2+(d^j-1)\cdot\rd_{\widevec{\xi'\xi}}\Crucial_{P^j}$.
On the other hand, 
$\rd_{\widevec{\xi'\xi}}\Crucial_{P^j}\le 0$
(by \eqref{eq:same}), where
the equality holds if 
$\vec{v}\in(T_\xi[\alpha_j,\beta_j])\setminus\{\widevec{\xi\infty}\}$
and otherwise,
\begin{gather*}
(d^j-1)\cdot\rd_{\widevec{\xi'\xi}}\Crucial_{P^j}\le
\begin{cases}
 -1 & \text{if }d\text{ is odd}\\
 -1/2 & \text{if }d\text{ is even}
\end{cases}
\end{gather*}
(by the range \eqref{eq:sloperange} of the slopes of $\Crucial_{P^j}$
and the piecewise affineness of $\Crucial_\phi$ on $(\sH^1,\rho)$). Now the proof of the item (\ref{item:upper}) is also complete.
\end{proof} 

For every $j\ge 1$ and every $n\ge j$, 
by Fact \ref{th:inparticular}(\ref{item:endext}), 
Lemma \ref{lem:deg}(\ref{item:minresintree}), and
Theorem \ref{thm:bary}(\ref{item:stationary}), 
we have
\begin{gather}
[\xi_n^{(j)},\eta_n^{(j)}]\equiv[\alpha_j,\beta_j]
\subset\{\cdot\prec\xi_P\}\cap\Gamma_{j-1},\label{eq:barymin}
\end{gather}
which is in fact the case for every $n\ge j-1$
also by \eqref{eq:samefinite}. 

Let us see by step by step the desired
\begin{gather}
 \alpha_j\equiv\alpha_{d-1}=\beta_{d-1}\equiv\beta_j
\quad\text{for every }j\ge d-1.\label{eq:minreslocstationary}
\end{gather}

\begin{mclaim}\label{th:leftendequal}
\yo{$\alpha_j=\alpha_{d-1}$ for every $j\ge d-1$.}
\end{mclaim}

\begin{proof}
Pick $j\ge d-1$.
Then by \eqref{eq:barymin} applied to $(j=)d-1$ 
and Lemma \ref{th:basetoinfty}(\ref{item:diverge})
under the non-simpleness assumption, we have
\begin{gather*}
 \alpha_{d-1}\prec\xi_P\prec P^{d-2}(\alpha_{d-1})\precneqq P^j(\alpha_{d-1}),
\end{gather*}
\yo{so that $(P^j)_*(\widevec{\alpha_{d-1}\infty})=\widevec{(P^j(\alpha_{d-1}))\infty}\neq\widevec{(P^j(\alpha_{d-1}))\alpha_{d-1}}$ so
$\alpha_{d-1}\in\sP^1\setminus U((P^j)_*(\widevec{\alpha_{d-1}\infty}))$
by Fact \ref{th:polynomial}(\ref{head:direcsurp}) and that}
\begin{align*}
 s_{P^j}(\widevec{\alpha_{d-1}\infty}) 
 &=d^j-\deg_{\alpha_{d-1}}(P^j)\quad(\text{by Fact \ref{th:polynomial}(\ref{head:direcsurp})})\\
 &=d^j-
\deg_{\alpha_{d-1}}(P^{d-1})\cdot\prod_{\ell=0}^{j-d+1}\deg_{P^{d-1+\ell}(\alpha_{d-1})}(P)\\
&\le d^j-\frac{d^{d-1}+1}{2}\cdot d^{j-d+1}\\
&=\yo{\frac{d^j-d^{j-d+1}}{2}\le\frac{d^j-1}{2}}
\quad(\text{using Lemma \ref{lem:deg}\eqref{item:degreelower} and Lemma \ref{th:basetoinfty}(\ref{item:diverge})(\ref{item:totram}})).
\end{align*}
\yo{Hence by the difference formula \eqref{eq:difference},
Fact \ref{th:local}, and the defining equality \eqref{eq:surplusdefining},
we have the first positivity
\begin{gather*}
 \rd_{\widevec{\alpha_{d-1}\infty}}\Crucial_{P^j}
=\frac{1}{2}+\frac{1}{d^j-1}\bigl(0-((P^j)^*\delta_{\alpha_{d-1}})(U(\widevec{\alpha_{d-1}\infty}))\bigr)\ge 0
\end{gather*}
noting that $\xi\mapsto(\xi,P^j(\xi)\wedge_{\alpha_{d-1}}\xi)\equiv 0$
on $(\alpha_{d-1},\infty]$ since $P^j(\alpha_{d-1})\neq\alpha_{d-1}$ and 
$\widevec{\alpha_{d-1}(P^j(\alpha_{d-1}))}=\widevec{\alpha_{d-1}\infty}$ 
and that
$((P^j)^*\delta_{\alpha_{d-1}})(U(\widevec{\alpha_{d-1}\infty}))
=s_{P^j}(U(\widevec{\alpha_{d-1}\infty}))+0(\le(d^j-1)/2)$ 
since $\alpha_{d-1}\in\sP^1\setminus U((P^j)_*(\widevec{\alpha_{d-1}\infty}))$.} 

\yo{On the other hand,} for every $\vec{v}\in T_{\alpha_{d-1}}\sP^1\setminus\{\widevec{\alpha_{d-1}\infty}\}$,
we also have 
$\widevec{\alpha_{d-1}(P^j(\alpha_{d-1}))}=\widevec{\alpha_{d-1}\infty}\neq\vec{v}$ and 
\begin{gather*}
  m_{P^j}(\vec{v})
 =m_{P^{d-1}}(\vec{v})\cdot  m_{P^{j-d+1}}\bigl((P^{d-1})_*\vec{v}\bigr)
 \le\frac{d^{d-1}}{2}\cdot d^{j-d+1}=\frac{d^j}{2}
\end{gather*}
(by Lemma \ref{lem:deg}\eqref{item:73} \yo{applied to $P^{d-1}$}),
\yo{so that by the difference formula \eqref{eq:difference},
Fact \ref{th:local}, the defining equality \eqref{eq:surplusdefining},
and Fact \ref{th:polynomial}(\ref{head:direcsurp}),
we have the second positivity
\begin{gather*}
 \rd_{\vec{v}}\Crucial_{P^j}
=\frac{1}{2}+\frac{1}{d^j-1}\bigl(1-((P^j)^*\delta_{\alpha_{d-1}})(U(\vec{v}))\bigr)\ge\frac{1}{2}\Bigl(1-\frac{d^j-2}{d^j-1}\Bigr)>0
\end{gather*}
noting that $\xi\mapsto\rho(\xi,P^j(\xi)\wedge_{\alpha_{d-1}}\xi)=\rho(\alpha_{d-1},\xi)$ on $U(\vec{v})$ so $\rd_{\vec{v}}(\xi\mapsto\rho(\xi,P^j(\xi)\wedge_{\alpha_{d-1}}\xi))=1$
since $P^j(\alpha_{d-1})\neq\alpha_{d-1}$ and 
$\widevec{\alpha_{d-1}(P^j(\alpha_{d-1}))}\neq\vec{v}$ 
and that $((P^j)^*\delta_{\alpha_{d-1}})(U(\vec{v}))
\le s_{P^j}(\vec{v})+m_{P^j}(\vec{v})=0+m_{P^j}(\vec{v})(\le d^j/2)$. Hence we have 
$\alpha_{d-1}\in[\alpha_j,\beta_j]$ by the convexity \eqref{eq:convex} of $\Crucial_{P^j}$,
and in turn $\alpha_{d-1}=\alpha_j$ since the second positivity is strict.}
\end{proof}

\begin{mclaim}
$\alpha_{d-1}=\beta_{d-1}$. 
\end{mclaim}

\begin{proof}
Suppose to the contrary that $\alpha_{d-1}\precneqq\beta_{d-1}(\prec\xi_P)$.
Since $\alpha_{d-1}\in[\alpha_{d-1},\beta_{d-1}]\subset\{\cdot\prec\xi_P\}\cap\Gamma_{d-2}$ 
(by \ref{eq:barymin}), we have both
$\xi_P\prec P^{d-2}(\alpha_{d-1})\precneqq P^{d-2}(\beta_{d-1})$ and $(P^{d-2})_*\widevec{\beta_{d-1}\alpha_{d-1}}=\widevec{P^{d-2}(\beta_{d-1})P^{d-2}(\alpha_{d-1})}$
(using Fact \ref{th:polynomial}(\ref{item:ordpres})).
Hence we must have
\begin{align*}
\frac{d^{d-1}+1}{2}
&=m_{P^{d-1}}\bigl(\widevec{\beta_{d-1}\alpha_{d-1}}\bigr)\quad(\text{by Lemma \ref{lem:deg}(\ref{item:upper}) applied to $\beta_d$ and $\widevec{\beta_{d-1}\alpha_{d-1}}$})\\
&=m_{P^{d-2}}\bigl(\widevec{\beta_{d-1}\alpha_{d-1}}\bigr)\cdot m_P\bigl((P^{d-2})_*\widevec{\beta_{d-1}\alpha_{d-1}}\bigr)\\
&=m_{P^{d-2}}\bigl(\widevec{\beta_{d-1}\alpha_{d-1}}\bigr)\cdot d
\quad(\text{by Lemma \ref{th:basetoinfty}(\ref{item:directional}))}\\
&\ge m_{P^{d-2}}\bigl(\widevec{\alpha_{d-1}\beta_{d-1}}\bigr)\cdot d
=m_{P^{d-2}}\bigl(\widevec{\alpha_{d-1}\infty}\bigr)\cdot d
\quad(\text{by Fact \ref{th:polynomial}(\ref{head:locdegdirec})})\\
&=\deg_{\alpha_{d-1}}(P^{d-2})\cdot d
\quad(\text{by Fact \ref{th:polynomial}(\ref{head:direcsurp})})\\
&=\deg_{\alpha_{d-1}}(P^{d-2})\cdot\deg_{P^{d-2}(\alpha_{d-1})}P=\deg_{\alpha_{d-1}}\bigl(P^{d-1}\bigr)
\quad(\text{by Lemma \ref{th:basetoinfty}(\ref{item:totram})})\\
&\ge\frac{d^{d-1}+1}{2}\quad(\text{by Lemma \ref{lem:deg}(\ref{item:degreelower}) applied to }\alpha_{d-1}\in[\alpha_{d-1},\beta_{d-1}]),
\end{align*}
which yields
$\deg_{\alpha_{d-1}}(P^{d-2})=(d^{d-1}+1)/(2d)\not\in\{1,2,\ldots,d^{d-2}\}$.
This is a contradiction.
\end{proof}

The following final claim will complete the proof of 
\yo{\eqref{eq:minreslocstationary}}.
\begin{mclaim}
$\beta_j\prec\beta_{d-1}$ for every $j\ge d-1$.
\end{mclaim}

\begin{proof}
This is clear when $j=d-1$. Pick $j\ge d-1$.
Since $\beta_j\in[\alpha_j,\beta_j]\subset\{\cdot\prec\xi_P\}\cap\Gamma_{j-1}
\subset\Gamma_j\subset\Gamma_{j+1}$ (by \ref{eq:barymin} and Fact \ref{th:inparticular}(\ref{item:endext})), we have
$\{P^j(\beta_j),P^{j+1}(\beta_j)\}\subset
[\xi_P,\infty]\subset[\beta_j,\infty]$, so that
\begin{gather*}
 \beta_j\prec P^j(\beta_j)\quad\text{and}\quad\beta_j\prec P^{j+1}(\beta_j).
\end{gather*} 
Noting that $\widevec{\beta_{j}\infty}\in T_{\beta_{j}}\Gamma_{j-1}$, we have $\rd_{\widevec{\beta_j\infty}}\Crucial_{P^j}>0$ 
by \eqref{eq:same} (and the piecewise affineness of $\Crucial_\phi$ on $(\sH^1,\rho)$),
and in turn have
\begin{gather*}
 \nu_{P^j,\Gamma_{j-1}}\bigl(U(\widevec{\beta_j\infty})\bigr)=\frac{1}{2}-\rd_{\widevec{\beta_j\infty}}\Crucial_{P^j}<\frac{1}{2}
\end{gather*}
also by the slope formula \eqref{eq:slope}. 
Hence also by Lemma \ref{lem:measure-computation}
applied to $P^j$ at $\beta_j$, we have 
\begin{gather*}
 \frac{1}{2}
 <1-\nu_{P^j,\Gamma_{j-1}}\bigl(U(\widevec{\beta_j\infty})\bigr)
 =
 \nu_{P^j,\Gamma_{j-1}}(\{\cdot\prec\beta_j\})
 =\frac{\deg_{\beta_j}(P^j)-1}{d^j-1},
\end{gather*}
so that $\deg_{\beta_j}(P^j)>(d^j+1)/2$, and in turn
noting that $\deg_{P^j(\beta_j)}(P)=d$ 
(by Lemma \ref{th:basetoinfty}(\ref{item:totram})), 
\begin{gather*}
 \deg_{\beta_j}\bigl(P^{j+1}\bigr)=\deg_{\beta_j}\bigl(P^j\bigr)\cdot\deg_{P^j(\beta_j)}(P)
 >\frac{d^j+1}{2}\cdot d. 
\end{gather*}

Consequently, by the slope formula \eqref{eq:slope}
and Lemma \ref{lem:measure-computation}
applied to $P^{j+1}$ at $\beta_j$, we have
 \begin{multline*}
 \rd_{\widevec{\beta_j\infty}}\Crucial_{P^{j+1}}
 =\frac{1}{2}-\nu_{P^{j+1},\Gamma_{j-1}}\bigl(U(\widevec{\beta_j\infty})\bigr)=-\frac{1}{2}+
 \nu_{P^{j+1},\Gamma_{j-1}}(\{\cdot\prec\beta_j\})\\
 =-\frac{1}{2}+\frac{\deg_{\beta_j}(P^{j+1})-1}{d^{j+1}-1}
 >\Bigl(\frac{1}{2}-1\Bigr)+\frac{(d^j+1)d-2}{2(d^{j+1}-1)}
 =\frac{1}{2}-\frac{d^{j+1}-d}{2(d^{j+1}-1)}>0, 
 \end{multline*}
so that $\beta_j\not\in[\alpha_{j+1},\beta_{j+1})$ by \eqref{eq:same},
which together with $\alpha_{j+1}\prec\beta_{j+1}$ and
$\alpha_{j+1}=\alpha_j(=\alpha_{d-1})\prec\beta_j$ yields 
$\beta_{j+1}\prec\beta_j$. Now we are done by induction.
\end{proof}

\begin{remark}
 For completeness, let us also see that
 {\em if either {\rm (a)} $j\ge d-1$ or {\rm (b)} $d$ is even and $j\ge 1$, then 
 for every $n\ge 0$, 
 $\BC(\nu_{P^j,\Gamma_n})=r_{\sP^1,\Gamma_n}(\MinResLoc_{P^{d-1}})$,
 any side in which is a singleton}; indeed, pick $j\ge 1$. 
 If $d$ is even, then so is the degree $d^j$ 
 of $P^j$, and then $\MinResLoc_{P^j}$ is a singleton 
 by \eqref{eq:same} and \eqref{eq:sloperange}.
 Hence  if either the above conditions (a) or (b) is the case,
 then also by Theorem \ref{thm:minres}(\ref{item:mrlstat}), 
 $\MinResLoc_{P^j}$ is a singleton $\{\xi_{\min}^{(j)}\}$, and for every $n\ge 0$,
 by \eqref{eq:samefinite} and the convexity \eqref{eq:convex},
 $\BC_{\Gamma_n}(\nu_{P^j,\Gamma_n})$
 is the singleton $\{r_{\sP^1,\Gamma_n}(\xi_{\min}^{(j)})\}$.\qed
\end{remark}

\subsection{Examples of the minimal resultant loci}\label{sec:bestpossible}
When $d=2$, the identity
$\BC(\nu_{P^j,\Gamma_n})=\MinResLoc_{P^j}\equiv\{\xi_P\}$
for $j\ge 1=d-1$ and $n\ge 0=d-2$ holds by 
Theorem \ref{thm:minres}(\ref{item:mrlstat})
(for the second identity, see the beginning of the
proof of Lemma \ref{lem:deg}).
In the following,
we also give some concrete $P\in K[z]$ of degree $d\ge 3$, which shows that
\eqref{eq:stationaryminres} is best possible.

Pick an integer $d\ge 3$ and a prime number $p>d$, and let
$|\cdot|=|\cdot|_p$ denote the $p$-adic absolute value on $K=\bC_p$. 
Set
\begin{gather*}
 P(z)=(d-1)pz^d-dz^{d-1},
\end{gather*}
so that $P'(z)=d(d-1)pz^{d-2}\bigl(z-1/p\bigr)$
and
$\crit_P=\{0,1/p,\infty\}$. 
Then $P$ is tame by \cite[Corollary 6.6]{Faber13I}.
Noting that
\begin{gather*}
|P(z)|=|p||z|^{d-1}\text{ on }|z|>|1/p|(>1),\quad|P(1/p)|=|1/p|^{d-1},\\
|P(z)|=|z|^{d-1}\text{ on }|z|\le 1,\quad\text{and}\quad|P(z)|=|1/p|\text{ on }|z|=|1/p|^{1/(d-1)},
\end{gather*}
we have
\begin{itemize}
 \item $\sJ_P\subset\{\cdot\prec\xi_{B(0,|1/p|)}\}$,
and indeed $\xi_P=\xi_{B(0,|1/p|)}$ (also by a Riemann-Hurwitz-type formula (Fact \ref{th:RH}) and Lemma \ref{th:basetoinfty}(\ref{item:totram})),
 \item $1/p\in\sfO_P(\infty)$, $\xi_P\in\Gamma_\infty\setminus\partial\Gamma_\infty$
(since $0=P(0)\precneqq\xi_{B(0,|1/p|)}=\xi_P\precneqq\xi_{B(0,|1/p|^{d-1})}=P(\xi_P)$),
and $r_{\sP^1,\Gamma_\infty}(1/p)=\xi_P$,
\yo{and $P$ restricts to a self-homeomorphism of $[0,\infty]$ preserving $\prec$
and increasing $\rho$ (by Fact \ref{th:polynomial}(\ref{head:monotone})),}
 \item 
$\sfO_P(0)=U(\widevec{\xi_g0})$ and $P(\xi_g)=\xi_g\in\sJ_P$, where
$\sfO_P(0)$ denotes the (Berkovich) immediate basin of $P$
associated to the superattractive fixed point $0$, 
\yo{and $\xi\precneqq P(\xi)$ for every $\xi\in(\xi_g,\infty)$, and}
 \item $\xi_\star:=\xi_{B(0,|1/p|^{1/(d-1)})}\in(\xi_g,\xi_P)\cap\sL_1$,
\end{itemize}
and note that
for every $\xi\in(0,\xi_P)$,
\begin{gather*}
 \deg_0(P^j)=m_P(\widevec{0\infty})=m_P(\widevec{\xi0})
=m_P(\widevec{\xi\infty})=\deg_\xi(P)\equiv d-1 
\end{gather*}
(by Facts \ref{th:polynomial}(\ref{head:locdegdirec})
and \ref{th:inparticular}(\ref{item:nonmaximal})).

When $d=3$, using the above observation, let us see
\begin{gather*}
\MinResLoc_{P^j}=
\begin{cases}
 [\xi_g,\xi_P] &\text{for }j=1,\\
 \{\xi_\star\}(\subset(\xi_g,\xi_P)) &\text{for any }j\ge 2(=d-1); 
\end{cases}
\end{gather*}
indeed, when $j=1$, for every $\xi\in(\yo{\xi_g},\xi_P)$, we compute
\begin{gather*}
 s_P(\widevec{\xi\infty})
=d-m_P(\widevec{\xi\infty})
=d-(d-1)=1\Bigl(<\frac{d}{2}\quad\text{if }d\ge 3\Bigr)\quad
\text{and}\\
m_P(\vec{v})+s_P(\vec{v})
\le(d-1)+0=d-1\Bigl(\le\frac{d+1}{2}\quad\text{if }d=3\Bigr)
\quad\text{for any }\vec{v}\in T_\xi\sP^1\setminus\{\widevec{\xi\infty}\} 
\end{gather*}
also by Facts \ref{th:polynomial}(\ref{head:direcsurp}) and \ref{th:inparticular}(\ref{item:nonmaximal}). Hence 
\yo{by a computation
similar to that in the proof of Claim \ref{th:leftendequal} in Subsection \ref{sec:barystat},
we have both $\rd_{\widevec{\xi\infty}}\Crucial_P=0$ and 
$\rd_{\vec{v}}\Crucial_P\ge 0$ for any 
$\vec{v}\in T_\xi\sP^1\setminus\{\widevec{\xi\infty}\}$.
Then} $(\xi_g,\xi_P)\subset\MinResLoc_P$ \yo{by the convexity \eqref{eq:convex}},
and in turn $[\xi_g,\xi_P]\yo{=}\MinResLoc_P$
by Theorem \ref{thm:minres}(\ref{item:bcmrl}) and
$\xi_g\in\sJ_P=\partial\Gamma_\infty\setminus\{\infty\}$
in this case. Similarly,
when $j\ge 2$, we compute not only
\begin{multline*}
s_{P^j}\bigl(\widevec{\xi_\star\infty}\bigr)
=d^j-m_{P^j}\bigl(\widevec{\xi_\star\infty}\bigr)
=d^j-m_P\bigl(\widevec{\xi_\star\infty}\bigr)\cdot m_{P^{j-1}}\bigl(\widevec{\xi_P\infty}\bigr)\\
=d^j-(d-1)d^{j-1}=d^{j-1}<\frac{d^j}{2} 
\end{multline*}
(by Fact \ref{th:polynomial}(\ref{head:direcsurp})
and Lemma \ref{th:basetoinfty}(\ref{item:directional})) but also 
\begin{gather*}
 m_{P^j}(\vec{v})+s_{P^j}(\vec{v})
=m_{P^j}(\vec{v})
=m_P(\vec{v})\cdot m_P(P_*\vec{v})
\cdot m_{P^{j-2}}\bigl((P^2)_*\vec{v}\bigr)\le
(d-1)^2d^{j-2}\le\frac{d^j+1}{2}
\end{gather*}
for any $\vec{v}\in(T_{\xi_\star}\sP^1)\setminus\{\widevec{\xi_\star\infty}\}$ (by Facts \ref{th:polynomial}(\ref{head:direcsurp}) and \ref{th:inparticular}(\ref{item:nonmaximal})). \yo{By a computation
similar to that in the proof of Claim \ref{th:leftendequal} in Subsection \ref{sec:barystat}, we have both 
\begin{gather*}
 \rd_{\widevec{\xi_\star\infty}}\Crucial_P
=\frac{1}{2}\Bigl(1-\frac{2d^{j-1}}{d^j-1}\Bigr)
\ge\frac{1}{2}\Bigl(1-\frac{d^j-2}{d^j-1}\Bigr)>0
\end{gather*}
and $\rd_{\vec{v}}\Crucial_P\ge 0$ for any 
$\vec{v}\in T_\xi\sP^1\setminus\{\widevec{\xi\infty}\}$.
Then} $\xi_\star\in\MinResLoc_{P^j}$ \yo{by the convexity \eqref{eq:convex}},
and in turn $\{\xi_\star\}=\MinResLoc_{P^j}$
by Theorem \ref{thm:minres}(\ref{item:mrlstat}) in this case.

An argument similar to that above also yields the equality
\begin{gather*}
\MinResLoc_{P^j}
=
\begin{cases}
 \{\xi_g\} \text{\ \ if }d\in\{4,5\}\text{ and }1\le j\le d-2\text{ or if }d\in\{6, 7, 8\}\text{ and }1\le j\le d-3,\\ 
 \{\xi_\star\} \text{\ \ if }d\in\{4,5\}\text{ and }j\ge d-1\text{ or if }d\in\{6, 7, 8\}\text{ and }j\ge d-2.
\end{cases}
\end{gather*}

\section{Proof of Theorem \ref{thm:equi}}\label{sec:pf-equi}
Let $P\in K[z]$ be a polynomial of degree $d>1$;
the $\sL_n$, $\Gamma_n$, 
and $\Gamma_\infty$ are defined using $P$.
For every $n\ge 0$, set
\begin{gather}
\nu_n:=\sum_{\xi\in \sL_n}\delta_\xi\quad\text{on }\sP^1
\quad\text{and}\quad
N_n:=\nu_n(\sP^1)=\#\sL_n\in\{1,\ldots,d^n\},\label{eq:nomult}
\end{gather}
where the sum in the definition of $\nu_n$ 
does {\em not} take into account the local degrees 
of $P^n$ at each $\xi\in\sL_n=P^{-n}(\xi_P)$.
As seen in Theorem \ref{th:convergence} in Appendix, 
the probability measure $\nu_n/N_n$ tends to $\mu_P$ 
weakly on $\sP^1$ as $n\to\infty$, that will be used
in the proof of Theorem \ref{thm:equi}. 

\begin{proof}[Proof of Theorem \ref{thm:equi} in the simple case]
 When $P$ is simple, Theorem \ref{thm:equi} holds since
 for every $j\ge 1$ and every $n\ge 0$,
 we have already seen that 
 $\overline{\nu^+_{P^j,\Gamma_n}}=\overline{|\nu_{P^j,\Gamma_n}|}
 =\nu_{P^j,\Gamma_n}\equiv\delta_{\xi_P}$ on $\sP^1$
 in Section \ref{sec:support}, and
 for every $j\ge 1$ and every $n\ge 0$ 
 $\mu_{P^j}\equiv\delta_{\xi_P}$ on $\sP^1$
 for every $j\ge 1$ (see Subsection \ref{sec:simpleness}).
\end{proof}

In the rest of this section,
suppose that $P$ is non-simple and tame.
Then for every $j\ge 1$, $P^j$ is also tame,
so that $\sR_{P^j}=\Span(\crit_{P^j})$,
and then for every $n\ge 0$,
\begin{gather}
 \#\bigl((\partial\Gamma_n\setminus\{\infty\})\cap\sR_{P^j}\bigr)
\le\#\bigl(\crit_{P^j}\setminus\{\infty\}\bigr)\le d^j-1. \label{eq:bounddegree}
\end{gather}
For any $n\ge j\ge 1$, we note that
$\nu_{P^j,\Gamma_n}(\Gamma_n)=1$, that
$\supp(\nu_{P^j,\Gamma_n})=V(\Gamma_n)\setminus(\sZ_{P^j,\Gamma_n}\sqcup\{\infty\})$,
and that 
$\emptyset\neq\supp(\nu^-_{P^j,\Gamma_n})\subset\partial\Gamma_n\setminus\{\infty\}$
(see Theorem \ref{thm:support}(\ref{head:weight}) and
the definition \eqref{eq:Z} of $\sZ_{P^j,\Gamma_n}$).

Pick $j\ge 1$. We begin by making some reductions in the proof of 
Theorem \ref{thm:equi} under the non-simple and tame 
assumption. 
By Theorem \ref{thm:support}(\ref{head:weight})(\ref{head:zero}), 
we have the comparison
\begin{gather}
0\le\nu_n-(d^j-1)\cdot\nu^-_{P^j,\Gamma_n}
\le\#\bigl(\sZ_{P^j,\Gamma_n}\sqcup((\partial\Gamma_n
\setminus\{\infty\})\cap\sR_{P^j})\bigr)=O(1)
\quad\text{on }\Gamma_n\label{eq:comparison}
\end{gather}
between the positive measures as $n\to\infty$, so that
\begin{gather*}
\nu^-_{P^j,\Gamma_n}(\Gamma_n)=\frac{N_n}{d^j-1}+O(1)
\end{gather*}
as $n\to\infty$, and in turn have both
\begin{align}
\notag\nu^+_{P^j,\Gamma_n}(\Gamma_n)=&
1+\nu^-_{P^j,\Gamma_n}(\Gamma_n)
=\frac{N_n}{d^j-1}+O(1)\quad\text{and}\\
|\nu_{P^j,\Gamma_n}|(\Gamma_n)=&
\bigl(\nu^-_{P^j,\Gamma_n}+\nu^+_{P^j,\Gamma_n}\bigr)(\Gamma_n)
=2\cdot\frac{N_n}{d^j-1}+O(1)
\quad\text{as }n\to\infty. 
\label{eq:total-num}
\end{align}
In particular,
for every continuous function $f$ on $\sP^1$, we have
\begin{gather*}
\Bigl|\int_{\sP^1}f\ \Bigl(\overline{\nu^-_{P^j,\Gamma_n}}-\frac{\nu_n}{N_n}\Bigr)\Bigr|\le\bigl(\sup_{\sP^1}|f|\bigr)\cdot O(N_n^{-1})\quad\text{and}\\
\Bigl|\int_{\sP^1}f\ \Bigl(2\cdot\overline{|\nu_{P^j,\Gamma_n}|}
-\Bigl(\overline{\nu^+_{P^j,\Gamma_n}}+\overline{\nu^-_{P^j,\Gamma_n}}\Bigr)\Bigr)\Bigr|
\le \bigl(\sup_{\sP^1}|f|\bigr)\cdot O(N_n^{-1})\quad\text{as }n\to\infty, 
\end{gather*} 
so that also by the second convergence assertion in Theorem \ref{th:convergence}, one desired convergence
\begin{gather}
 \lim_{n\to\infty}\overline{|\nu_{P^j,\Gamma_n}|}=\mu_{P}\quad\text{weakly on }\sP^1\label{eq:reduced} 
\end{gather}
will conclude the others.

Let us see \eqref{eq:reduced}.
Fix a $C^1$-test function $f$, so that $f=(r_{\sP^1,\Gamma})^*(f|\Gamma)$ on $\sP^1$
for some finite subtree $\Gamma\subset\sP^1$.
 
 {\bfseries Case 1.} Suppose $\Gamma\cap\Gamma_\infty=\emptyset$.
Then by Fact \ref{th:inparticular}(\ref{item:exhaust}),
there is a point $\xi_\Gamma\in\Gamma$ such that
for every $n\ge 0$, 
$r_{\sP^1,\Gamma}(\Gamma_n)=r_{\sP^1,\Gamma}(\Gamma_\infty)=\{\xi_\Gamma\}$.
Hence for every $n\ge 0$,
 \begin{gather*}
 \int_{\sP^1}f\ \overline{|\nu_{P^j,\Gamma_n}|}
 =\int_{\{\xi_{\Gamma}\}}f\delta_{\xi_{\Gamma}}=f(\xi_{\Gamma}),
 \end{gather*}
and similarly $\int_{\sP^1}f\mu_P=f(\xi_{\Gamma})$
(also noting that $\supp\mu_P=:\sJ_P=(\partial\Gamma_\infty)\setminus\{\infty\}$),
so we are done in this case.

{\bfseries Case 2.} Suppose
both $\Gamma\cap\Gamma_\infty\neq\emptyset$ and $\Gamma\subset\sfO_P(\infty)$. 
 Then by Fact \ref{th:inparticular}(\ref{item:exhaust}), there is 
 $s\gg 1$ so large that for every $n\ge s$, 
 \begin{gather*}
 r_{\sP^1,\Gamma}(\Gamma_\infty)=r_{\sP^1,\Gamma}(\Gamma_n)=\Gamma\cap\Gamma_n=\Gamma\cap\Gamma_\infty\subset\Gamma_s\setminus
(\partial\Gamma_s\setminus\{\infty\}). 
 \end{gather*}

\begin{claim}\label{claim:test}
Both
\begin{multline*}
 \int_{\sP^1}f\overline{|\nu_{P^j,\Gamma_n}|}
 = \sum_{\eta\in\Gamma\cap\Gamma_s}
 f(\eta)\biggl(\sum_{\xi\in\partial\Gamma_s\setminus\{\infty\}\text{ such that }
 r_{\Gamma_s,\Gamma\cap\Gamma_s}(\xi)=\eta}
 \frac{\nu_n}{N_n}(\{\cdot\prec\xi\})\biggr)
+O(N_n^{-1})\\
\quad\text{as }n\to\infty
\end{multline*}
and
\begin{gather*}
 \int_{\sP^1}f\mu_P
 =\sum_{\eta\in\Gamma\cap\Gamma_s}
 f(\eta)\biggl(\sum_{\xi\in\partial\Gamma_s\setminus\{\infty\}\text{ such that }r_{\Gamma_s,\Gamma\cap\Gamma_s}(\xi)=\eta}\mu_P(\{\cdot\prec\xi\})\biggr)
\end{gather*}
hold.
\end{claim}
\begin{proof}[Proof of Claim]
 For every $n\ge s$, by Theorem \ref{thm:support}(\ref{head:weight})
and Fact \ref{th:inparticular}({\ref{item:endext}}),
 we have 
 \begin{gather*}
 \bigl((r_{\sP^1,\Gamma_s})_*\overline{|\nu_{P^j,\Gamma_n}|}\bigr)(\{\xi\})
 =\begin{cases}
  |\nu_{P^j,\Gamma_{s}}|(\{\xi\})/|\nu_{P^j,\Gamma_n}|(\Gamma_n) & \text{for every }\xi\in V(\Gamma_s)\setminus\partial\Gamma_s,\\
  \overline{|\nu_{P^j,\Gamma_n}|}(\{\cdot\prec\xi\}) & \text{for every }\xi\in\partial\Gamma_s\setminus\{\infty\},\\
0 & \text{for every }\xi\in\Gamma_s\setminus(V(\Gamma_s)\setminus\{\infty\}), 
 \end{cases}
 \end{gather*}
 and then by $f=(r_{\sP^1,\Gamma})^*(f|\Gamma)$,
 $r_{\sP^1,\Gamma\cap\Gamma_s}
 =r_{\Gamma_s,\Gamma\cap\Gamma_s}\circ r_{\sP^1,\Gamma_s}$,
 and the asymptotic \eqref{eq:total-num},
 we compute
\begin{align*}
\notag\int_{\sP^1}f\overline{|\nu_{P^j,\Gamma_n}|}
\Biggl(&=\int_{\Gamma\cap\Gamma_n}f\bigl((r_{\sP^1,\Gamma})_*\overline{|\nu_{P^j,\Gamma_n}|}\bigr)
=\int_{\Gamma\cap\Gamma_s}f\bigl((r_{\sP^1,\Gamma\cap\Gamma_s})_*\overline{|\nu_{P^j,\Gamma_n}|}\bigr)\\
\notag&=\int_{\Gamma_s}\bigl((r_{\Gamma_s,\Gamma\cap\Gamma_s})^*(f|(\Gamma\cap\Gamma_s))\bigr)\bigl((r_{\sP^1,\Gamma_s})_*\overline{|\nu_{P^j,\Gamma_n}|}\bigr)=\Biggr)\\
\notag&=\sum_{\eta\in\Gamma\cap\Gamma_s}f(\eta)\biggl(
\sum_{\xi\in V(\Gamma_s)\setminus\partial\Gamma_s\text{ such that }r_{\Gamma_s,\Gamma\cap\Gamma_s}(\xi)=\eta}
|\nu_{P^j,\Gamma_{s}}|(\{\xi\})\biggr)
/|\nu_{P^j,\Gamma_n}|(\Gamma_n)
+\label{eq:boundary}\\
\notag &+\sum_{\eta\in\Gamma\cap\Gamma_s}f(\eta)\biggl(\sum_{\xi\in\partial\Gamma_s\setminus\{\infty\}\text{ such that }r_{\Gamma_s,\Gamma\cap\Gamma_s}(\xi)=\eta}
\overline{|\nu_{P^j,\Gamma_n}|}(\{\cdot\prec\xi\})\biggr)\\
&=O(N_n^{-1})+\sum_{\eta\in\Gamma\cap\Gamma_s}f(\eta)\biggl(\sum_{\xi\in\partial\Gamma_s\setminus\{\infty\}\text{ such that }r_{\Gamma_s,\Gamma\cap\Gamma_s}(\xi)=\eta}
\overline{|\nu_{P^j,\Gamma_n}|}(\{\cdot\prec\xi\})\biggr)
\end{align*}
as $n\to\infty$.

For every $\xi\in\partial\Gamma_s\setminus\{\infty\}$
and every $n\ge s$, by Proposition \ref{th:measure-sum}
(and Fact \ref{th:inparticular}(\ref{item:endext})),
we have
$\nu_{P^j,\Gamma_n}(\{\cdot\prec\xi\})
=1-\nu_{P^j,\Gamma_n}(U(\widevec{\xi\infty}))
=1-\nu_{P^j,\Gamma_s}(U(\widevec{\xi\infty}))
=\nu_{P^j,\Gamma_s}(\{\cdot\prec\xi\})
=\nu_{P^j,\Gamma_s}(\{\xi\})$,
so that 
 \begin{gather*}
 |\nu_{P^j,\Gamma_n}|(\{\cdot\prec\xi\})
 =\nu_{P^j,\Gamma_s}(\{\xi\})+2\cdot\nu^-_{P^j,\Gamma_n}(\{\cdot\prec\xi\})
=2\cdot\nu^-_{P^j,\Gamma_n}(\{\cdot\prec\xi\})
+O(1)\quad\text{as }n\to\infty.
\end{gather*}
Hence for every $\xi\in\partial\Gamma_s\setminus\{\infty\}$, 
using the comparison \eqref{eq:comparison}
and the asymptotic \eqref{eq:total-num},
we have
\begin{gather*}
 \overline{|\nu_{P^j,\Gamma_n}|}(\{\cdot\prec\xi\})
 =\frac{2\cdot\nu_n(\{\cdot\prec\xi\})/(d^j-1)+O(1)}{2\cdot N_n/(d^j-1)+O(1)}
 =\frac{\nu_n}{N_n}(\{\cdot\prec\xi\})+O(N_n^{-1})
\quad\text{as }n\to\infty, 
\end{gather*}
which completes the proof of the first desired asymptotic equality.

Since $\supp\mu_P=\sJ_P=\partial\Gamma_\infty\setminus\{\infty\}$,
the second desired equality also holds 
by a computation similar to (and simpler than) that in the above.
\end{proof}

Let us conclude Case 2.
Pick $\eta\in\Gamma\cap\Gamma_s(\subset\Gamma_s\setminus\partial\Gamma_s)$ and $\xi\in\partial\Gamma_s\setminus\{\infty\}$ satisfying
$r_{\Gamma_s,\Gamma\cap\Gamma_s}(\xi)=\eta$. 
Then there is $\xi'\in\Gamma_s$
so close to $\xi$ that 
$\xi\precneqq\xi'\prec\eta$ and $(\xi,\xi')\cap V(\Gamma_s)=\emptyset$.
 By Fact \ref{th:inparticular}({\ref{item:endext}}),
 for every $n\ge s$,
 we have $\nu_n(U(\widevec{\xi\infty})\cap U(\widevec{\xi'\xi}))=0$,
 and since $\supp\mu_P=\sJ_P=\partial\Gamma_\infty\setminus\{\infty\}$,
 we also have $\mu_P(U(\widevec{\xi\infty})\cap U(\widevec{\xi'\xi}))=0$ by Fact \ref{th:inparticular}(\ref{item:exhaust}). 
Hence,
since the characteristic function $\chi_{\{\cdot\prec\xi\}}$ on $\sP^1$ extends to a $C^1$-function on $\sP^1$ which is $\equiv 0$ on $\sP^1\setminus U(\widevec{\xi'\xi})$, 
the second convergence assertion in Theorem \ref{th:convergence} yields
\begin{gather*}
 \frac{\nu_n}{N_n}(\{\cdot\prec\xi\})\to\mu_P(\{\cdot\prec\xi\})\quad\text{as }n\to\infty
\end{gather*}
 (with an $O(n^2d^{-n})$-error estimate from 
 the proof of Theorem \ref{th:convergence}), 
 and we are done in this case by the above Claim.

 {\bfseries Case 3.} 
 Suppose $\Gamma\cap\Gamma_\infty\neq\emptyset$ 
 and $\Gamma\not\subset\sfO_P(\infty)$. Set
 $\Gamma':=\Gamma\cap(\sfO_P(\infty)\sqcup\sJ_P)$,
 which is a subtree in $\Gamma$.
 For every $n\ge 0$, $r_{\sP^1,\Gamma}(\Gamma_n)\subset 
 r_{\sP^1,\Gamma}(\Gamma_\infty)\subset\Gamma'$
 (using Fact \ref{th:inparticular}(\ref{item:exhaust})), and then
 recalling $f=(r_{\sP^1,\Gamma})^*(f|\Gamma)$ and
 $r_{\sP^1,\Gamma'}=r_{\Gamma,\Gamma'}\circ r_{\sP^1,\Gamma}$, we have
 \begin{multline*}
 \int_{\sP^1} f\overline{|\nu_{P^j,\Gamma_n}|}
 =\int_{\Gamma}f\bigl((r_{\sP^1,\Gamma})_*\overline{|\nu_{P^j,\Gamma_n}|}\bigr)\\
 =\int_{\Gamma}\bigl((r_{\Gamma,\Gamma'})^*(f|\Gamma')\bigr)\bigl((r_{\sP^1,\Gamma})_*\overline{|\nu_{P^j,\Gamma_n}|}\bigr)
 =\int_{\Gamma'}f\bigl((r_{\sP^1,\Gamma'})_*\overline{|\nu_{P^j,\Gamma_n}|}\bigr). 
 \end{multline*}
We similarly have
$\int_{\sP^1}f\mu_P=\int_{\Gamma'}f((r_{\sP^1,\Gamma'})_*\mu_P)$
since $\supp\mu_P=:\sJ_P=(\partial\Gamma_\infty)\setminus\{\infty\}$.

Pick $\epsilon>0$, and
also a $C^1$-function $g$ on $\sP^1$ such that
$\sup_{\Gamma'}|f-g\circ r_{\Gamma',\Gamma''}|<\epsilon$
for some subtree $\Gamma''$ in $\Gamma'$ which is
contained in $\sfO_P(\infty)$ (see \cite[Proof of Proposition 5.4]{Baker10}). Then
\begin{gather*}
\sup_{n\ge 0}\Bigl|\int_{\Gamma'}(f-g\circ r_{\Gamma',\Gamma''})
\bigl((r_{\sP^1,\Gamma'})_*\overline{|\nu_{P^j,\Gamma_n}|}\bigr)\Bigr|<\epsilon\quad\text{and}\\
 \Bigl|\int_{\Gamma'}(f-g\circ r_{\Gamma',\Gamma''})\ (r_{\sP^1,\Gamma'})_*\mu_P\Bigr|<\epsilon.
\end{gather*}
Finally, set 
$G:=(r_{\sP^1,\Gamma''})^*(g|\Gamma'')$, 
so that $G=(r_{\sP^1,\Gamma''})^*(G|\Gamma'')$ 
on $\sP^1$, and recall that $\Gamma''\subset\sfO_P(\infty)$. 
 Since $r_{\sP^1,\Gamma''}=r_{\Gamma',\Gamma''}\circ r_{\sP^1,\Gamma'}$, 
 for every $n\ge 0$, we have 
 \begin{gather*}
 \int_{\Gamma'}(g\circ r_{\Gamma',\Gamma''})\bigl((r_{\sP^1,\Gamma'})_*\overline{|\nu_{P^j,\Gamma_n}|}\bigr)
 =\int_{\Gamma''}g\bigl((r_{\sP^1,\Gamma''})_*\overline{|\nu_{P^j,\Gamma_n}|}\bigr)
 =\int_{\sP^1}G\overline{|\nu_{P^j,\Gamma_n}|}. 
 \end{gather*}
We similarly have
$\int_{\Gamma'}(g\circ r_{\Gamma',\Gamma''})\bigl((r_{\sP^1,\Gamma'})_*\mu_P\bigr)
=\int_{\sP^1}G\mu_P$.
Now we are done in this final case since
$\lim_{n\to\infty}\int_{\sP^1}G\overline{|\nu_{P^j,\Gamma_n}|}=\int_{\sP^1}G\mu_P$
by the previous Case 2. \qed

\section*{Appendix: equidistribution of iterated pullbacks without multiplicities}\label{sec:ap}

Let $P\in K[z]$ be a polynomial of degree $d>1$;
the $\sL_n$, $\Gamma_n$, and $\Gamma_\infty$ are defined using $P$.
Recall that for every $n\ge 0$, 
\begin{gather*}
\sL_n:=P^{-n}(\xi_P),\quad
\nu_n:=\sum_{\xi\in \sL_n}\delta_\xi\text{ on }\sP^1,
\quad\text{and}\quad
N_n:=\nu_n(\sP^1)=\#\sL_n\in\{1,\ldots,d^n\}.
\end{gather*}

The following 
equidistribution of iterated pullbacks under $P^n$
of the base point $\xi_P$
{\em without} multiplicities
plays a key role in the proof of Theorem \ref{thm:equi},
and might be of independent interest.
 
\begin{apthm}\label{th:convergence}
If $P$ is non-simple and tame, then
the limit $\lim_{n\to\infty}N_n/d^n>0$ exists,
and moreover
$\lim_{n\to\infty}\nu_n/N_n=\mu_P$ weakly on $\sP^1$.
If $P$ is simple, then for every $n\ge 1$,
$\nu_n\equiv\mu_P$ on $\sP^1$ and $N_n\equiv 1$.
\end{apthm}

\begin{proof}[Proof of Theorem \ref{th:convergence} in the simple case]
 When $P$ is simple, $P^{-1}(\xi_P)=\{\xi_P\}$,
so that $\mu_P=\delta_{\xi_P}$ on $\sP^1$ 
(see Subsection \ref{sec:simpleness}). Moreover,
for every $n\ge 0$, we have already seen
$\nu_n\equiv\delta_{\xi_P}$ and $N_n\equiv 1$
in Section \ref{sec:support}.
\end{proof}

From now on, suppose that $P$ is non-simple and tame. 
Recall also 
\begin{gather*}
 \sC_P:=r_{\sP^1,\Gamma_\infty}
\bigl(\crit_P\cap(\sfO_P(\infty)\setminus\{\infty\})\bigr)
\end{gather*}
in \eqref{eq:critretract},
$\sR_P=\Span(\crit_P)$,
and $\#(\crit_P\setminus\{\infty\})\le d-1$
under the assumption.
By Fact \ref{th:inparticular}(\ref{item:exhaust}), we have 
$r_{\sP^1,\Gamma_\infty}(\crit_P\setminus\sfO_P(\infty))\subset
\sJ_P=\partial\Gamma_\infty\setminus\{\infty\}$,
and set
\begin{gather*}
\kappa:=\#\bigl(r_{\sP^1,\Gamma_\infty}(\crit_P\setminus\sfO_P(\infty))\bigr)
\in\{0,1,\ldots,d-1\}.
\end{gather*}
By Fact \ref{th:inparticular}(\ref{item:exhaust}) and Lemma \ref{lem:tree-vertices}, 
there is $n_0\gg 1$ so large that
\begin{gather*}
\sC_P\subset\Span(\sL_{n_0})\setminus\sL_{n_0}
=\{\cdot\prec\xi_P\}\cap(\Gamma_{n_0}\setminus\partial\Gamma_{n_0})
\end{gather*}
and that if $\kappa>0$, then
for every $n\ge n_0$, $r_{\Gamma_\infty,\Gamma_n}$ restricts
to a bijection $r_{\sP^1,\Gamma_\infty}(\crit_P\setminus\sfO_P(\infty))
\overset{\cong}{\to}
r_{\sP^1,\Gamma_n}(\crit_P\setminus\sfO_P(\infty))$.

If $\kappa=0$, then for every $n\ge n_0$, 
by a Riemann-Hurwitz-type formula (Fact \ref{th:RH}),
we have
$\nu_{n+1}=P^*\nu_{n}$ on $\sP^1$, so that
$\nu_n=(P^{n-n_0})^*\nu_{n_0}$ on $\sP^1$ and that 
$N_n=d^{n-n_0}N_{n_0}$.
Hence $N_n/d^{n-n_0}\equiv N_{n_0}>0$ and,
by \eqref{eq:canonical} (or \eqref{eq:canonicalquant}), we have
\begin{gather*}
 \frac{\nu_n}{N_n}=\frac{(P^{n-n_0})^*\nu_{n_0}}{d^{n-n_0}N_{n_0}}\to\frac{N_{n_0}\cdot\mu_P}{N_{n_0}}=\mu_P\quad\text{weakly on }\sP^1
\text{ as }n\to\infty 
\end{gather*}
(with an $O(d^{-n})$-error estimate for $C^1$-test functions), and we are done in the $\kappa=0$ case.
 
Suppose now that $\kappa>0$. 
Then for every $n\ge n_0$, writing
$r_{\sP^1,\Gamma_{\infty}}\bigl(\crit_P\setminus\sfO_P(\infty)\bigr)=\{\xi_i:i\in\{1,\ldots,\kappa\}\}$,
setting
$\xi^{(n)}_i:=r_{\Gamma_{\infty},\Gamma_n}(\xi_i)$
for each $i\in\{1,\ldots,\kappa\}$,
and noting $r_{\sP^1,\Gamma_n}=r_{\Gamma_\infty,\Gamma_n}\circ r_{\sP^1,\Gamma_\infty}$, 
by Fact \ref{th:inparticular}({\ref{item:endext}})
and a Riemann-Hurwitz-type formula (Fact \ref{th:RH}), 
we compute
\begin{gather*}
 \nu_n=P^*\nu_{n-1}-\sum_{\xi\in r_{\sP^1,\Gamma_n}(\crit_P\setminus\sfO_P(\infty))}\bigl(\deg_\xi (P)-1\bigr)\delta_\xi
 =P^*\nu_{n-1}-\sum_{i=1}^{\kappa}\bigl(\deg_{\xi_i}(P)-1\bigr)\delta_{\xi^{(n)}_i},
\end{gather*}
so recursively
\begin{align}
\notag\nu_n&=(P^{n-n_0})^*\nu_{n_0}-\sum_{\ell=0}^{n-n_0-1}
 \sum_{i=1}^{\kappa}\bigl(\deg_{\xi_i}(P)-1\bigr)(P^{\ell})^*\delta_{\xi^{(n-\ell)}_i}\\
 &=(P^{n-n_0})^*\nu_{n_0}-
 \sum_{i=1}^{\kappa}\bigl(\deg_{\xi_i}(P)-1\bigr)\sum_{\ell=0}^{n-n_0-1}(P^{\ell})^*\delta_{\xi^{(n-\ell)}_i}\quad\text{on }\sP^1,\label{eq:compute}
\end{align}
and then
$N_n=d^{n-n_0}N_{n_0}-(\sum_{i=1}^{\kappa}\bigl(\deg_{\xi_i}(P)-1\bigr))
\cdot\sum_{\ell=0}^{n-n_0-1}d^{\ell}$.
Hence setting
\begin{gather*}
 M_0:=\sum_{i=1}^{\kappa}\bigl(\deg_{\xi_i}(P)-1\bigr)
(\ge\kappa>0), 
\end{gather*}
we have $M_0=dN_{n_0}-N_{n_0+1}$ and
\begin{gather}
 0<\frac{N_n}{d^{n-n_0}}
 =N_{n_0}-M_0\frac{1-d^{-(n-n_0)}}{d-1}\searrow N_{n_0}-\frac{M_0}{d-1}
 =\frac{N_{n_0+1}-N_{n_0}}{d-1}\quad\text{as }n\to\infty,\label{eq:limitratio}
\end{gather}
so that $N_{n_0+1}\ge N_{n_0}$, and indeed $N_{n_0+1}>N_{n_0}$
by $\sL_{n_0+1}\subset\{\cdot\precneqq\xi_P\}\cap\Gamma_\infty$
and Fact \ref{th:inparticular}(\ref{item:nonmaximal}). 
Hence the first assertion in Theorem \ref{th:convergence}
holds; we indeed have
\begin{gather*}
 \lim_{n\to\infty}\frac{N_n}{d^n}
=\frac{N_{n_0}-M_0/(d-1)}{d^{n_0}}\in(0,1).
\end{gather*} 

Let us see the second assertion.
For every $n>n_0$, from \eqref{eq:compute},
we also have
\begin{align}
\notag \frac{\nu_n}{N_n}
 =&\frac{(P^{n-n_0})^*\nu_{n_0}}{N_n}
 -\frac{1}{N_n}\cdot\Bigl(\sum_{i=1}^{\kappa}\bigl(\deg_{\xi_i}(P)-1\bigr)\Bigr)\cdot\Bigl(\sum_{\ell=0}^{n-n_0-1}d^{\ell}\Bigr)\mu_P\\
\notag &-\frac{1}{N_n}\sum_{i=1}^{\kappa}\bigl(\deg_{\xi_i}(P)-1\bigr)\sum_{\ell=0}^{n-n_0-1}
 d^{\ell}\cdot\biggl(\frac{(P^{\ell})^*\delta_{\xi^{(n-\ell)}_i}}{d^{\ell}}-\mu_P\biggr)\\
\notag =&\frac{1}{N_n/d^{n-n_0}}
\biggl(\frac{(P^{n-n_0})^*\nu_{n_0}}{d^{n-n_0}}
 -M_0\cdot\frac{1-d^{-(n-n_0)}}{d-1}\mu_P\biggr)\\
 &-\frac{1}{N_n}\sum_{i=1}^{\kappa}\bigl(\deg_{\xi_i}(P)-1\bigr)\sum_{\ell=0}^{n-n_0-1}
 d^{\ell}\cdot\biggl(\frac{(P^{\ell})^*\delta_{\xi^{(n-\ell)}_i}}{d^{\ell}}-\mu_P\biggr)
\label{eq:withoutmult}
\end{align}
on $\sP^1$, and we note that
by \eqref{eq:canonical} (or \eqref{eq:canonicalquant})
and \eqref{eq:limitratio},
\begin{multline*}
\frac{1}{N_n/d^{n-n_0}}
\biggl(\frac{(P^{n-n_0})^*\nu_{n_0}}{d^{n-n_0}}
 -M_0\cdot\frac{1-d^{-(n-n_0)}}{d-1}\mu_P\biggr)\\
 \to 
\frac{1}{N_{n_0}-M_0/(d-1)}\Bigl(N_{n_0}\cdot\mu_P
-\frac{M_0}{d-1}\mu_P\Bigr)
 =\mu_P\quad\text{weakly on }\sP^1\text{ as }n\to\infty
\end{multline*}
(with an $O(d^{-n})$-error estimate for $C^1$-test functions on $\sP^1$).

We claim that \eqref{eq:withoutmult} 
tends to $0$ weakly on $\sP^1$ as $n\to\infty$
(indeed with an $O(n^2d^{-n})$-error estimate for 
$C^1$-test functions on $\sP^1$);
indeed, for every $i\in\{1,\ldots,\kappa\}$ and 
every $\ell\in\{0,\ldots,n-n_0\}$, 
there is a unique decreasing
sequence $(\xi^{(n-\ell)}_i(t))_{t=0}^{n-j}$ in $(\sP^1,\prec)$ such that 
$\xi^{(n-\ell)}_i(t)\in\sL_t$ 
for every $t\in\{0,\ldots,n-\ell\}$,
\begin{gather*}
\xi^{(n-\ell)}_i(0)=\xi_P,\quad\text{and}\quad
\xi^{(n-\ell)}_i(n-\ell)=\xi^{(n-\ell)}_i 
\end{gather*}
(by Fact \ref{th:inparticular}(\ref{item:ends})).
Then for every $i\in\{1,\ldots,\kappa\}$ 
and every $\ell\in\{0,\ldots,n-n_0-1\}$,
since for each $t\in\{0,\ldots,n-\ell-1\}$, 
$P^t$ restricts to a homeomorphism
$[\xi^{(n-\ell)}_i(t+1),\xi^{(n-\ell)}_i(t)]\to[\xi,\xi_P]$ 
for some $\xi\in\sL_1$
without decreasing $\rho$ strictly
(by Fact \ref{th:polynomial}(\ref{head:monotone})),
we have
\begin{gather*}
\rho\bigl(\xi^{(n-\ell)}_i,\xi_P\bigr)
 \le\sum_{t=0}^{n-\ell-1}\rho\bigl(\xi^{(n-\ell)}_i(t+1),\xi^{(n-\ell)}_i(t)\bigr)
 \le(n-\ell)\cdot\max_{\sL_1}\rho(\,\cdot\,,\xi_P).
\end{gather*}
Consequently, 
for every $C^1$-test function $f$ on $\sP^1$, 
also by \eqref{eq:canonicalquant}
and the first assertion in Theorem \ref{th:convergence},
we have
\begin{multline*}
 \biggl|\int_{\sP^1}f(\cdot)\frac{1}{N_n}\sum_{i=1}^{\kappa}\bigl(\deg_{\xi_i}(P)-1\bigr)\sum_{\ell=0}^{n-n_0-1}d^{\ell}\cdot\biggl(\frac{(P^{\ell})^*\delta_{\xi^{(n-\ell)}_i}}{d^{\ell}}-\mu_P\biggr)(\cdot)\biggr|\\
=
\frac{1}{N_n}\cdot M_0\sum_{\ell=0}^{n-n_0-1}d^{\ell}\cdot\max_{i\in\{1,\ldots,\kappa\}}\biggl|\int_{\sP^1}f\biggl(\frac{(P^{\ell})^*\delta_{\xi^{(n-\ell)}_i}}{d^{\ell}}-\mu_P\biggr)\biggr|\\
\le
\frac{M_0}{N_n}\sum_{\ell=0}^{n-n_0-1}
\biggl(|\Delta f|(\sP^1)\cdot\Bigl(\max_{i\in\{1,\ldots,\kappa\}}\rho\bigl(\xi^{(n-\ell)}_i,\xi_g\bigr)+2\sup_{\sP^1}|g_P|\Bigr)\biggr)\\
\le\frac{M_0}{N_n}\cdot|\Delta f|(\sP^1)\cdot C_P
\sum_{\ell=0}^{n-n_0-1}(n-\ell)
 =O\Bigl(\Bigl(\sum_{\ell=0}^{n-n_0-1}(n-\ell)\Bigr)N_n^{-1}\Bigr)
 =O(n^2d^{-n})\quad\text{as }n\to\infty,
\end{multline*}
setting $C_P:=(\max_{\sL_1}\rho(\,\cdot\,,\xi_P))+(\rho(\xi_g,\xi_P)+2\sup_{\sP^1}|g_P|)/(n_0+1)\in\bR_{>0}$. Hence we are also done 
in the $\kappa>0$ case. \qed

\begin{acknowledgements}
The first author was partially supported by ISF Grant 1226/17.
The second author was partially supported by JSPS Grant-in-Aid 
for Scientific Research (C), 19K03541 and 23K03129,
and (B), 19H01798.
The authors thank Jan Kiwi for mentioning Trucco's work.
\end{acknowledgements}

 %\bibliographystyle{plain}
 %\bibliography{references} 

\end{document}